\documentclass[a4paper,11pt]{article}
\usepackage{amsfonts}
\usepackage{amsmath}
\usepackage{amssymb}
\usepackage{amsthm}
\usepackage{enumerate}
\usepackage[mathscr]{eucal}
\usepackage{eqlist}
\usepackage{graphicx}
\usepackage{esint}
\IfFileExists{ulem.sty}{\usepackage{ulem}}{} % No ulem commands are used.
\usepackage{xcolor}
\definecolor{roundrevision}{RGB}{126,52,156}
\IfFileExists{zebra.sty}{\usepackage{zebra}}{} % Optional annotation package.
\usepackage{tikz}
\definecolor{reviewgreen}{RGB}{0,128,64}

\usepackage[colorlinks=true,linkcolor=blue,citecolor=blue]{hyperref}
\usepackage{indentfirst}
\usepackage{mathtools}\mathtoolsset{showonlyrefs}

\usepackage[body={17true cm, 27true cm}]{geometry}

\newtheorem{Theorem}{Theorem}[section]
\newtheorem{Definition}{Definition}
\newtheorem{Lemma}[Theorem]{Lemma}
\newtheorem{Corollary}[Theorem]{Corollary}
\newtheorem{Remark}[Theorem]{Remark}
\newtheorem{Proposition}[Theorem]{Proposition}

\newtheorem*{Conjecture}{Conjecture}
\numberwithin{equation}{section} \allowdisplaybreaks
\IfFileExists{mathdots.sty}{\usepackage{mathdots}}{} % No mathdots commands are used.

\definecolor{criticaldelete}{RGB}{180,85,0}
\begin{document}
	\title{Symmetry and critical points of second Neumann eigenfunctions on
isosceles trapezoids and kites\footnote{\footnotesize H. Deng is supported by NSFC Program (Grant No.12326303, No.12001276).
			C. Gui is supported by NSFC Key Program (Grant No.12531010), University of Macau research grants CPG2024-00016-FST, CPG2025-00032-FST, CPG2026-00027-FST, SRG2023-00011-FST, MYRGGRG2023-00139-FST-UMDF, UMDF Professorial Fellowship of Mathematics, Macao SAR FDCT 0003/2023/RIA1 and Macao SAR FDCT 0024/2023/RIB1.
			X. Yang and J. Zou are supported by NSFC Program (Grant No. 12431018). R. Yao is supported by NSFC Program (Grant No. 12671237) and Guangdong Basic and Applied Basic Research Foundation (Grant No. 2025A1515011856).}}

	\author{\small Haiyun Deng, Changfeng Gui, Xuyong Jiang, Xiaoping Yang, Ruofei Yao, Jun Zou}
	%\date{Received:  / Accepted: }
	
	\date{}
	\maketitle
	\renewcommand{\labelenumi}{[\arabic{enumi}]}
	
	\begin{abstract}  
We classify the symmetry and multiplicity of the second Neumann eigenvalue on isosceles trapezoids and kites as the height varies. Precisely, away from a unique transition height, the second Neumann eigenvalue is simple, and its eigenfunction is either symmetric or antisymmetric. At a transition height, it has multiplicity two, with one symmetric and one antisymmetric basis eigenfunction. We prove that every nonzero real second Neumann eigenfunction has no interior critical points, including arbitrary linear combinations at a transition height, thereby verifying the \textit{hot spots conjecture} for both classes of domains. In the simple-eigenvalue case, we also classify the nonvertex boundary critical points for trapezoids and kites. Reflection reduces the spectral classification to a comparison of a Neumann branch and a mixed Dirichlet--Neumann branch. Energy estimates establish strict transversality at every crossing and hence uniqueness wherever a transition occurs. To treat the two-dimensional eigenspace, we use Jacobian estimates and directional and radial comparison arguments to exclude cancellation of the two basis gradients.
	\end{abstract}
	
	{\bf Keywords:} Hot spots; Neumann eigenfunctions; Isosceles trapezoids; Kites; Symmetry
	
	{{\bf 2020 Mathematics Subject Classification.} Primary: 35B38, 35J05; Secondary: 35J25, 58J50.}
	
	\tableofcontents
	
	\section{Introduction}
	~~~~In this paper, we investigate the following eigenvalue problem with Neumann boundary condition 
	\begin{equation}\label{1.1}\begin{array}{l}
			\left\{
			\begin{array}{l}
				\Delta u+\mu u=0~~\mbox{in}~~\Omega\subset \mathbb{R}^2,\\
				\frac{\partial u}{\partial n}=0 ~~\mbox{on}~~\partial\Omega,
			\end{array}
			\right.
	\end{array}\end{equation}
	where $n$ denotes the unit outward normal vector to $\partial\Omega.$ The corresponding eigenvalues $\mu_j$ are nonnegative and can be arranged as
	$$0=\mu_1<\mu_2\le\mu_3\le\cdots\to\infty.$$

	In 1974, Rauch \cite{Rauch1974} proposed a conjecture at an academic conference: for the second Neumann eigenfunction $u$ of the boundary value problem \eqref{1.1}, the maximum and minimum values of $u$ are attained only on the boundary of the domain. This problem, known as the hot spot conjecture, inaugurated a half-century of exploration in the mathematical community. For general domains in $\mathbb{R}^n$, the conjecture is false. For instance, even for planar domains with interior holes, Neumann eigenfunctions can attain their maxima inside the region. The conjecture has been proven false for certain planar multiply connected domains \cite{Burdzy1999Ann,Burdzy2005Duke} and convex domains in higher dimensions (for sufficiently large dimension) \cite{deDios2024Arxiv}. To be more precise, the hot spots conjecture regarding second Neumann eigenfunctions is now often stated as follows.

	\begin{Conjecture} The second Neumann eigenfunction attains its maximum and minimum only on the boundary
		of the domain $\Omega$ provided that  $\Omega$ is a convex domain in $\mathbb{R}^2$, or a simply connected planar domain.
	\end{Conjecture}
	
	The first positive result was due to Kawohl for cylindrical domains (see \cite[Corollary 2.15]{Kawohl1985} ); he further noted that the conjecture holds for specific domains including parallelepipeds, balls, and annuli (see \cite[page 46]{Kawohl1985}). In 1999, Ba\~{n}uelos and Burdzy achieved the first major progress through probabilistic methods in \cite{Banuelos1999JFA}. For the simple second Neumann eigenvalue, they coupled Brownian motion to the eigenfunction by deforming initial conditions, combined with eigenvalue estimates, to prove that the conjecture holds for various planar convex domains including certain special lip domains, where a ``lip domain" refers to a planar domain situated between the graphs of two Lipschitz functions whose Lipschitz constants are at most one. Subsequently, many important advances have been made in this aspect. In 2000, Jerison and Nadirashvili introduced the method of domain deformation in \cite{Jerison2000JAMS}, proving the hot spots conjecture via a continuity method on planar domains with two axes of symmetry. Their method can be applied to handle more complex situations. For example, in the case of multiple eigenvalues, they proved that the odd eigenfunctions associated with the smallest non-trivial eigenvalue exhibit monotonic properties. Another advance in this field was made by Atar and Burdzy in \cite{Atar2004JAMS}, who proved the simplicity of the second Neumann eigenvalue in lip domains and the hot spots conjecture in such domains. In 2012, an online collaborative project named ``Polymath 7" \cite{Polymath2012} was officially launched with the primary goal of tackling the hot spots conjecture in acute triangles.
	In 2022, building on the momentum of the project, Judge and Mondal \cite{Judge2020Ann,Judge2022Ann} proved that second Neumann eigenfunctions on triangles have no interior critical points, thereby establishing the hot-spots conclusion in the triangular setting. Building on this work, in 2026, Chen, Gui, and Yao \cite{Chen2026Invent} obtained a finer classification: they showed that a second Neumann eigenfunction
	has at most one non-vertex critical point and established directional monotonicity, together with a sharp description of the extremal locations. For the related results, see \cite{Chen2025AMPA,Hatcher2025JST,Mai2026DCDS} and references therein.
	
	Meanwhile, Judge and Mondal \cite{Judge2022CPDE,Judge2025Que} also studied the critical points of eigenfunctions on polygonal domains. In addition, in 2002, Atar and Burdzy \cite{Atar2002ECP} investigated the location of the nodal lines of the second Neumann eigenfunctions on certain triangular domains. In 2020, Nigam, Siudeja and Young \cite{Nigam2020FCM}  provided a computer-assisted proof of a modified version of Schiffer's conjecture on the regular pentagon, establishing the existence of a nontrivial Neumann eigenfunction that remains strictly positive on the boundary. 
	
	In higher-dimensional domains, the conjecture also holds for balls, annular domains, and parallelepipeds \cite{Kawohl1985,Pütter1992}. To the best of our knowledge, the first published result was obtained by Kawohl in \cite{Kawohl1985} for cylinders using the method of monotone rearrangement. However, Kawohl also pointed out in \cite{Kawohl1985} (Chapter 2, Section 5, Remark 2.36) that this method is difficult to apply to general symmetric convex domains. In 2019, Chen, Li, and Wang \cite{Chen2019JMPA} proved the hot spots conjecture for long rotationally symmetric domains in $\mathbb{R}^n$ by employing the continuity method from \cite{Jerison2000JAMS}. They showed that the second Neumann eigenfunction, which is odd in the $x_n$ variable, is a Morse function on the boundary, has exactly two critical points and is monotone in the direction from its minimum point to its maximum point.

	\subsection{Main results}
	
	~~~However, several aspects of the hot spots problem on polygonal domains remain to be understood, particularly the relation between the geometry of a symmetric domain and the symmetry type and critical point structure of its second Neumann eigenfunction. In this paper, we study these questions for two classes of symmetric quadrilaterals: isosceles trapezoids and kites, with non-square rhombuses included as a special case. Our main objective is to determine how the symmetry type of the second Neumann eigenfunction depends on the geometry of the domain and to identify the corresponding distribution of its critical points. In particular, we characterize transitions between symmetric and antisymmetric eigenfunctions through the comparison of the associated spectral branches, determine the corresponding simplicity or multiplicity of the second Neumann eigenvalue, and describe the location of the non-vertex critical points.
	The symbol $Q$ denotes an open polygonal domain, including the triangular member $K_{1,h}$ of the kite family. Throughout this paper, we focus on non‑vertex (i.e., nontrivial) critical points, i.e., set 
	$$\mathcal{C}_{nv}(u)=\{p\in \overline{Q}:\nabla u=0, p~\mbox{is not a vertex of}~Q\}.$$
    
	For convenience, throughout this paper we assume that the longer base $P_1P_2$ of the trapezoid is placed on the bottom, lying on the $x$‑axis, and the shorter base $P_3P_4$ is on top. The first main result is as follows.

	\begin{Theorem}\label{thm1.1}
		Let $Q_{\alpha,h}$ be the isosceles trapezoid $P_1P_2P_3P_4$ with vertices
		\[\quad P_1=(-1,0),\quad P_2=(1,0),\quad P_3=(1-h\cot \alpha,h),\quad P_4=(-1+h\cot \alpha,h).\]
		where the base angle $\alpha\in(0,\pi/2)$ is fixed and $h\in(0,\tan\alpha)$.
		Let $u$ be the second Neumann eigenfunction on $Q_{\alpha,h}$ and $\mu_2$ be the corresponding eigenvalue. Then
		\begin{enumerate}[(1)]
			\item  If  $\alpha\le \frac{\pi}{3}$, then $\mu_2$ is simple and $u$ is
			antisymmetric about the $y$-axis.
			\item If  $\alpha> \frac{\pi}{3}$, then there exists a unique critical height $\hat{h}(\alpha)$,
			such that:
			\begin{enumerate}[(a)]
				\item if $h<\hat{h}(\alpha)$, then $\mu_2$ is simple and $u$  is antisymmetric about the $y$-axis;
				\item if $h>\hat{h}(\alpha)$, then $\mu_2$ is simple and $u$ is symmetric about the $y$-axis;  \item if  $h=\hat{h}(\alpha)$, then $\mu_2$ has multiplicity exactly two. Its second eigenspace is spanned by one symmetric eigenfunction $u^s$ and one antisymmetric eigenfunction $u^a$.
			\end{enumerate}
		\end{enumerate}
	\end{Theorem}

	The following results concern the hot spots conjecture on kites.	For $a,h>0$, we define the following polygonal family:
	\[
	K_{a,h}=P_1P_2P_3P_4, \quad P_1=(0,0),\quad P_2=(a,-h),\quad P_3=(1,0),\quad P_4=(a,h).
	\]
	It is a convex kite for $0<a<1$; for $a=1$, $P_2,P_3,P_4$ are collinear and the domain becomes the isosceles triangle $P_1P_2P_4$; for $a>1$, $P_3$ is reentrant and the domain is a concave kite.
	\begin{Theorem}\label{thm1.2}
		Let $u$ be the second Neumann eigenfunction on $K_{a,h}$ and $\mu_2$ be the corresponding eigenvalue, then the following assertions hold for the family $K_{a,h}$.
		\begin{enumerate}[(1)]
			\item If \(0<a<1\), there exists a unique critical height \(\widetilde h(a)\)
			such that \(u\) is symmetric about the \(x\)-axis for \(h<\widetilde h(a)\)
			and antisymmetric for \(h>\widetilde h(a)\), and, in either case,
			$\mu_2$ is simple. At \(h=\widetilde h(a)\), $\mu_2$ has multiplicity precisely two,
			and its second eigenspace is spanned by one symmetric eigenfunction $u^s$ and one
			antisymmetric eigenfunction $u^a$.  In the special case \(a=1/2\),
			the domain \(K_{a,h}\) reduces to a rhombus, for which \(\widetilde h(1/2)=1/2\).
			
			\item If $a=1$, the domain is the isosceles triangle $P_1P_2P_4$.
			For $0<h<1/\sqrt3$, $\mu_2$ is simple, $u$ is symmetric about the
			$x$-axis. For $h>1/\sqrt3$,
			$\mu_2$ is simple and $u$ is antisymmetric. At $h=1/\sqrt3$, the triangle is
			equilateral and $\mu_2$ has multiplicity exactly two, with one symmetric
			and one antisymmetric basis eigenfunction.
			
			\item If $1<a<2$, there exists a unique critical height
			$\widetilde h(a)$ such that $u$ is symmetric
			about the $x$-axis for $h<\widetilde h(a)$ and antisymmetric
			for $h>\widetilde h(a)$. In either case, $\mu_2$ is simple.
			At $h=\widetilde h(a)$, $\mu_2$ has multiplicity exactly two,
			and its eigenspace is spanned by one symmetric eigenfunction
			$u^s$ and one antisymmetric eigenfunction $u^a$.
			
			\item If $a\ge2$, then $\mu_2$ is simple and $u$ is antisymmetric for every $h>0$.
		\end{enumerate}
	\end{Theorem}

\begin{Theorem}\label{thm1.3}
Let $\Omega$ be an isosceles trapezoid or a kite, then every real nonzero second Neumann eigenfunction $u$ on $\Omega$ satisfies
\[ \nabla u(p)\neq 0 \quad \text{for every}\quad p\in\Omega.\]
In particular, the maximum and minimum of $u$ are attained only on $\partial\Omega$. 
\end{Theorem}

\begin{Corollary}\label{cor1.4}
As a consequence of Theorems~\ref{thm1.1} and~\ref{thm1.2},
when the second Neumann eigenvalue of $\Omega$ is simple,
the nonvertex boundary critical points of a corresponding
eigenfunction $u$ can be characterized as follows:
\begin{enumerate}[(1)]
    \item If $\Omega$ is an isosceles trapezoid $Q_{\alpha,h}$, then
    $\mathcal{C}_{nv}(u)=\varnothing$ when $u$ is
    antisymmetric, whereas $\mathcal{C}_{nv}(u)$ consists
    precisely of the midpoints of the two bases when
    $u$ is symmetric.

    \item If $\Omega=K_{a,h}$ is a kite with $a\neq 1$,
    then the following hold:
    \begin{enumerate}[(i)]
        \item If $u$ is symmetric with respect to the
        $x$-axis, then $\mathcal{C}_{nv}(u)\neq\varnothing$
        if and only if $\triangle P_1P_3P_4$ is acute
        but not superequilateral. In this case, all
        nonvertex critical points lie in the interiors
        of the shorter outer edges. Here, an isosceles
        triangle is called superequilateral if
        its vertex angle exceeds $\pi/3$.

        \item If $u$ is antisymmetric with respect to
        the $x$-axis, then
        $\mathcal{C}_{nv}(u)\neq\varnothing$
        if and only if $\triangle P_1P_3P_4$ is
        non-isosceles and $\angle P_1P_4P_3>\pi/2$.
        In this case, all nonvertex critical points
        lie in the interiors of the longer outer edges.
    \end{enumerate}
\end{enumerate}
\end{Corollary}

Figure~\ref{fig:review-geometry} presents schematic diagrams of the main classes of domains considered in this paper.
	\begin{figure}[htbp]
		\centering
			\begin{minipage}[t]{0.47\textwidth}
				\centering
				\begin{tikzpicture}[x=1cm,y=1cm,blue,font=\scriptsize,line width=.65pt]
					\fill[blue!6] (0,0)--(1.5,0)--(.7,1.4)--(0,1.4)--cycle;
					\draw (-1.5,0)--(1.5,0)--(.7,1.4)--(-.7,1.4)--cycle;
					\draw[dashed] (0,-.3)--(0,1.75);
					\node[below left] at (-1.5,0) {$P_1$};
					\node[below right] at (1.5,0) {$P_2$};
					\node[above right] at (.7,1.4) {$P_3$};
					\node[above left] at (-.7,1.4) {$P_4$};
					\node[below left] at (0,0) {$O$};
					\node[above left] at (0,1.4) {$C$};
					\node at (.53,.6) {$R_{\alpha,h}$};
					\node[left] at (0,.7) {D/N};
					\node[below] at (.75,0) {N};
					\node[below] at (.35,1.37) {N};
					\node[right] at (1.1,.7) {N};
					\node at (0,-.65) {(a) Isosceles trapezoid};
				\end{tikzpicture}
			\end{minipage}\hfill
			\begin{minipage}[t]{0.47\textwidth}
				\centering
				\begin{tikzpicture}[x=1cm,y=1cm,blue,font=\scriptsize,line width=.65pt]
					\fill[blue!6] (0,0)--(2.6,0)--(1,1.05)--cycle;
					\draw (0,0)--(1,-1.05)--(2.6,0)--(1,1.05)--cycle;
					\draw[dashed] (-.1,0)--(2.85,0);
					\node[left] at (0,0) {$P_1$};
					\node[below] at (1,-1.05) {$P_2$};
					\node[right] at (2.6,0) {$P_3$};
					\node[above] at (1,1.05) {$P_4$};
					\node at (1.1,.4) {$T_{a,h}$};
					\node[below] at (1.3,0) {D/N};
					\node[above left] at (.5,.525) {N};
					\node[above right] at (1.8,.525) {N};
					\node at (1.3,-1.65) {(b) Convex kite, $0<a<1$};
				\end{tikzpicture}
			\end{minipage}
			
			\medskip
			\begin{minipage}[t]{0.47\textwidth}
				\centering
				\begin{tikzpicture}[x=1cm,y=1cm,blue,font=\scriptsize,line width=.65pt]
					\fill[blue!6] (0,0)--(2.5,0)--(2.5,1.05)--cycle;
					\draw (0,0)--(2.5,-1.05)--(2.5,1.05)--cycle;
					\draw[dashed] (-.1,0)--(2.7,0);
					\fill (2.5,0) circle (1.1pt);
					\node[left] at (0,0) {$P_1$};
					\node[below] at (2.5,-1.05) {$P_2$};
					\node[right] at (2.6,0) {$P_3$};
					\node[above] at (2.5,1.05) {$P_4$};
					\node at (1.85,.35) {$T_{1,h}$};
					\node[below] at (1.3,0) {D/N};
					\node[above left] at (1.25,.525) {N};
					\node[right] at (2.5,.525) {N};
					\node at (1.4,-1.65) {(c) isosceles triangle, $a=1$};
				\end{tikzpicture}
			\end{minipage}\hfill
			\begin{minipage}[t]{0.47\textwidth}
				\centering
				\begin{tikzpicture}[x=1cm,y=1cm,blue,font=\scriptsize,line width=.65pt]
					\fill[blue!6] (0,0)--(1.9,0)--(2.7,1.05)--cycle;
					\draw (0,0)--(2.7,-1.05)--(1.9,0)--(2.7,1.05)--cycle;
					\draw[dashed] (-.1,0)--(2.3,0);
					\node[left] at (0,0) {$P_1$};
					\node[below] at (2.7,-1.05) {$P_2$};
					\node[right,fill=white,inner sep=1pt] at (2.1,0) {$P_3$};
					\node[above] at (2.7,1.05) {$P_4$};
					\node at (1.5,.23) {$T_{a,h}$};
					\node[below] at (1.1,0) {D/N};
					\node[above left] at (1.35,.525) {N};
					\node[below right] at (2.3,.525) {N};
					\node at (1.4,-1.65) {(d) Concave kite, $a>1$};
				\end{tikzpicture}
			\end{minipage}
		\caption{The domains and their shaded half-domains.
				The dashed segment in each domain is the reflection axis: $OC$ in (a)
				and $P_1P_3$ in (b)--(d). On the half-domain it carries a Dirichlet
				condition for the odd branch and a Neumann condition for the even
				branch; all other half-domain edges carry Neumann conditions.}
		\label{fig:review-geometry}
	\end{figure}

\subsection{Sketch of the proofs}

The proofs address two main difficulties: classifying symmetry transitions
and excluding interior critical points at transition heights.
The first requires control of the crossings of two eigenvalue branches.
The second requires ruling out cancellation between the gradients of
symmetric and antisymmetric eigenfunctions in a two-dimensional eigenspace.

\medskip
\noindent\textbf{1. Symmetry decomposition.}
Reflection reduces the symmetry classification to a comparison of
Neumann and mixed Dirichlet--Neumann eigenvalues on the half-domains:
\[
\begin{aligned}
\mu_2(Q_{\alpha,h})&=\min\{\mu(\alpha,h),\lambda(\alpha,h)\},\\
\mu_2(K_{a,h})&=\min\{\mu(a,h),\lambda(a,h)\}.
\end{aligned}
\]
Here $\mu$ is the second Neumann eigenvalue and $\lambda$ is the first
mixed eigenvalue, with Dirichlet condition on the reflection axis and
Neumann conditions on the remaining edges. Their eigenfunctions extend
to symmetric and antisymmetric eigenfunctions on the full domain,
respectively. We determine the endpoint orderings and prove that every
crossing has the same strict orientation. These facts give uniqueness
whenever a symmetry transition occurs.

\medskip
\noindent\textbf{2. Uniqueness of symmetry transitions.}
For isosceles trapezoids $Q_{\alpha,h}$, explicit estimates give
$\lambda(\alpha,h)<\mu(\alpha,h)$ for sufficiently small $h$, while the
limiting isosceles triangle determines the ordering as
$h\rightarrow\tan\alpha$. Using the nodal structure of the symmetric
second Neumann eigenfunction and the monotonicity results on right
trapezoids, we show that $\lambda(\alpha,h)$ is strictly increasing
with respect to $h$, while $\mu(\alpha,h)$ is strictly decreasing after
the first intersection of the two branches. Hence the two branches
can cross at most once, which yields the unique symmetry transition
and the multiplicity statement in Theorem~\ref{thm1.1}.

For kites $K_{a,h}$, both branches decrease with the height, so their
individual monotonicity does not exclude repeated crossings.
We first determine their relative ordering at small and large heights.
For concave kites, the small-height comparison follows from weighted
one-dimensional limit problems and identifies the threshold $a=2$.
The main difficulty is to establish strict transversality at every
crossing. For normalized Neumann and mixed eigenfunctions $\varphi$
and $\nu$ on the upper half-triangle $T_{a,h}$, set
\[
E_N=\int_{T_{a,h}}\varphi_y^2\,d\sigma,
\qquad
E_D=\int_{T_{a,h}}\nu_y^2\,d\sigma.
\]
The height derivative formulas reduce transversality to $E_D>E_N$
at a crossing $\mu(a,h)=\lambda(a,h)=\Lambda$.
A quadratic-form identity yields the uniform estimate
$E_N\le\Lambda/8$. For $0<a<1$, the vertical Poincar\'e inequality
for the mixed eigenfunction then proves transversality.

For $1<a<2$, define the geometric quantity
\[
\Theta(a,h)=\sin^2\!\left[
\frac12\left(\arctan\frac{h}{a}
+\arctan\frac{h}{a-1}\right)\right].
\]
A mixed Rellich identity and strict monotonicity of the bottom normal
flux give $E_D/\lambda>\Theta(a,h)$ in the relevant height range.
For $1<a\le3/2$, a Robin comparison on two subtriangles, combined
with vertical compression, excludes every crossing with $E_D\le E_N$.
For $3/2<a<2$, we project the Neumann eigenfunction onto the first
positive weighted one-dimensional mode. A uniform spectral gap controls
the orthogonal remainder, and two exact weak identities eliminate the
spectral residual, yielding $E_N/\mu<\Theta(a,h)$ at any putative
nontransverse crossing. This argument works at positive height without
differentiating thin-domain asymptotics. Consequently, every crossing
is strictly downward for $\lambda-\mu$. The endpoint orderings then
give the unique transition height in the transition range. For $a\ge2$,
the mixed branch is smaller for every $h>0$. These arguments establish
Theorem~\ref{thm1.2}.

\medskip
\noindent\textbf{3. Critical points at transition heights.}
The exclusion of interior critical points requires an additional
argument at transition heights. For symmetric or antisymmetric
eigenfunctions, half-domain monotonicity and the mixed critical-point
classification exclude critical points within each half-domain.
Monotonicity or the Hopf lemma treats the open symmetry axis, and
reflection gives the conclusion on the full domain.
At a transition height, however, the second eigenspace is
two-dimensional. The absence of critical points for the two basis
eigenfunctions alone does not exclude cancellation of their gradients
in a linear combination. We therefore establish comparison estimates
that treat every nonzero real eigenfunction in this eigenspace.

For isosceles trapezoids, a comparison argument and a minimum principle
show that the Jacobian of the symmetric and antisymmetric basis
eigenfunctions is strictly positive. For convex kites, directional and
rotational inequalities locate any possible critical point, which a
radial comparison function excludes by Lemma~\ref{lem.criticalpoint-comparison}. For concave kites,
directional and rotational derivatives directly exclude such points.
The square and equilateral triangle are treated separately.
This completes the proof of Theorem~\ref{thm1.3} and verifies the hot spots
conjecture for both classes of domains.
    
	\subsection{Outline of the paper}
	Section~\ref{sec:preliminaries} collects the maximum principles and
	corner estimates used below. Sections~\ref{sec:trapezoids}
	and~\ref{sec:kites} establish the spectral and symmetry classifications
	for isosceles trapezoids and kites, proving Theorems~\ref{thm1.1}
	and~\ref{thm1.2}. Section~\ref{sec:critical-points} studies critical
	points: it first treats symmetric and antisymmetric eigenfunctions and the boundary
	classification in Corollary~\ref{cor1.4}, then excludes interior
	critical points at the critical heights for isosceles trapezoids,
	convex kites and concave kites, completing the proof of
	Theorem~\ref{thm1.3}.

	\section{Preliminaries}\label{sec:preliminaries}
	~~~~In this section,  we will first recall some facts about the maximum principle. Throughout the paper, we assume that $\Omega$ is a bounded and connected domain in $\mathbb{R}^{2}.$ We use weak solutions whenever a polygonal corner prevents classical regularity up to the whole boundary. Let $L$ be a uniformly second order elliptic operator,  defined by
	\begin{equation}\label{2.1}
		Lu=\sum_{i,j=1}^n a_{ij}(x)u_{x_ix_j}+\sum_{i=1}^n b_{i}(x)u_{x_i}+c(x)u,
	\end{equation}
	where $a_{ij}=a_{ji}, a_{ij}\in C(\overline{\Omega}), b_i,c\in L^{\infty}(\Omega),$ and $\sum_{i,j=1}^n a_{ij}(x)\xi_i\xi_j\geq \vartheta|\xi|^2$ for any $x\in \Omega, \xi\in \mathbb{R}^n$ and some positive constant $\vartheta.$ Now let $\lambda_1(L,\Omega)$ be the first eigenvalue of the above elliptic operator $L$ in $\Omega$, given by
	$$\lambda_1(L,\Omega)=\sup \{\lambda:\exists u>0~\mbox{such~that}~(L+\lambda)u\leq 0\},$$
	where $u\in W^{2,n}_{loc}(\Omega)\cap C(\overline{\Omega}).$ 
	
	We will need the following generalized versions of the maximum principle, BNV principal eigenvalue theorem, and Hopf lemma.
	\begin{Lemma}\label{lem2.1}
		(Serrin's comparison principle, see \cite{Serrin1971ARMA}) Assume that $u\in C^2(\Omega)\cap C(\overline{\Omega})$ satisfies $Lu\geq 0,$ where the operator $L$ is defined by (\ref{2.1}). If $u\leq 0$ in $\Omega$, then either $u<0$ in $\Omega$ or $u\equiv 0$ in $\Omega$. In particular, if $w$ satisfies $(\Delta+\lambda)w=0$ in a bounded and connected domain $\Omega$, and $w\ge0$ in $\Omega$, then either $w\equiv0$ or $w>0$ in $\Omega$.
	\end{Lemma}
	
	\begin{Lemma}\label{lem2.2}
		(BNV principal eigenvalue theorem, see \cite{BerestyckiVaradhan1994}) 
		For a bounded connected domain, the maximum principle holds in $\Omega$ if and only if $\lambda_1(L,\Omega)>0.$  In particular, if $\lambda<\lambda_1({\color{black}\Delta},\Omega)$ and
		\[
		(\Delta+\lambda)w\ge0\quad\text{in }\Omega,\qquad w\le0\quad\text{on }\partial\Omega,
		\]
		then $w\le0$ in $\Omega$.
	\end{Lemma}
	
	We denote the second Neumann eigenvalue and the first Dirichlet eigenvalue of the Laplace operator $\Delta$ on the domain $\Omega$ by $\mu_2(\Delta,\Omega)$ and $\lambda_1(\Delta,\Omega)$, respectively. Since  $\mu_2(\Delta,\Omega)<\lambda_1(\Delta,\Omega)$ (see \cite[Theorem 4.1]{Rohleder2025}), by Lemma  \ref{lem2.2}, we have the following refined maximum principle.
	\begin{Lemma}\label{lem2.3}
		(Refined maximum principle) Suppose that $u$ satisfies $\Delta u+\mu_2 u=0~\mbox{in}~\Omega$ with $u\leq 0~\mbox{on}~\partial\Omega$, then $u\leq 0$ in $\Omega$.
	\end{Lemma}	
	
	\begin{Lemma}\label{lem2.4}
		(Hopf Lemma, see \cite{Gilbarg1983}) Suppose that
$u\in C^2(\Omega)$ satisfies $(\Delta+\mu_2)u=0$ and $u<0$ in $\Omega$. Let $x_0\in\partial\Omega$ satisfy an interior sphere condition, assume that the boundary is $C^2$ near $x_0$, and that $u$ is $C^1$ up to that boundary neighbourhood. If $u(x_0)=0$, then
$\partial_nu(x_0)>0$. More generally, $\partial_\nu u(x_0)>0$ for every direction $\nu$ with $\nu\cdot n(x_0)>0$.
	\end{Lemma}	
	\noindent We use this Hopf lemma only at relative interior points of edges, never at polygonal vertices.

    \begingroup\color{black}
	We recall the Bessel expansions near a polygonal vertex.
	In polar coordinates $(r,\vartheta)$ centered at the vertex, let
	$S_{\omega,\rho}=\{(r,\vartheta):0<r<\rho,\ 0<\vartheta<\omega\}$,
	where $0<\omega<\pi$ and $\rho>0$. Suppose that
	$w\in H^1(S_{\omega,\rho})$ satisfies $(\Delta+\Lambda)w=0$,
	with $\Lambda>0$, and $\partial_n w=0$ on the side $\vartheta=0$.

	For $q\ge0$, the Bessel function of the first kind has the expansion
	\[
		J_q(z)=\sum_{k=0}^{\infty}
		\frac{(-1)^k}{k!\,\Gamma(q+k+1)}
		\left(\frac z2\right)^{q+2k},\qquad z>0.
	\]
	As in \cite[Section~2]{Hatcher2024PAMS}, substituting
	$z=\sqrt\Lambda r$ and factoring out $r^q$ gives
	\[
		J_q(\sqrt\Lambda r)=r^qg_q(r^2),\qquad
		g_q(s)=\left(\frac{\sqrt\Lambda}{2}\right)^q
		\sum_{k=0}^{\infty}
		\frac{(-\Lambda s/4)^k}{k!\,\Gamma(q+k+1)}.
	\]
	Thus $g_q$ is entire and
	$g_q(0)=(\sqrt\Lambda/2)^q/\Gamma(q+1)>0$.
	Set $H_q(s)=g_q(s)/g_q(0)$, so that $H_q(0)=1$.
	We absorb $g_q(0)$ into the
	expansion coefficients below; the dependence on the fixed
	parameter $\Lambda$ is omitted from the notation.

	\begin{Lemma}\label{lem:corner-bessel-expansions}
		(See \cite[Section~4]{Judge2020Ann} and
		\cite[Section~2, equations~(2.1)--(2.3)]{Hatcher2024PAMS}.)
		With the notation above, the cited Bessel expansions have the
		following consequences.

		(i) Suppose that $\partial_n w=0$ also on $\vartheta=\omega$.
		Put ${\color{black}\delta}=\pi/\omega$ and $c_0=w(0)$.
		Then
		\begin{equation}\label{eq.add7.neumann}
			w=c_0J_0(\sqrt\Lambda r)
			+\sum_{j\ge1}c_jr^{j{\color{black}\delta}}H_{j{\color{black}\delta}}(r^2)\cos(j{\color{black}\delta}\vartheta).
		\end{equation}
		In particular,
		\begin{equation}\label{eq:corner-neumann-two-modes}
			w=c_0J_0(\sqrt\Lambda r)
			+c_1r^{\color{black}\delta}H_{\color{black}\delta}(r^2)\cos({\color{black}\delta}\vartheta)+O(r^{2{\color{black}\delta}}),
		\end{equation}
		and
		\begin{equation}\label{eq.add7.corner-first}
			w=c_0\left(1-\frac{\Lambda r^2}{4}\right)
			+c_1r^{\color{black}\delta}\cos({\color{black}\delta}\vartheta)+O(r^{\min\{4,2{\color{black}\delta}\}}).
		\end{equation}

		(ii) Suppose instead that $w=0$ on $\vartheta=\omega$.
		Put ${\color{black}\delta}=\pi/(2\omega)$.
		Then
		\begin{equation}\label{eq:corner-mixed-first}
			w=d_0r^{\color{black}\delta}\cos({\color{black}\delta}\vartheta)+O(r^{\min\{{\color{black}\delta}+2,3{\color{black}\delta}\}}).
		\end{equation}
		If $w>0$ in the sector, then $d_0>0$.

		All remainder estimates are uniform for $0\le\vartheta\le\omega$
		and may be differentiated:
		\begin{equation}\label{eq:corner-remainder-derivatives}
			D^\zeta O(r^\gamma)=O(r^{\gamma-|\zeta|})
			\quad\text{for every multi-index }\zeta\in\mathbb N_0^2.
		\end{equation}
	\end{Lemma}
	\begin{proof}
		We use the Bessel series in \cite[Section~4]{Judge2020Ann} and
		\cite[Section~2]{Hatcher2024PAMS}: after normalizing the radial
		factors as above, the Neumann series is \eqref{eq.add7.neumann},
		while the mixed series is
		\[
			w=\sum_{j\ge0}d_jr^{(2j+1){\color{black}\delta}}H_{(2j+1){\color{black}\delta}}(r^2)
			\cos((2j+1){\color{black}\delta}\vartheta).
		\]

		We now deduce the remaining assertions from these series.
		The identities $J_0(\sqrt\Lambda r)=1-\Lambda r^2/4+O(r^4)$ and
		$H_q(r^2)=1+O(r^2)$ give the stated truncations.
		On a smaller sector, the series tails and their derivatives converge
		uniformly after the corresponding powers of $r$ are factored out;
		termwise differentiation therefore gives
		\eqref{eq:corner-remainder-derivatives}.

		Finally, if $w>0$, its first nonzero angular term cannot have order
		$(2j+1){\color{black}\delta}$ with $j\ge1$, since the corresponding cosine changes sign
		in $0<\vartheta<\omega$.
		Thus the first term has order ${\color{black}\delta}$, and its coefficient $d_0$ is positive.
	\end{proof}

	\begin{Remark}\label{rem:corner-regularity}
		The differentiated expansions give
		\begin{equation}\label{eq:corner-gradient-bounds}
			|\nabla w|=
			\begin{cases}
				O(r^{\min\{1,{\color{black}\delta}-1\}})&\text{in case (i)},\\
				O(r^{{\color{black}\delta}-1})&\text{in case (ii)}.
			\end{cases}
		\end{equation}
		Thus $\nabla w\to0$ at Neumann corners and acute mixed corners.
		Moreover, near a Neumann corner or a mixed corner with
		$\omega\le\pi/2$,
		\begin{equation}\label{eq:corner-regularity}
			w\in H^2(S_{\omega,\rho'})\cap C^1(\overline{S_{\omega,\rho'}}),
			\qquad 0<\rho'<\rho.
		\end{equation}
		Indeed, ${\color{black}\delta}>1$ makes the singular second derivatives square
		integrable; at a right mixed corner, ${\color{black}\delta}=1$, the leading term is
		linear and the remainder is $O(r^3)$.
		See also \cite[Remark~4.5]{Hatcher2024PAMS}.
	\end{Remark}

	\endgroup

	\section{Symmetry on isosceles trapezoids}\label{sec:trapezoids}
	
	Throughout this section, let $Q_{\alpha,h}=P_1P_2P_3P_4$ be the isosceles trapezoid with
	\begin{equation}\label{eq:trapezoid-vertices}
		P_1=(-1,0),\quad P_2=(1,0),\quad
		P_3=(1-h\cot\alpha,h),\quad P_4=(-1+h\cot\alpha,h),
	\end{equation}
	where $h$ is the height and $\alpha=\angle P_1P_2P_3$ the base angle, satisfying
	\[
	0<\alpha<\frac\pi2,\qquad 0<h<\tan\alpha.
	\]
	Let $O=(0,0)$ and $C=(0,h)$ be the midpoints of the two bases. The right half-domain is
	\begin{equation}\label{eq:trapezoid-half-domain}
		R_{\alpha,h}=OP_2P_3C=Q_{\alpha,h}\cap\{x>0\}.
	\end{equation}
	When $\alpha$ is fixed, we also write $Q_h=Q_{\alpha,h}$.

	Let $H^1_s(Q_{\alpha,h})$ denote the subspace of functions in
	$H^1(Q_{\alpha,h})$ that are symmetric about the $y$-axis,
	and let $H^1_a(Q_{\alpha,h})$ denote the subspace of functions in
	$H^1(Q_{\alpha,h})$ that are antisymmetric about the $y$-axis.
	We define the smallest symmetric eigenvalue
	$\mu(\alpha,h)$ and the smallest antisymmetric eigenvalue $\lambda(\alpha,h)$ as follows:
	\begin{align}\label{eq:trapezoid-spectra}
		\mu(\alpha,h) &=\inf_{\varphi\in H^1_s(Q_{\alpha,h})\backslash\{0\}\atop\int_{Q_{\alpha,h}} \varphi=0} 
		\frac{\int_{Q_{\alpha,h}} |\nabla\varphi|^2d\sigma}{\int_{Q_{\alpha,h}} \varphi^2d\sigma}, \\
		\lambda(\alpha,h) &=\inf_{\varphi\in H^1_a(Q_{\alpha,h})\backslash\{0\}}
		\frac{\int_{Q_{\alpha,h}} |\nabla\varphi|^2d\sigma}{\int_{Q_{\alpha,h}} \varphi^2d\sigma}.	
	\end{align}
	Clearly, $\mu(\alpha,h)$ is the second Neumann eigenvalue of right trapezoid $R_{\alpha,h}$,
	and $\lambda(\alpha,h)$ is the first eigenvalue of the mixed boundary problem
	\begin{equation} \label{eq.mix}
		\left\{ 
		\begin{aligned}
			& \Delta \nu + \lambda(\alpha,h) \nu=0 \text{ in } R_{\alpha,h}, \\
			& \nu =0 \text{ on } OC, \\
			& \partial_n \nu  =0 \text{ on } \partial R_{\alpha,h}\backslash OC.
		\end{aligned} \right.
	\end{equation}
	Since the Neumann Laplacian commutes with
	reflection across the $y$-axis, every Neumann eigenfunction can be
	decomposed into its symmetric and antisymmetric parts. Consequently,
	$\mu_2(Q_{\alpha,h})=\min\{\mu(\alpha,h),\lambda(\alpha,h)\}$.

	\subsection{Right trapezoids}\label{subsec:right-trapezoids}
	In this subsection, we study the Neumann and mixed eigenfunctions on $R_{\alpha,h}$,
	investigate the non-existence of
	non-vertex critical points and establish directional monotonicity of the eigenfunctions and
	height monotonicity of the relevant eigenvalues.
	
	We first recall the definition of lip domain, which plays a key role in the following analysis.
	\begin{Definition}\label{def1}
		(\cite{Atar2004JAMS}) A planar set  $D$ is called a lip domain if it is Lipschitz, open, bounded, connected, and given by
		$$D=\{(x_1,x_2):f_1(x_1)<x_2<f_2(x_1)\},$$
		where $f_1,f_2$ are functions with Lipschitz constants at most 1.
	\end{Definition}
	
	The following lemma gives the key properties of the second Neumann eigenfunction on lip domains.
	
	\begin{Lemma} \label{lem.lip} (\cite{Atar2004JAMS,Rohleder2023}) 
		Assume that $\Omega$ is a lip domain with a piecewise $C^\infty$ boundary
		whose corners are convex. Then
		\begin{enumerate}[(1)]
			\item If $\Omega$ is not a square, then the second Neumann eigenvalue $\mu_2(\Omega)$ is simple.
			\item If $\Omega$ is not a rectangle, then the second Neumann eigenfunction,
			possibly after a sign change, satisfies that the directional derivatives in the directions of
			$e_1 + e_2$ and $e_1 - e_2$ are positive inside $\Omega$.
		\end{enumerate}
		where $e_1$ and $e_2$ are the standard basis vectors in $\mathbb{R}^2$.
	\end{Lemma}
	
% KEEP: Lemma lem.y and its boundary argument are used for monotonicity,
% the spectral proof and Section 5; they are not duplicate deletion candidates.
	Applying Lemma \ref{lem.lip} to the right trapezoids $R_{\alpha,h}$ gives the following result.
	\begin{Lemma} \label{lem.y}
		Let $\varphi$ be a second Neumann eigenfunction on
		$R_{\alpha,h}=OP_2P_3C$. Then
		\begin{enumerate}[(1)]
			\item The second Neumann eigenvalue is simple.
			\item $\varphi$ does not have any non-vertex critical points, and  $\varphi$ is monotonic along the four edge directions.
			\item $\varphi$ attains its global extrema at $P_2$ and $C$; $O$ and $P_3$ are saddle points.
		\end{enumerate}
	\end{Lemma}
	\begin{proof}
		Introduce the rotated orthonormal basis
		$\displaystyle
		\widetilde e_1=\frac{\partial_x-\partial_y}{\sqrt2},
		\widetilde e_2=\frac{\partial_x+\partial_y}{\sqrt2}.
		$
		In the $(\widetilde e_1,\widetilde e_2)$ coordinates,
		the right trapezoids $R_{\alpha,h}$ are lip domains.
		Lemma \ref{lem.lip} implies that the second Neumann eigenvalue
		is simple. Moreover,
		possibly after a sign adjustment, we have
		\begin{align*}
			\partial_x \varphi > 0,\quad
			\partial_y \varphi < 0 \quad \text{and} \quad
			\partial_\tau \varphi < 0 \quad \text{in } R_{\alpha,h},
		\end{align*}
		where $\tau=(-\cos\alpha,\sin\alpha)$ points from $P_2$ toward $P_3$.
		
		Suppose that there exists a non-vertex critical point $p$ in $\overline{R_{\alpha,h}}$.
		Then $p$ must lie on one of the edges of $R_{\alpha,h}$. If $p$ lies on the edge $OP_2$,
		we note that $\partial_x \varphi$ attains its minimum zero at $p$.
		By  Hopf lemma, the second-order mixed derivative $\partial_y\partial_x \varphi$
		is positive at $p$. However, the Neumann boundary condition
		for $\varphi$ implies that $\partial_y\partial_x \varphi = \partial_x\partial_y \varphi$ vanishes at $p$,
		which yields a contradiction. The same argument applies to the remaining edges.
		Therefore, there are no non-vertex critical points in $\overline{R_{\alpha,h}}$.
		Thus, we have
		\begin{align*}\color{black}
			\partial_x \varphi > 0 \text{ in }
                \overline{R_{\alpha,h}}\backslash\{\overline{OC}, P_2,P_3\},\quad
			\partial_y \varphi < 0 \text{ in }
                \overline{R_{\alpha,h}}\backslash\{\overline{OP_2},\overline{CP_3} \},\quad
			\partial_\tau \varphi < 0 \text{ in }
                \overline{R_{\alpha,h}}\backslash\{\text{Vertices}\}.
		\end{align*}
		This further implies that $P_2$ is the unique global maximizer of $\varphi$,
		$C$ is the unique global minimizer of $\varphi$, and $O$ and $P_3$ are saddle points of $\varphi$.
	\end{proof}

	Next, we examine the monotonicity of the second Neumann eigenvalue with respect to the height of the trapezoid, which will provide a crucial tool for determining the symmetry or antisymmetry of the second Neumann eigenfunction later.
	\begin{Lemma} \label{lem.dec}
		Fix $\alpha$ and consider the family $R_{\alpha,h}=OP_2P_3C$.
		At any height for which the nodal line of the second Neumann eigenfunction connects
		the slanted edge $P_2P_3$ to the vertical edge $OC$, we have
		$\partial_h\mu(\alpha,h)<0$. Hence $\mu(\alpha,h)$ is strictly decreasing
		on every interval on which this nodal configuration persists.
	\end{Lemma}
	\begin{proof}
		Lemma \ref{lem.y} shows that $\mu(\alpha,h)$ is simple, and the domain varies only at the top edge $P_3C$
		moving in the direction $n = (0,1)$ with 
		unit speed. A direct calculation gives the deformation vector field
		\begin{align}
			V(x,y)=\frac{1}{h}\left(-\frac{xy}{\tan\alpha-y}, y\right).
		\end{align}
		Thus the Hadamard formula (see Theorem 2.5.7 in \cite{Henrot2006}) yields
		\begin{align}
			\partial_h\mu(\alpha,h)=	\int_{R_{\alpha,h}}\operatorname{div}
			\left[\bigl(|\nabla \varphi|^2-\mu(\alpha,h)\varphi^2\bigr)V	\right]d\sigma,
		\end{align}
		where $\varphi$ is the second Neumann eigenfunction with
		the normalization $\|\varphi\|_{L^2(R_{\alpha,h})}=1$.
		Hence, by the divergence theorem and the Neumann boundary condition,
		noting that $V\cdot n=0$ on the three fixed edges and
		$V\cdot n=1$ on $CP_3$, we obtain
		\begin{equation}\label{eq:hadamard}
			\begin{aligned}
				\partial_h\mu(\alpha,h)&=\int_{{C}P_3}\left(|\nabla \varphi|^2-\mu(\alpha,h)\varphi^2\right) dx \\
				&= \int_{{C}P_3} \left( (\partial_x {\varphi})^2 - \mu(\alpha,h) \varphi^2 \right) dx.
			\end{aligned}
		\end{equation}
		Substituting $\mu(\alpha,h) \varphi = -(\partial_{xx} \varphi + \partial_{yy} \varphi)$ into \eqref{eq:hadamard}, we have
		\begin{equation}
			\begin{aligned}
				\partial_h\mu(\alpha,h) &= \int_{CP_3} \left[ (\partial_x \varphi)^2 + \varphi(\partial_{xx} {\varphi} + \partial_{yy} \varphi) \right] dx \\
				&= \int_{CP_3} \partial_x (\varphi \partial_x \varphi) dx + \int_{CP_3} \varphi \partial_{yy}\varphi dx \\
				&= \int_{CP_3} \varphi \partial_{yy} \varphi  dx.
			\end{aligned}
		\end{equation}
		\textcolor{black}{The endpoint term $[\varphi\partial_x\varphi]_{C}^{P_3}$ vanishes: all corners of $R_{\alpha,h}$ are convex Neumann corners, so Remark~\ref{rem:corner-regularity} gives $\nabla\varphi\to0$ at $P_3$ and $C$.}

		Now,  we determine the signs of $\varphi$ and $\partial_{yy} \varphi$
		on $P_3C$. By Lemma \ref{lem.y}, after possibly changing the sign of $\varphi$,
		we have $\partial_y\varphi<0$ in $R_{\alpha,h}$.
		Together with the Neumann boundary condition, the Hopf lemma
		(Lemma \ref{lem2.4}) then gives $\partial_{yy}\varphi>0$ on the open edge $P_3C$.
		Finally, since the nodal line of $\varphi$ connects the edge $P_2P_3$ and the edge $OC$,
		this implies that the nodal line divides $R_{\alpha,h}$ into an upper
		region and a lower region. Since $\partial_y \varphi < 0$ holds in $R_{\alpha,h}$,
		we have that $\varphi$ must be negative on the upper region,
		in particular, $\varphi<0$ holds on $P_3C$.
		Combining these facts, we complete the proof. 
	\end{proof}

	For the first eigenfunction with mixed boundary conditions on a right trapezoid, we have the following result.
	
	\begin{Lemma} \label{lem.x}
		Let $\nu$ be the first eigenfunction of the mixed boundary problem \eqref{eq.mix}.	
		Then
		\begin{align*}
			\partial_x\nu&>0\quad\text{in }\overline{R_{\alpha,h}}\setminus\{P_2,P_3\},\\
			\partial_y\nu&<0\quad\text{in }R_{\alpha,h}\cup P_2P_3.
		\end{align*}
	\end{Lemma}

	\begin{proof}
		We reflect $R_{\alpha,h}$ across the edge $OP_2$ and set
		\[
		D=R_{\alpha,h}\cup\{(x,-y):(x,y)\in R_{\alpha,h}\}.
		\]
		We then reflect $D$ across the line $x=0$ and set
		\[
		\Omega= D\cup\{(-x,y):(x,y)\in  D\}.
		\]
		By extending $\nu$ evenly across $OP_2$ and oddly across $x=0$,
		we obtain a Neumann eigenfunction $U$ on $\Omega$.
		The domain $\Omega$ is a convex polygon symmetric about both coordinate axes,
		and $\lambda(\alpha,h)$ is the lowest Neumann eigenvalue among those eigenfunctions
		that are antisymmetric about $x=0$.
		
		Therefore, by \cite[Theorem~1.1]{Jerison2000JAMS}, after possibly changing the sign of $\nu$, we have
		\[
		\partial_x \nu > 0 \quad\text{and}\quad \partial_y \nu < 0 \quad\text{in } R_{\alpha,h}.
		\]
		Finally, by the Hopf lemma, exactly as in the proof of Lemma \ref{lem.y}, the corresponding strict inequalities extend to the relative interiors of the boundary edges as well. This proves the lemma.
	\end{proof}

	For the mixed boundary value problem, we examine the monotonicity of the first eigenvalue with respect to the height,
	which will provide a crucial tool for determining the symmetry or antisymmetry of the second Neumann eigenfunction later.
	\begin{Lemma} \label{lem.inc}
		Fix $\alpha$ and consider the family $R_{\alpha,h}=OP_2P_3C$,
		let $\lambda(\alpha,h)$ be the first mixed eigenvalue of \eqref{eq.mix}.
		Then $\lambda(\alpha,h)$ is strictly increasing as a function of $h$.
	\end{Lemma}
	\begin{proof}
		It is well known that the first eigenvalue $\lambda(\alpha,h)$ can be characterized as
		\begin{align}
			\lambda(\alpha,h)=\inf_{v\in H^1(R_{\alpha,h})\setminus\{0\} \atop  v=0 \text{ on } OC}
			\frac{\int_{R_{\alpha,h}} |\nabla v|^2 d\sigma}{\int_{R_{\alpha,h}} v^2 d\sigma}.
		\end{align}
		Let $h < h'$, and let $\nu'$ denote the first positive eigenfunction of $R_{\alpha,h'}$.
		Since $R_{\alpha,h} \subset R_{\alpha,h'}$, let ${\nu}$ be the restriction of $\nu'$ on $R_{\alpha,h}$.
		It is evident that $\nu = 0$ on the edge $OC$, and thus $\nu$ is an admissible test function.
		Thus, 
		\begin{align*}
			\lambda(\alpha,h) &\le \frac{\int_{R_{\alpha,h}} |\nabla \nu|^2d\sigma}{\int_{R_{\alpha,h}} {\nu}^2d\sigma} 
			= \frac{\int_{\partial {R_{\alpha,h}}} {\nu}{\partial_n \nu ds} - \int_{{R_{\alpha,h}}} {\nu}{\Delta \nu d\sigma}}{\int_{{R_{\alpha,h}}} {\nu}^2d\sigma}.
		\end{align*}
		Since ${\nu}'$ satisfies the mixed boundary conditions, Lemma \ref{lem.x}
		yields $\partial_y \nu'<0$ in $R_{\alpha,h'}$. It follows that
		\begin{align*}
			\int_{\partial {R_{\alpha,h}}} {\nu}{\partial_n \nu ds} =\int_{\partial {R_{\alpha,h}}} {\nu}'{\partial_n \nu' ds} = \int_{CP_3}  {\nu}'\partial_y  {\nu}'dx <0,
		\end{align*}
		we obtain 
		\begin{align*}
			\lambda(\alpha,h) <
			\frac{ - \int_{R_{\alpha,h}}{\nu} {\Delta \nu\,d\sigma}}{\int_{R_{\alpha,h}} {\nu}^2d\sigma}
			=\lambda(\alpha,h').
		\end{align*}
		This completes the proof.
	\end{proof}

	\subsection{Isosceles trapezoids}\label{subsec:isosceles-trapezoids}
	
	We first recall the comparison of the two branches on the limiting isosceles triangle.
	\begin{Lemma}\label{lem.tri} (\cite{Laugesen2010JDE})
		At the limiting height $h=\tan\alpha$, $Q_{\alpha,h}$ becomes the isosceles triangle
		$T_\alpha=\triangle P_1P_2P_\alpha$, where $P_\alpha=(0,\tan\alpha)$. Then
		\begin{align*}
			\mu(\alpha,\tan\alpha)<\lambda(\alpha,\tan\alpha)\text{ if } \alpha>\frac\pi3,
			\qquad
			\lambda(\alpha,\tan\alpha)<\mu(\alpha,\tan\alpha) \text{ if } \alpha<\frac\pi3,
		\end{align*}
		and $\lambda(\alpha,\tan\alpha)=\mu(\alpha,\tan\alpha)$ holds if $\alpha=\frac\pi3$.
	\end{Lemma}
	
	For an isosceles trapezoid, we have the following simple estimate for $\lambda(\alpha,h)$.
	\begin{Lemma} \label{lem.D}
		Fix $\alpha$ and let $Q_h=Q_{\alpha,h}$ be as above.
		Then
		\begin{align}
			\lambda(\alpha,h)< \frac{\pi^2}{4}\frac{1}{1-h\cot\alpha}.
		\end{align}
	\end{Lemma}   
	\begin{proof}
		Let $v=\sin(\pi x/2)$, let $k(y)=1-y\cot\alpha$,
		which describes the slanted edge,  a direct computation yields
		\begin{equation}
			\begin{aligned}
				\int_{Q_h}v^2 d\sigma &= 2 \int_{0}^{h} dy \int_{0}^{k(y)} \sin^2(\pi x/2) dx \\
				&=\int_{0}^{h} \left( k(y)- \pi^{-1} \sin\left(\pi k(y)\right)\right) dy \\
				&= h- \frac{h^2}{2}\cot\alpha -\pi^{-1} \int_{0}^{h}\sin(\pi\cot\alpha y)dy \\
				& > h-h^2\cot\alpha ,
			\end{aligned}
		\end{equation}
		and
		\begin{equation}
			\begin{aligned}
				\int_{Q_h}|\nabla v|^2 d\sigma &= \frac{\pi^2}{2} \int_{0}^{h} dy \int_{0}^{k(y)} \cos^2(\pi x/2) dx \\
				&=\frac{\pi^2}{4}\left(h-\frac{h^2}{2}\cot\alpha\right)+\frac{\pi}{4}\int_{0}^{h}\sin(\pi\cot\alpha y) dy\\
				& < \frac{\pi^2}{4}h .
			\end{aligned}
		\end{equation}
		The function $v\in H^1_a(Q_h)$ is admissible in the variational characterization. Hence
		\begin{align}
			\lambda(\alpha,h)\le \frac{\int_{Q_h}|\nabla {v}|^2 d\sigma}{\int_{Q_h} {v}^2 d\sigma}< \frac{\pi^2}{4}\frac{1}{1-h\cot\alpha}.
		\end{align}
	\end{proof}
	
	Next, we have the following lower bound estimate for $\mu(\alpha,h)$.
	\begin{Lemma} \label{lem.C}
		Let $Q_h=Q_{\alpha,h}$ be as above and assume additionally that $h\le\cot\alpha$.
		Then
		\begin{align}
			\mu(\alpha,h)\ge\frac{1-h\cot\alpha }{1+\cot\alpha+\cot^2\alpha -h\cot\alpha} \pi^2.
		\end{align}
	\end{Lemma}
	\begin{proof}
		Let $\mathcal R_h=(-1,1)\times(0,h)$ and define a map $\gamma(\xi, \eta):\mathcal R_h\to Q_h$ by
		\begin{equation}
			\left\{
			\begin{aligned}
				&x=k(\eta)\xi, \\
				&y=\eta.
			\end{aligned}
			\right.
		\end{equation}
		where $k(\eta)=1-\eta\cot\alpha$. The map $\gamma$ preserves symmetry with respect to the $y$-axis.
		
		For any ${v}\in H^1_s(Q_h)$ satisfying $\displaystyle\int_{Q_h} {v}  dxdy = 0$,
		we define $\widetilde v=v\circ \gamma\in H^1_s(\mathcal R_h)$,
		and choose a constant $c$ such that $\displaystyle\int_{\mathcal R_h} (\widetilde v+c) d\xi d\eta = 0$.
		It follows that
		\begin{equation*}
			\begin{aligned}
				\int_{\mathcal R_h} |\widetilde v+c|^2 d\xi d\eta &= \int_{Q_h} |v+c|^2\frac{1}{k(y)}dxdy \\
				& > \int_{Q_h} |{v}+c|^2 dxdy \ge \int_{Q_h} {v}^2dxdy,
			\end{aligned}
		\end{equation*}
		where the last inequality uses $\displaystyle\int_{Q_h} {v} dxdy = 0$.

		Since 
		\begin{align*}
			&\partial_\xi(\widetilde v+c)=\partial_x {v} \frac{\partial x}{\partial \xi}
			+ \partial_y{v} \frac{\partial y}{\partial \xi}=\partial_x v  k(\eta), \\
			&\partial_\eta(\widetilde v+c)=\partial_xv \frac{\partial x}{\partial \eta}
			+ \partial_yv\frac{\partial y}{\partial \eta}=-\partial_xv (\cot\alpha) \xi +\partial_yv .
		\end{align*}
		Then
		\begin{equation*}
			\begin{aligned}
				\int_{\mathcal R_h} |\nabla(\widetilde v+c)|^2 d\xi d\eta
				&=\int_{\mathcal R_h} (\partial_x v)^2 k^2(\eta) d\xi d\eta
				+\int_{\mathcal R_h} \left(\partial_xv  (\cot\alpha) \xi-\partial_yv \right)^2 d\xi d\eta \\
				&=\int_{Q_h} (\partial_x v)^2 k(y) dxdy
				+\int_{Q_h} \left((\cot\alpha)\partial_xv\frac{x}{k(y)}-\partial_yv\right)^2\frac{1}{k(y)}dxdy,
			\end{aligned}
		\end{equation*}
		notice that
		\begin{align*}
			\left(\cot\alpha\partial_x {v}\frac{x}{k(y)}-\partial_y {v} \right)^2
			\le (\cot\alpha+\cot^2\alpha) {(\partial_x v)^2} \frac{x^2}{k(y)^2}
			+\left(1+\cot\alpha\right) {(\partial_y v)^2},
		\end{align*}
		then
		\begin{equation*}
			\begin{aligned}
				\int_{\mathcal R_h} |\nabla(\widetilde v+c)|^2 d\xi d\eta
				\le &\int_{Q_h} (\partial_x v)^2 dxdy \\
				& +\int_{Q_h} \left[ (\cot\alpha+\cot^2\alpha)(\partial_x v)^2
				+\left(1+\cot\alpha\right)(\partial_y v)^2 \right] k^{-1}(h)dxdy \\
				\le &  \left(1+(\cot\alpha+\cot^2\alpha)k^{-1}(h)\right) \int_{Q_h} |\nabla v|^2 dxdy,
			\end{aligned}
		\end{equation*}
		where the last inequality holds by virtue of $h\le\cot\alpha$.
		
		{\color{black}By separation of variables, we find that the smallest symmetric eigenvalue of $\mathcal R_h$ equals $\min\left\{\pi^2,\frac{\pi^2}{h^2}\right\}$.
		Since $h<\tan\alpha$ and $h\le\cot\alpha$, then $h\le1$ and the smallest symmetric eigenvalue of $\mathcal R_h$ is $\pi^2$.}
		Hence, for the symmetric zero-mean function
		$\widetilde v+c$, the variational principle gives Poincar\'e inequality
		\[
		\pi^2 \int_{{\mathcal R_h}} |{\widetilde v}+c|^2d\xi d\eta
		\le
		\int_{{\mathcal R_h}} |\nabla({\widetilde v}+c)|^2d\xi d\eta.
		\]
		Thus for any ${v}\in H^1_s(Q_h)$ satisfying $\int_{Q_h} v  dxdy = 0$, we have
		\begin{align*}
			\pi^2 \int_{Q_h} v^2dxdy < 
			\left(1+(\cot\alpha+\cot^2\alpha)k^{-1}(h)\right) \int_{Q_h} |\nabla v|^2 dxdy.
		\end{align*}
		That is
		\begin{align}
			\mu(\alpha,h)\ge \frac{k(h)}{k(h)+\cot\alpha+\cot^2\alpha} \pi^2.
		\end{align}
	\end{proof}
	
	Combining the upper bound estimate for $\lambda(\alpha,h)$
	and the lower bound estimate for $\mu(\alpha,h)$,
	we obtain the following comparison result for sufficiently small height.
	\begin{Lemma}\label{cor.a}
		Fix $\alpha$,
		there exists $h_0(\alpha)>0$ such that, for all $h< h_0(\alpha)$, we have
		\begin{align}
			\lambda(\alpha,h) < \mu(\alpha,h).
		\end{align}
	\end{Lemma}
	\begin{proof}
		If the base angle $\alpha\le \frac{\pi}{4}$, then $Q_h$ is a lip domain. Lemma \ref{lem.lip}
		implies that the second Neumann eigenvalue $\mu_2$ is simple and, after possibly changing the sign of $u$,
		$\partial_x u >0$ in $Q_h$.
		Let $O$ be the midpoint of $P_1P_2$.
		If $\partial_x u =0$ at $O$,  then  $\partial_x u$ attains its minimum value zero  at $O$,
		and Hopf Lemma \ref{lem2.4} yields a contradiction.
		Thus $O$ is not a critical point. If $u$ is
		symmetric about the $y$-axis, then $O$ would be a critical point, which is a contradiction.
		Hence $u$ is antisymmetric about the $y$-axis,
		which implies that
		$\lambda(\alpha,h)< \mu(\alpha,h)$ for all $h<\tan\alpha$.
		We consider the following functions.
		\begin{align}
			f(h)&=\pi^2\frac{1-h \cot\alpha}{1+\cot\alpha+\cot^2\alpha-h\cot\alpha},
			\quad g(h)=\frac{\pi^2}{4}\frac{1}{1-h\cot\alpha}.
		\end{align}
		If $\cot\alpha+\cot^2\alpha<3$, then $g(0)<f(0)$.
		By the continuity of $f(h)$ and $g(h)$, there exists a constant $h_0$, such that $g(h)< f(h)$ holds for $h< h_0$. By Lemmas \ref{lem.D} and \ref{lem.C}, we have
		\begin{equation}
			\lambda(\alpha,h)<g(h)\leq f(h)\leq \mu(\alpha,h).
		\end{equation}
		Thus if $\alpha>\frac{\pi}{4}$, then
		$\cot\alpha+\cot^2\alpha<3$ holds, and
		the conclusion still holds.
	\end{proof}
	
	Combining the monotonicity results established
	above, we will show that the symmetric and antisymmetric eigenvalue branches can cross at most once.
	\begin{Lemma}\label{lem.imp}
		Fix $\alpha$. Then there exists at most one $\hat{h}$ such that
		$\mu(\alpha,\hat h) = \lambda(\alpha,\hat h)$. In fact,
		$\lambda(\alpha,h)$ is strictly increasing with respect to $h$,
		while $\mu(\alpha,h)$ is strictly decreasing after the first intersection of
		the two branches $\mu(\alpha,h)$ and $\lambda(\alpha,h)$.
	\end{Lemma}
	\begin{proof}
		Let $O$ and $C$ be the midpoints of $P_1P_2$
		and $P_3P_4$ respectively, and let $u$ be the
		eigenfunction corresponding to the lowest positive
		symmetric eigenvalue $\mu(\alpha,h)$.
		If $\mu(\alpha,h) \le \lambda(\alpha,h)$, then $u$ is the
		second Neumann eigenfunction of $Q_h$ and is
		symmetric with respect to the $y$-axis.
		Consequently, $\mu(\alpha,h)$ is equal to the second
		Neumann eigenvalue of the right trapezoid
		$OP_2P_3C$, and $u$ is the corresponding second
		Neumann eigenfunction on $OP_2P_3C$.

		By \cite[Theorem 5.2]{Judge2020Ann}, the nodal set of $u$ in
		$OP_2P_3C$ consists of a single simple arc.
		Moreover, by \cite[Lemma 3.3]{Judge2020Ann},
		this arc cannot have both endpoints on the same Neumann edge.
		We now claim that the nodal arc connects $P_2P_3$ to $OC$. Indeed, if its
		endpoints lie on two adjacent outer edges or on the two bases,
		reflection across $OC$ produces at least three nodal domains in $Q_h$,
		contrary to Courant's nodal domain theorem. If one endpoint lies on
		$OC$ and the other lies on one of the bases, reflection across $OC$
		produces a nodal arc in $Q_h$ whose two endpoints lie on the same Neumann edge, contradicting \cite[Lemma 3.3]{Judge2020Ann}.
        Therefore, the nodal arc in $OP_2P_3C$
		connects the slanted edge $P_2P_3$ to the symmetry edge $OC$. Hence, Lemma \ref{lem.dec} implies that
		$\mu(\alpha,h)$ is strictly
		decreasing in $h$, and Lemma \ref{lem.inc}
		implies that $\lambda(\alpha,h)$ is strictly increasing in $h$.

		Corollary \ref{cor.a} implies that for sufficiently small $h$,
		$\mu(\alpha,h) \neq \lambda(\alpha,h)$. Let $\hat{h} > 0$ be the first value (if it exists) such that
		$\mu(\alpha,\hat h) = \lambda(\alpha,\hat h)$. The monotonicity of $\mu(\alpha,h)$ and $\lambda(\alpha,h)$
		implies that $\mu(\alpha,h) <\lambda(\alpha,h)$ holds for all $h > \hat{h}$.
		This completes the proof.
	\end{proof}
	
	We now prove Theorem~\ref{thm1.1}.
	
	\begin{proof} [Proof of Theorem \ref{thm1.1}]
		
		Fix the base angle $\alpha$ and let $Q_h=Q_{\alpha,h}$.
		Denote the two eigenvalue branches by
		$\mu(\alpha,h)$ and $\lambda(\alpha,h)$, as defined in
		\eqref{eq:trapezoid-spectra}. 
		By Lemma \ref{lem.y}, $\mu(\alpha,h)$ is simple, while the simplicity of
		$\lambda(\alpha,h)$ follows from the simplicity of the first eigenvalue of the mixed problem.
		The standard spectral stability theorem for the Neumann Laplacian under Hausdorff convergence
		and uniform cone condition (see \cite[Chapter 2]{Henrot2006}) gives their continuity on the
		nondegenerate interval $0<h<\tan\alpha$, as well as at the degenerate endpoint $h=\tan\alpha$.

		Since the Neumann Laplacian commutes with
		reflection across the $y$-axis, every Neumann eigenfunction can be
		decomposed into its symmetric and antisymmetric parts. Consequently,
		$\mu_2(Q_h)=\min\{\mu(\alpha,h),\lambda(\alpha,h)\}$. Hence, if
		$\mu(\alpha,h)<\lambda(\alpha,h)$,
		then $\mu_2(Q_h)=\mu(\alpha,h)$ is simple and its eigenfunction is symmetric
		with respect to the $y$-axis. Similarly, if
		$\lambda(\alpha,h)<\mu(\alpha,h)$,
		then $\mu_2(Q_h)=\lambda(\alpha,h)$ is simple and its eigenfunction is
		antisymmetric. If
		$\mu(\alpha,h)=\lambda(\alpha,h),$ then the
		second eigenspace has dimension exactly two and is spanned by one
		symmetric eigenfunction and one antisymmetric eigenfunction (\cite[Corollary 5.4]{Judge2020Ann}), and the corresponding symmetric and antisymmetric eigenfunctions are
		linearly independent. 
		
		(1) If the base angle $\alpha{<} \frac{\pi}{3}$,
		the continuity of the eigenvalues together with Lemma \ref{lem.tri} implies that
		\begin{align}
			\lim_{h\to \tan\alpha} \lambda(\alpha,h)= \lambda(\alpha,\tan\alpha)
			<
			\mu(\alpha,\tan\alpha)=\lim_{h\to \tan\alpha}\mu(\alpha,h) .
		\end{align}
		By Lemma \ref{cor.a} and Lemma \ref{lem.imp}, we have that
		\begin{align}
			\lambda(\alpha,h) < \mu(\alpha,h) \text{ holds for all } h<\tan\alpha,
		\end{align}
		so the second Neumann eigenfunction of $Q_h$ is antisymmetric.
		
		If $\alpha=\pi/3$, then the limiting triangle is equilateral, and Lemma \ref{lem.tri} gives
		$\lambda(\alpha,\tan\alpha)=\mu(\alpha,\tan\alpha)$.
		Lemma \ref{cor.a} gives $\lambda(\alpha,h)<\mu(\alpha,h)$ for sufficiently small $h$.
		If a crossing occurred at some $\hat h<\tan\alpha$, Lemma \ref{lem.imp} would
		imply that $\lambda(\alpha,h)-\mu(\alpha,h)$ is positive and strictly increasing on $(\hat h,\tan\alpha)$.
		This contradicts the limit above. Hence no crossing occurs when $\alpha=\pi/3$,
		and again $\lambda(\alpha,h)<\mu(\alpha,h)$ for all $h<\tan\alpha$.

		(2) If $\alpha>\pi/3$,
		the continuity of the eigenvalues together with Lemma \ref{lem.tri} implies that
		\begin{align}
			\lim_{h\to \tan\alpha}\mu(\alpha,h)=\mu(\alpha,\tan\alpha) < \lambda(\alpha,\tan\alpha)=\lim_{h\to \tan\alpha}\lambda(\alpha,h) .
		\end{align}
		On the other hand, Lemma \ref{cor.a} implies that when $h$ is sufficiently small, 
		we have $\lambda(\alpha,h) < \mu(\alpha,h)$. Therefore
		Lemma \ref{lem.imp} implies that
		there exists a unique critical height $\hat{h}(\alpha)$, such that
		\begin{align*} 
			\lambda(\alpha,h)<\mu(\alpha,h)\text{ if }h<\hat h(\alpha),\qquad
			\mu(\alpha,h)<\lambda(\alpha,h)\text{ if }h>\hat h(\alpha), 
		\end{align*}
		and $\mu(\alpha,h)=\lambda(\alpha,h)$ holds if $h=\hat h(\alpha)$.
		Thus if $h<\hat{h}(\alpha)$, then the second Neumann eigenfunction is antisymmetric.
		If $h>\hat{h}(\alpha)$, then the second Neumann eigenfunction is symmetric.
		If $h=\hat{h}(\alpha)$, then the second Neumann eigenvalue has multiplicity two.
	\end{proof}

	\section{Symmetry on kites}\label{sec:kites}
	
	Throughout this section, let
	\[
	T=T_{a,h}=\triangle P_1P_3P_4,
	\quad P_1=(0,0),\quad P_3=(1,0),\quad P_4=(a,h),
	\]
	where $a,h>0$.
	Reflecting $T$ across the $x$-axis gives the kite domain $K_{a,h}$.
	Thus $0<a<1$ corresponds to convex kites, $a=1$ to isosceles
	triangles, and $a>1$ to concave kites.
	We write $\mathcal S(T)=\{v\in H^1(T):v=0 \text{ on } P_1P_3 \}$ and set
	\begin{equation}\label{eq:kite-spectra}
		\mu(a,h)=\inf_{\substack{v\in H^1(T)\setminus\{0\}\\\int_Tv=0}}
		\frac{\int_T|\nabla v|^2d\sigma}{\int_Tv^2d\sigma},
		\quad
		\lambda(a,h)=\inf_{v\in \mathcal S(T)\setminus\{0\}}
		\frac{\int_T|\nabla v|^2d\sigma}{\int_Tv^2d\sigma}.
	\end{equation}
	The corresponding eigenfunctions are denoted by $\varphi$ and
	$\nu$, respectively, with the normalization
	\begin{equation}\label{eq:kite-normalization}
		\int_T\varphi^2d\sigma=\int_T\nu^2d\sigma=1,\qquad
		\int_T\varphi d\sigma=0,\qquad \nu>0\quad\hbox{in }T.
	\end{equation}
	In particular, $\varphi$ satisfies Neumann conditions on all three
	edges, whereas $\nu$ satisfies a Dirichlet condition on $P_1P_3$ and Neumann
	conditions on $P_1P_4\cup P_3P_4$. We use the notation
	\begin{equation}\label{eq:kite-energies}
		E_N=\int_T\varphi_y^2d\sigma,\qquad E_D=\int_T\nu_y^2d\sigma,
		\qquad F_a(h)=\lambda(a,h)-\mu(a,h).
	\end{equation}
	The dependence of $\varphi,\nu,E_N,E_D$ on $(a,h)$ is suppressed when
	there is no ambiguity. For $a>1$, we also set
	\begin{equation}\label{eq:kite-angles}
		c=a-1,\qquad \alpha=\arctan\frac ha,\qquad
		\beta=\arctan\frac hc,\qquad
		\Theta(a,h)=\sin^2\frac{\alpha+\beta}{2}.
	\end{equation}
	Here $\alpha$ and $\beta$ are the angles of $P_1P_4$ and $P_3P_4$ to the
	positive $x$-axis.
	We retain the standard notation $J_m$ for the Bessel function of the
	first kind and $j_{m,k}$ for its $k$th positive zero.
	
	Reflection decomposes the Neumann problem on $K_{a,h}$ into the symmetry
	and antisymmetry subspaces. Their lowest positive eigenvalues are $\mu(a,h)$
	and $\lambda(a,h)$, respectively; hence
	\begin{equation}\label{eq:kite-reflection}
		\mu_2(K_{a,h})=\min\{\mu(a,h),\lambda(a,h)\}.
	\end{equation}
	The first mixed eigenvalue is simple. The second Neumann eigenvalue
	of a non-equilateral triangle is also simple
	\cite{Atar2004JAMS,Siudeja2015MZ}.

The goal of this section is to determine, for each fixed $a>0$, which of the two symmetry branches $\mu(a,h)$ and $\lambda(a,h)$ gives the second Neumann eigenvalue of the kite $K_{a,h}$ as the height $h$ varies. For $0<a<2$, comparing the cases of small and large heights gives opposite orderings for the two branches and hence the existence of at least one crossing. The main difficulty is to prove that this crossing is unique. Since both $\mu(a,h)$ and $\lambda(a,h)$ decrease with $h$, their separate monotonicity is not sufficient for this purpose. The key point is therefore strict transversality: at every crossing
$\mu(a,h)=\lambda(a,h)$,
we will prove that the two branches cross with the same strict orientation. The height-derivative formulas below reduce this question to the energy inequality
\[
E_D>E_N,
\]
or equivalently $\partial_h\lambda(a,h)<\partial_h\mu(a,h)$.
Establishing this inequality is the central task of the section. A common set of variational and energy estimates is developed first, after which the proof is divided according to the geometry of the kite: the convex range $0<a<1$ and the two concave ranges $1<a\le3/2$ and $3/2<a<2$, where different comparison arguments are required. Once strict transversality is known, the endpoint orderings force a unique transition height and hence determine the symmetry and multiplicity of the second Neumann eigenvalue. The case $a\ge2$, in which no crossing occurs, is treated separately.

For completeness, the scalar implication used here is the following. A continuous function $F$ that is differentiable at each of its zeros and satisfies $F'<0$ there has at most one zero on an interval. Indeed, immediately to the right of a zero it is negative and immediately to the left it is positive. If two zeros existed, continuity would give a first zero between a negative value after the first and a positive value before the second; at that zero the required positive sign on the left is impossible. For $F_a$, continuity follows by pulling the variational problems back to a fixed triangle. Near a simple eigenvalue the pulled-back forms depend smoothly on $h$, and the height derivatives are given below. The only equilateral half-triangle occurs at $a=1/2$, $h=\sqrt3/2$; Lemma~\ref{lem:kite-large} excludes a crossing there. Thus both branches are simple near every crossing. Together with the endpoint orderings, strict transversality gives exactly one crossing in the transition range.

{\color{black}We begin by recording the height-derivative formulas for the two symmetry branches, which reduce the transversality question to a comparison of their vertical energies.}
	\begin{Lemma}\label{lem:kite-height}
		Fix $a>0$. At every height at which $\mu(a,h)$ is simple, one has
		\begin{equation}\label{eq:kite-height-derivatives}
			\partial_h\mu(a,h)=-\frac2hE_N,\qquad
			\partial_h\lambda(a,h)=-\frac2hE_D,\qquad
			F_a'(h)=-\frac2h(E_D-E_N).
		\end{equation}
		In particular, both
		spectral branches are strictly decreasing with respect to height.
	\end{Lemma}
	\begin{proof}
		Consider the map
		\[
		\gamma_h:T_{a,1}\longrightarrow T_{a,h},\qquad
		\gamma_h(\xi,\eta)=(\xi,h\eta).
		\]
		Write
		$\varphi_h=\varphi(\,\cdot\,;a,h)$ and define its normalized pullback by
		$
		\widetilde\varphi_h(\xi,\eta)
		=h^{1/2}\varphi_h(\xi,h\eta).
		$
		The change of variables gives
		\begin{align}
			& \int_{T_{a,1}}\widetilde\varphi_h^2 d\xi d\eta
			=\int_{T_{a,h}}\varphi_h^2 dx dy=1,  \label{eq:kite-height-pullback-mass}\\
			& \mu(a,h)
			=\int_{T_{a,1}}
			\left((\partial_\xi\widetilde\varphi_h)^2
			+h^{-2}(\partial_\eta\widetilde\varphi_h)^2\right)
			d\xi d\eta.
			\label{eq:kite-height-pullback-energy}
		\end{align}
		The pulled-back weak equation is
		\begin{equation}\label{eq:kite-height-weak}
			\int_{T_{a,1}}\left(
			\partial_\xi\widetilde\varphi_h\,\partial_\xi\chi
			+h^{-2}\partial_\eta\widetilde\varphi_h\,\partial_\eta\chi\right)d\xi d\eta
			=\mu(a,h)\int_{T_{a,1}}\widetilde\varphi_h\chi d\xi d\eta,
			\quad \text{ for every }\chi\in H^1(T_{a,1}).
		\end{equation}
		Since the eigenvalue is simple, standard perturbation theory for simple
		eigenvalues allows
		$\widetilde\varphi_h$ to be chosen differentiably with respect to
		$h$. 
		Differentiating \eqref{eq:kite-height-pullback-energy} therefore
		gives
		\begin{align}\label{eq:kite-height-full-derivative}
			\partial_h\mu(a,h)
			={}&2\int_{T_{a,1}}\left(
			\partial_\xi\widetilde\varphi_h\,
			\partial_\xi\partial_h\widetilde\varphi_h
			+h^{-2}\partial_\eta\widetilde\varphi_h\,
			\partial_\eta\partial_h\widetilde\varphi_h\right)d\xi d\eta
			-2h^{-3}\int_{T_{a,1}}
			(\partial_\eta\widetilde\varphi_h)^2d\xi d\eta.
		\end{align}
		Taking $\chi=\partial_h\widetilde\varphi_h$ in
		\eqref{eq:kite-height-weak} shows that the integral in the first term of \eqref{eq:kite-height-full-derivative} equals
		\[
		\mu(a,h)\int_{T_{a,1}}
		\widetilde\varphi_h\,\partial_h\widetilde\varphi_h d\xi d\eta.
		\]
		On the other hand, differentiating
		\eqref{eq:kite-height-pullback-mass} gives
		\[
		2\int_{T_{a,1}}
		\widetilde\varphi_h\,\partial_h\widetilde\varphi_hd\xi d\eta=0.
		\]
		Thus the first term in \eqref{eq:kite-height-full-derivative} vanishes, and
		\begin{equation}\label{eq:kite-mu-height-calculation}
			\partial_h\mu(a,h)
			=-2h^{-3}\int_{T_{a,1}}
			(\partial_\eta\widetilde\varphi_h)^2d\xi d\eta
			=-\frac2h\int_{T_{a,h}}(\partial_y\varphi_h)^2dxdy.
		\end{equation}
		
		For the mixed problem, set
		$\widetilde\nu_h(\xi,\eta)=h^{1/2}\nu(\xi,h\eta;a,h)$,
		so $\partial_h\widetilde\nu_h\in\mathcal S(T_{a,1})$.
		Repeating the preceding differentiation yields
		\begin{align*}
			\partial_h\lambda(a,h)
			&=2\lambda(a,h)\int_{T_{a,1}}
			\widetilde\nu_h\,\partial_h\widetilde\nu_hd\xi d\eta
			-2h^{-3}\int_{T_{a,1}}
			(\partial_\eta\widetilde\nu_h)^2d\xi d\eta =-\frac2h\int_{T_{a,h}}(\partial_y\nu)^2dxdy,
		\end{align*}
		where the first term vanishes by the differentiated normalization.

		Finally, neither vertical energy can vanish. If $\varphi_y\equiv0$,
		the Neumann condition on the sloping edge $P_1P_4$ forces
		$\varphi_x=0$, contradicting the positive eigenvalue. If
		$\nu_y\equiv0$, its zero trace on $P_1P_3$ forces it to vanish on
		an open subset of $T$, and hence identically. This contradicts its
		normalization.
	\end{proof}

    {\color{black}We localize the possible crossing region. A simple eigenvalue comparison excludes crossings above an explicit height threshold, and rules them out entirely when \(a\ge2\).}
	\begin{Lemma}\label{lem:kite-large}
		Let $a>0$ and $h>0$. If 
		$h^2\ge\max\{0,2a-a^2,1-a^2\}$, then $\lambda(a,h)<\mu(a,h)$.
		Consequently, if $1<a<2$ and $\lambda(a,h)=\mu(a,h)$, then $h^2<a(2-a)$.
		If $a\ge2$, then $\lambda(a,h)<\mu(a,h)$ for all $h>0$.
	\end{Lemma}
	\begin{proof}
		The side lengths are $1$, $\sqrt{a^2+h^2}$ and
		$\sqrt{(a-1)^2+h^2}$. Under the stated height condition, $P_1P_3$ is
		a shortest edge. For a triangle, let $\lambda^X$ denote
		the first mixed eigenvalue with Dirichlet condition on the edge $X$.
		The eigenvalue comparison in \cite[Theorem 1.1]{Siudeja2016PAMS}
		gives
		\[
        \min\bigl\{
        \lambda^{P_1P_3},
        \lambda^{P_1P_4},
        \lambda^{P_3P_4}
        \bigr\}<\mu(a,h).
        \]
		The mixed eigenvalue ordering in
		\cite[Theorem 1.4]{Chen2026Invent} shows that a shortest edge
		minimizes $\lambda^X$ when all side lengths are distinct.
		If two sides have equal lengths, reflection symmetry gives equality
		for their corresponding mixed eigenvalues; thus the same minimum
		property holds when shortest edges are tied.
		Consequently $\lambda(a,h)=\lambda^{P_1P_3}<\mu(a,h)$, proving the result.
	\end{proof}

    {\color{black}We next derive a uniform upper bound for the vertical energy of the Neumann branch at a crossing. This estimate will be a key ingredient in the transversality arguments below.}
	\begin{Lemma}\label{lem:kite-eighth}
		At every crossing $\lambda(a,h)=\mu(a,h)=\Lambda$, a normalized
		second Neumann eigenfunction satisfies
		\begin{equation}\label{eq:kite-eighth}
			E_N\le\frac{\Lambda}{8}.
		\end{equation}
	\end{Lemma}
	\begin{proof}
		Let $T=T_{a,h}$ and consider the symmetric bilinear form
		\[
		\mathcal Q(v,w)=\int_T\nabla v\cdot\nabla w\,d\sigma
		-\Lambda\int_Tvw\,d\sigma.
		\]
		Since $\Lambda$ is the first mixed eigenvalue,
		\[
		\mathcal Q(v,v)\ge0\qquad\text{for every }v\in\mathcal S(T).
		\]
		We shall apply this fact to
		\[
		f=y\varphi,\qquad g=\partial_y\varphi.
		\]
		\textcolor{black}{Remark~\ref{rem:corner-regularity}, applied at the
		Neumann corners of $T$, gives $\varphi\in H^2(T)$ and
		$\nabla\varphi\to0$ at every vertex. Thus $f,g\in H^1(T)$.
		Both vanish on $P_1P_3$, by the factor $y$ in $f$ and the Neumann
		condition for $g$, so $f,g\in\mathcal S(T)$.}
		We first compute $\mathcal Q(f,f)$. The eigenvalue equation implies
		\begin{align}\label{eq:kite-f-boundary}
			(\Delta+\Lambda)f=2\partial_y\varphi=2g,
			\qquad \partial_n f=n_y\varphi\quad\text{on every edge},
		\end{align}
		where $n_y$ denotes the $y$-component of $n$.
		Integration by parts and the divergence theorem give
		\begin{align}\label{eq:kite-Qff}
			\mathcal Q(f,f)
			&=-\int_T f(\Delta f+\Lambda f)d\sigma
			+\int_{\partial T}f\partial_n f\,ds\notag\\*
			&=-2\int_Ty\varphi\varphi_yd\sigma
			+\int_{\partial T}yn_y\varphi^2\,ds\notag\\*
			&=-\int_Ty\partial_y(\varphi^2)d\sigma
			+\int_T\partial_y(y\varphi^2)d\sigma
			=\int_T\varphi^2d\sigma=1.
		\end{align}
		
		Next we compute $\mathcal Q(g,g)$.
		Let $\tau$ denote a unit tangent vector along an edge, so that $\{n, \tau\}$ forms an orthonormal basis.
		Then $\partial_y=\tau_y\partial_{\tau}+n_y\partial_{n}$ holds on any edge,
		where $\tau_y$ denotes the $y$-component of $\tau$.
		On every edge, the Neumann condition gives
		\begin{equation}\label{eq:kite-g-boundary}
			g=\tau_y\partial_\tau\varphi,\qquad
			\partial_{\tau n}\varphi=0.
		\end{equation}
		Using $\partial_{nn}\varphi+
		\partial_{\tau\tau}\varphi=-\Lambda\varphi$, we obtain
		\begin{equation}\label{eq:kite-ng-boundary}
			\partial_n g
			=\tau_y\partial_{n\tau}\varphi+n_y\partial_{nn}\varphi
			=-n_y(\Lambda\varphi+\partial_{\tau\tau}\varphi).
		\end{equation}
		Also $(\Delta+\Lambda)g=0$ in $T$. Since $n_y$ and $\tau_y$ are
		constant along an edge, integration along each edge gives
		\begin{equation}\label{eq:kite-Qgg}
			\begin{aligned} 
				\mathcal Q(g,g) &= \int_{\partial T}g \partial_{n}g ds  
				=- \int_{\partial T}  \tau_y n_y\partial_{\tau}\varphi \left(\Lambda\varphi+\partial_{\tau\tau}\varphi\right)ds  \\
				&= -   \Lambda \int_{\partial T} \tau_y n_y\varphi \partial_{\tau}\varphi ds -
				\int_{\partial T} \frac{1}{2}\tau_y n_y  \partial_{\tau} (\partial_{\tau}\varphi)^2 ds\\
				&=-  \Lambda \int_{\partial T} \tau_y n_y \varphi \partial_{\tau}\varphi ds -
				\sum_{3\text{ edges } P_iP_j}
				\frac{1}{2}\tau_y n_y (\partial_\tau \varphi)^2 \Big|_{P_i}^{P_j} \\
				&= -  \Lambda \int_{\partial T} \tau_y n_y  \varphi\partial_{\tau}\varphi ds
				=-  \Lambda \int_{\partial T}n_y\varphi g ds.
			\end{aligned}		
		\end{equation}
		\textcolor{black}{The endpoint terms in the preceding sum vanish by the
		Neumann gradient estimate \eqref{eq:corner-gradient-bounds}.}

		Finally, by using \eqref{eq:kite-f-boundary}, \eqref{eq:kite-Qgg}, and the divergence theorem,  we
		obtain that
		\begin{align}\label{eq:kite-Qfg}
			\mathcal Q(f,g)
			&=-\int_Tg(\Delta f+\Lambda f)d\sigma
			+\int_{\partial T}g\partial_n f\,ds\notag\\*
			&=-2\int_Tg^2d\sigma+\int_{\partial T}gn_y\varphi\,ds
			=-2E_N-\frac{\mathcal Q(g,g)}{\Lambda}.
		\end{align}
		Cauchy--Schwarz for the
		nonnegative form $\mathcal Q$ and \eqref{eq:kite-Qff} give
		\[
		\left(2E_N+\frac{\mathcal Q(g,g)}{\Lambda}\right)^2
		=|\mathcal Q(f,g)|^2
		\le\mathcal Q(f,f)\mathcal Q(g,g)=\mathcal Q(g,g).
		\]
		Consequently,
		\[
		2E_N\le\sqrt{\mathcal Q(g,g)}-\frac{\mathcal Q(g,g)}{\Lambda}
		=\frac\Lambda4
		-\frac1\Lambda\left(\sqrt{\mathcal Q(g,g)}-\frac\Lambda2\right)^2
		\le\frac\Lambda4.
		\]
		This proves \eqref{eq:kite-eighth}.
	\end{proof}
	
	\subsection[Convex kites: \texorpdfstring{$0<a<1$}{0 < a < 1}]{Convex kites: $0<a<1$}\label{subsec:kite-convex}

We begin with rhombuses, which correspond to $a=1/2$ in the family $K_{a,h}$.
Related symmetry results are already known: Kennedy \cite{Kennedy2018AM}
established the central antisymmetry of second Neumann eigenfunctions on
simply connected planar centrally symmetric domains, while Siudeja
\cite{Siudeja2016PAMS} studied symmetry properties of eigenfunctions on
rhombuses through mixed Dirichlet-Neumann eigenvalue comparisons.
For completeness, we record the symmetry and simplicity result for
non-square rhombuses that will be used below.

\begin{Proposition}\label{rhombus}
    Let $Q$ be the non-square rhombus $P_1P_2P_3P_4$ with
    $|P_1P_3|>|P_2P_4|$.
    Then its second Neumann eigenvalue is simple, and the corresponding
    eigenfunction $u$ is symmetric about $P_1P_3$,
    is antisymmetric about $P_2P_4$ and $P_2P_4$ is the nodal line of $u$.
\end{Proposition}

\begin{proof}
    Choose orthonormal coordinates $(X,Y)$ centered at the intersection
    of the diagonals, with the $X$-axis along $P_1P_3$.
    For some $A>B>0$, the rhombus is
    $|X|/A+|Y|/B<1$.
    Its upper and lower boundary graphs have Lipschitz constant
    $B/A<1$, so Lemma~\ref{lem.lip} gives simplicity.
    Each reflection therefore acts as multiplication by $1$ or $-1$
    on the one-dimensional second eigenspace.
    After changing the sign of $u$, the same lemma gives
    $u_X+u_Y>0$ and $u_X-u_Y>0$, and hence $u_X>0$ in the rhombus.
    Symmetry under $X\mapsto-X$ would imply $u_X=0$ on $X=0$,
    so $u$ is antisymmetric about $P_2P_4$.
    Antisymmetry under $Y\mapsto-Y$ would imply $u=0$,
    and thus $u_X=0$, along $Y=0$; hence $u$ is symmetric about $P_1P_3$.
    Finally, $u(0,Y)=0$ and $u_X>0$ show that its nodal set is exactly $X=0$.
\end{proof}

For general convex kites, we next derive simple bounds for the mixed
vertical energy and the Neumann eigenvalue.
Together with Lemma~\ref{lem:kite-eighth}, these estimates will yield
transversality at a crossing.

	\begin{Lemma}\label{lem:kite-mixed-vertical-lower}
		If $0<a<1$, then the normalized first mixed eigenfunction satisfies
		\begin{equation}\label{eq:kite-convex-vertical-lower}
			E_D=\int_T\nu_y^2\,dx\,dy\ge\frac{\pi^2}{4h^2}.
		\end{equation}
	\end{Lemma}
	\begin{proof}
		For any fixed $x\in(0,1)$, the maximum height is $y_{\max}(x)\le h$.
		Since $\nu(x,0)=0$,
		applying the Poincaré inequality on the interval $[0,y_{\max}(x)]$,
		we have
		\begin{align*}
			\int_0^{y_{\max}(x)}\nu_y(x,y)^2 dy
			\ge\frac{\pi^2}{4y_{\max}(x)^2}
			\int_0^{y_{\max}(x)}\nu(x,y)^2 dy
			\ge\frac{\pi^2}{4h^2}
			\int_0^{y_{\max}(x)}\nu(x,y)^2 dy.
		\end{align*}
		Integrating in $x$ and using $\int_T\nu^2=1$ proves the lemma.
	\end{proof}
	
	\begin{Lemma}\label{lem:kite-neumann-height-upper}
		For every $a,h>0$, one has
		\begin{equation}\label{eq:kite-neumann-height-upper}
			\mu(a,h)\le\frac{18}{h^2}.
		\end{equation}
	\end{Lemma}
	\begin{proof}
		Take the affine function $v(x,y)=y-h/3$. 
		A direct calculation gives $\int_{T}v =0$,
		and hence $v$ is an admissible test function.
		Moreover, a direct integration gives
		\[
		\int_T|\nabla v|^2d\sigma=\frac h2,
		\qquad
		\int_Tv^2d\sigma=\frac{h^3}{36}.
		\]
		The Neumann Rayleigh principle therefore gives
		$\mu(a,h)\le(h/2)/(h^3/36)=18/h^2$.
	\end{proof}

	We now establish transversality at the crossing of $\lambda(a,h)$ and $\mu(a,h)$ when $0<a<1$.
	\begin{Lemma}\label{lem:kite-convex-transverse}
		Let $0<a<1$. At every crossing at which $\mu(a,h)$ is simple,
		\begin{equation}\label{eq:kite-convex-transverse}
			E_D>E_N,\qquad \partial_h\lambda(a,h)<\partial_h\mu(a,h).
		\end{equation}
	\end{Lemma}
	\begin{proof}
		Suppose that $\lambda(a,h)=\mu(a,h)=\Lambda$.
		Combining Lemma~\ref{lem:kite-eighth} with
		Lemmas~\ref{lem:kite-mixed-vertical-lower} and
		\ref{lem:kite-neumann-height-upper}, we obtain
		\[
		E_N\le\frac\Lambda8\le\frac9{4h^2}
		<\frac{\pi^2}{4h^2}\le E_D.
		\]
		The height-derivative formulas in Lemma~\ref{lem:kite-height} now give
		\[
		\partial_h\lambda(a,h)=-\frac2hE_D
		<-\frac2hE_N=\partial_h\mu(a,h),
		\]
		as required.
	\end{proof}

	\subsection[Concave kites: \texorpdfstring{$1<a\le3/2$}{1 < a <= 3/2}]{Concave kites: $1<a\le3/2$}\label{subsec:kite-robin}
	
	We begin with the two thin-domain limits for $a>1$ and compare
	those limits on the full range $a>1$. We then establish a mixed
	energy estimate. These preparatory results will also be used in
	Subsection~\ref{subsec:kite-residual}. The transversality argument in
	the present range then follows from a Robin comparison on two
	subtriangles.
	
	For the thin-domain limits it is useful to work in the natural
	weighted energy spaces. If $g$ is positive in the interior of an
	interval $J$, write
	\[
	H^1_g(J)=\left\{f\in H^1_{\rm loc}(J):
	\int_Jg(x)\bigl(f(x)^2+f'(x)^2\bigr)\,dx<\infty\right\}.
	\]
	{\color{black}The next two lemmas identify the thin-domain limits of the mixed and Neumann eigenvalues, respectively.
    Their limiting values will then be compared to determine the ordering of the two branches for sufficiently small height.}
	
	\begin{Lemma}\label{lem:kite-mixed-thin-limit}
		For every fixed $a>1$,
		\begin{equation}\label{eq:kite-mixed-thin-limit}
			\lim_{h\to 0}\lambda(a,h)=\lambda_0(a)
			:=\frac{j_{0,1}^2}{(a-1)^2}.
		\end{equation}
	\end{Lemma}
	\begin{proof}
		Set $c=a-1$ and
		\[
		g(x)=\frac{a-x}{ac}\quad(1<x<a),\qquad
		V=\{f\in H^1_{g}(1,a):f(1)=0\}.
		\]
		Define the one-dimensional comparison eigenvalue by
		\begin{equation}\label{eq:kite-mixed-limit-variational}
			\lambda_0(a)=\inf_{f\in V\setminus\{0\}}
			\frac{\displaystyle\int_1^ag(x)f'(x)^2\,dx}
			{\displaystyle\int_1^ag(x)f(x)^2\,dx}.
		\end{equation}
		We first prove convergence to this variational value and then
		identify it explicitly.
		
		The vertical rescaling $(x,y)=(\xi,h\eta)$ sends the problem to the
		fixed triangle
		\[
		T_{a,1}=\{(\xi,\eta):0<\xi<a,\
		V_{\rm bot}(\xi)<\eta<V_{\rm top}(\xi)\},
		\]
		where
		\[
		V_{\rm top}(\xi)=\frac\xi a,\qquad
		V_{\rm bot}(\xi)=
		\begin{cases}
			0,&0<\xi<1,\\
			(\xi-1)/c,&1<\xi<a.
		\end{cases}
		\]
		Thus its vertical thickness on $1<\xi<a$ is $g(\xi)$.
		The pulled-back Rayleigh quotient is
		\begin{equation}\label{eq:kite-thin-Rayleigh}
			\mathcal R_h(U)=
			\frac{\displaystyle\int_{T_{a,1}}
				(U_\xi^2+h^{-2}U_\eta^2)\,d\xi\,d\eta}
			{\displaystyle\int_{T_{a,1}}U^2\,d\xi\,d\eta},
		\end{equation}
		with zero trace on the unchanged edge $P_1P_3$ in the mixed case.
		
		Let $f_0$ be a minimizer in
		\eqref{eq:kite-mixed-limit-variational}, normalized by
		$\displaystyle\int_1^agf_0^2d\xi=1$. Extend it by zero to $0<\xi<1$ and lift it
		to $T_{a,1}$ as $U_0(\xi,\eta)=f_0(\xi)$.
		$f_0(1)=0$ ensures that this lift belongs to
		$\mathcal S(T_{a,1})$. Its vertical derivative vanishes, so
		\begin{equation}\label{eq:kite-mixed-thin-upper}
			\lambda(a,h)\le\mathcal R_h(U_0)
			=\frac{\displaystyle\int_1^a g f_0'^2d\xi}
			{\displaystyle\int_1^a g f_0^2d\xi}
			=\lambda_0(a).
		\end{equation}

		For the reverse inequality, let
		\[
		\widetilde\nu_h(\xi,\eta)
		=h^{1/2}\nu(\xi,h\eta;a,h).
		\]
		It has unit $L^2(T_{a,1})$ norm and zero trace on $P_1P_3$.
		By \eqref{eq:kite-mixed-thin-upper},
		\begin{equation}\label{eq:kite-mixed-thin-bound}
			\int_{T_{a,1}}\left(
			|\partial_\xi\widetilde\nu_h|^2
			+h^{-2}|\partial_\eta\widetilde\nu_h|^2\right)d\xi d\eta=\lambda(a,h)\le\lambda_0(a).
		\end{equation}
		Given any sequence of heights tending to zero, we may pass to a
		subsequence such that
		\[
		\widetilde\nu_h\rightharpoonup U
		\quad\text{in }H^1(T_{a,1}),\qquad
		\widetilde\nu_h\longrightarrow U
		\quad\text{in }L^2(T_{a,1}).
		\]
		Moreover,
		$\|\partial_\eta\widetilde\nu_h\|_2^2\le h^2\lambda_0(a)$,
		so $U_\eta=0$. We can therefore write $U(\xi,\eta)=f(\xi)$.
		The zero trace on $P_1P_3$ passes to the weak limit and implies
		$f=0$ on $(0,1)$. The vertical thickness of $T_{a,1}$ is positive
		near $\xi=1$, so $f$ is an ordinary $H^1$ function in a neighborhood
		of $1$ and is continuous there. In particular, its restriction to
		$(1,a)$ satisfies $f(1)=0$ and belongs to $V$.
		Strong $L^2$ convergence preserves the mass:
		\[
		\int_1^a g f^2d\xi=\int_{T_{a,1}}U^2d\xi d\eta=1.
		\]
		The variational principle and weak lower semicontinuity now yield
		\begin{align}\label{eq:kite-mixed-thin-lower}
			\lambda_0(a)
			&\le\int_1^ag f'^2d\xi 
			=\int_{T_{a,1}}|U_\xi|^2 d\xi d\eta\\
			&\le\liminf_{h\to 0}
			\int_{T_{a,1}}|\partial_\xi\widetilde\nu_h|^2 d\xi d\eta
			\le\liminf_{h\to 0}\lambda(a,h).
		\end{align}
		Applying this argument to a sequence realizing the lower limit and
		combining it with \eqref{eq:kite-mixed-thin-upper} proves
		$\lambda(a,h)\to\lambda_0(a)$.
		
		It remains to identify $\lambda_0(a)$. The Euler--Lagrange equation for \eqref{eq:kite-mixed-limit-variational} is
		\[
		-(gf')'=\lambda_0gf\quad(1<\xi<a),\quad f(1)=0.
		\]
		With $r=a-\xi$ and $W(r)=f(a-r)$, this becomes
		\[
		-(rW')'=\lambda_0rW\quad(0<r<c).
		\]
		The solutions are linear combinations of
		$J_0(\sqrt{\lambda_0}\,r)$ and $Y_0(\sqrt{\lambda_0}\,r)$.
		The $Y_0$ solution has infinite weighted energy at $r=0$, since its
		derivative behaves as a nonzero multiple of $r^{-1}$ and
		$\int_0^\varepsilon r\cdot r^{-2}\,dr=\infty$.
		Thus the positive first eigenfunction is a multiple of
		$J_0(\sqrt{\lambda_0}(a-\xi))$. Its zero at $\xi=1$ gives
		$\sqrt{\lambda_0}\,c=j_{0,1}$, proving the formula.
	\end{proof}
	
	\begin{Lemma}\label{lem:kite-neumann-thin-limit}
		For every fixed $a>1$,
		\begin{equation}\label{eq:kite-neumann-thin-limit}
			\lim_{h\to 0}\mu(a,h)=\mu_0(a),
		\end{equation}
		where $\mu_0(a)$ is the first strictly positive solution of
		\begin{equation}\label{eq:kite-thin-matching}
			J_0(\sqrt\mu)J_1(c\sqrt\mu)
			+J_1(\sqrt\mu)J_0(c\sqrt\mu)=0,
			\qquad c=a-1.
		\end{equation}
	\end{Lemma}
	\begin{proof}
		The argument is parallel to that of Lemma \ref{lem:kite-mixed-thin-limit}, but the mean-zero condition must now be retained. 
		The full vertical thickness of $T_{a,1}$ is
		\[
		g(\xi)=
		\begin{cases}
			\xi/a,&0<\xi<1,\\
			(a-\xi)/(ac),&1<\xi<a.
		\end{cases}
		\]
		Its natural weighted Neumann form has first positive eigenvalue
		\begin{equation}\label{eq:kite-neumann-limit-variational}
			\mu_0(a)=
			\inf_{\substack{f\in H^1_{g}(0,a)\setminus\{0\}\\
					\int_0^agf=0}}
			\frac{\displaystyle\int_0^agf'^2d\xi}
			{\displaystyle\int_0^agf^2d\xi}.
		\end{equation}
		If $f_0$ is a
		normalized minimizer, its lift $U_0(\xi,\eta)=f_0(\xi)$ has zero
		mean because
		\[
		\int_{T_{a,1}}U_0\,d\xi\,d\eta
		=\int_0^ag(\xi)f_0(\xi)\,d\xi=0.
		\]
		It is therefore admissible in the second Neumann Rayleigh quotient,
		and \eqref{eq:kite-thin-Rayleigh} gives
		\begin{equation}\label{eq:kite-neumann-thin-upper}
			\mu(a,h)\le
			\frac{\displaystyle\int_0^ag f_0'^2d\xi}
			{\displaystyle\int_0^ag f_0^2d\xi}
			=\mu_0(a).
		\end{equation}
		
		{\color{black}For the reverse inequality, we pull back the normalized second Neumann eigenfunction} 
		$
		\widetilde\varphi_h(\xi,\eta)
		=h^{1/2}\varphi(\xi,h\eta;a,h).
		$
		Then
		\[
		\int_{T_{a,1}}\widetilde\varphi_h^2d\xi d\eta=1,\qquad
		\int_{T_{a,1}}\widetilde\varphi_h d\xi d\eta=0,
		\]
		and
		\[
		\int_{T_{a,1}}\left(
		|\partial_\xi\widetilde\varphi_h|^2
		+h^{-2}|\partial_\eta\widetilde\varphi_h|^2\right)d\xi d\eta
		=\mu(a,h)\le\mu_0(a).
		\]
		Along a subsequence we thus have weak $H^1$ and strong $L^2$
		convergence to a limit $U$, and
		$\partial_\eta\widetilde\varphi_h\to0$ in $L^2$ implies
		$U(\xi,\eta)=f(\xi)$. Passing to the limit in both the
		normalization and the mean-zero condition gives
		\[
		\int_0^ag f^2d\xi=1,\qquad \int_0^ag fd\xi=0.
		\]
		Hence $f$ is admissible in
		\eqref{eq:kite-neumann-limit-variational}. As in the mixed case,
		\begin{align}\label{eq:kite-neumann-thin-lower}
			\mu_0(a)
			&\le\int_0^ag f'^2d\xi
			=\int_{T_{a,1}}|U_\xi|^2d\xi d\eta\\
			&\le\liminf_{h\to 0}
			\int_{T_{a,1}}|\partial_\xi\widetilde\varphi_h|^2d\xi d\eta
			\le\liminf_{h\to 0}\mu(a,h).
		\end{align}
		Together with \eqref{eq:kite-neumann-thin-upper}, this proves the
		claimed limit.
		
		To obtain its characteristic equation, solve
		$-(g f')'=\mu_0g f$ on the two sides of $\xi=1$.
		Finite weighted energy excludes the singular $Y_0$ solution at each
		outer endpoint, so
		\[
		f(\xi)=
		\begin{cases}
			A J_0(\sqrt{\mu_0}\,\xi),&0<\xi<1,\\
			BJ_0(\sqrt{\mu_0}(a-\xi)),&1<\xi<a.
		\end{cases}
		\]
		Since $g$ is continuous and positive at $1$, membership in the
		form domain gives continuity of $f$, and the weak equation gives
		continuity of the flux $g f'$. As $g (1^-)=g (1^+)=1/a$,
		the matching conditions reduce to
		\[
		\begin{cases}
			A J_0(\sqrt{\mu_0})=BJ_0(c\sqrt{\mu_0}),\\
			-A J_1(\sqrt{\mu_0})=BJ_1(c\sqrt{\mu_0}).
		\end{cases}
		\]
		A nontrivial pair $(A,B)$ exists if and only if
		\begin{align}
			J_0(\sqrt{\mu })J_1(\sqrt{\mu }c)+J_1(\sqrt{\mu })J_0(\sqrt{\mu }c)=0.
		\end{align}
		Hence the first positive weighted eigenvalue is the first
		strictly positive solution of this equation.
	\end{proof}
	
	\begin{Lemma}\label{lem:kite-limit-comparison}
		For every $a>1$,
		let $\lambda_0(a)$ and $\mu_0(a)$ be the limit values
		given in Lemma \ref{lem:kite-mixed-thin-limit} and Lemma \ref{lem:kite-neumann-thin-limit}, respectively.
		Then $\mu_0(a)<\lambda_0(a)$ if and only if $a<2$.
	\end{Lemma}
	\begin{proof} 
		We first rewrite the equation in a more convenient form. Let 
		$x = \sqrt{\mu}, c = a - 1 > 0.$
		Then the equation \eqref{eq:kite-thin-matching} becomes
		\begin{align} \label{eq.500}
			J_0(x)J_1(cx) + J_1(x)J_0(cx) = 0,
		\end{align}
		and we are looking for its first positive root $x_0 = \sqrt{\mu_0(a)}$.
		The quantity $\lambda_0(a)$ is given by
		$\lambda_0(a) = \frac{j_{0,1}^2}{(a-1)^2} = \frac{j_{0,1}^2}{c^2}$, Hence
		\begin{align}
			\mu_0(a) < \lambda_0(a) \text{ is equivalent to } c x_0 < j_{0,1}.
		\end{align}

		Dividing  the equation \eqref{eq.500} by $J_1(x)J_1(cx)$ yields
		\begin{align} \label{eq.600}
			\frac{J_0(x)}{J_1(x)} + \frac{J_0(cx)}{J_1(cx)} = 0.
		\end{align}
		We define the function $R(t) = \frac{J_0(t)}{J_1(t)}$ for $t \in (0, j_{1,1})$,
		where $j_{1,1}$ is the first positive zero of $J_1$.
		Direct differentiation shows that
		\[
		R'(t)<-\frac{J_0^2(t)+J_1^2(t)}{2J_1^2(t)}<0,
		\]
		so that $R(t)$ is strictly decreasing from $+\infty$ (as $t\to 0^+$)
		to $-\infty$ (as $t\to j_{1,1}^-$). Therefore, 
		it vanishes exactly at the first zero of $J_0$, i.e. $R(j_{0,1}) = 0$,
		where $0 < j_{0,1} < j_{1,1}$.
		Then \eqref{eq.600} becomes 
		\begin{align} \label{eq.700}
			R(x) + R(cx) = 0.
		\end{align}
        {\color{black}If $c=1$, clearly we have $x_0=j_{0,1}$.
        If $c\neq 1$, then $x_0$ and $cx_0$ lie separately in the two intervals $(0,j_{0,1})$ and $(j_{0,1},j_{1,1})$.
        Therefore, if $c>1$, then $cx_0\in (j_{0,1},j_{1,1})$, so $cx_0>j_{0,1}$.
        If $c<1$, then $cx_0\in (0,j_{0,1})$, so $cx_0<j_{0,1}$.}
		This completes the proof.
	\end{proof}

	By virtue of the two limits in Lemma \ref{lem:kite-mixed-thin-limit} and Lemma \ref{lem:kite-neumann-thin-limit},
	together with their strict ordering stated in Lemma \ref{lem:kite-limit-comparison},
	we have the following results.
	
	\begin{Lemma}\label{lem:kite-thin-limit}
		For each fixed $1<a<2$, there is $h_->0$ such that
		$\mu(a,h)<\lambda(a,h)$ for all $h<h_-$.
	\end{Lemma}

    {\color{black}Two auxiliary properties of the mixed eigenfunction will be useful below.
    They concern its monotonicity and certain flux identities needed for the energy estimates.}
% KEEP: this monotonicity lemma, its corner estimates and boundary critical-point
% exclusion are needed to establish derivative signs for later energy estimates.
	\begin{Lemma}\label{lem:kite-mixed-monotonicity}
		For every fixed $a>1$, we have $\nu_x>0$ and $\nu_y>0$ in $T$. Moreover, 
		$\nu_y(x,0)^2$ is strictly increasing and integrable on $(0,1)$.
	\end{Lemma}
	\begin{proof}\color{black}
		The Neumann vertex $P_4$ has angle less than $\pi/2$.
		By \cite[Proposition 7.2]{Hatcher2024PAMS}, $P_4$ is the unique global maximum of $\nu$,
		$\nu$ has no critical points on either open
		Neumann edge, and $\nu_y>0$ in $T$. 
		In particular, $\nu_x>0$ holds on the two open Neumann edges,
		whereas on $P_1P_3$, the Dirichlet condition gives $\nu_x=0$.

        Let $\lambda_{1}$ denote the first eigenvalue with Dirichlet
		conditions on all three edges. A simple variational principle gives
		$\lambda\le\lambda_{1}$. Further, the inequality is strict. In fact, if equality held,
		the first positive Dirichlet eigenfunction would also minimize the
		mixed Rayleigh quotient, so it would also satisfy the Neumann
		condition on $P_1P_4$ and $P_3P_4$, contrary to the Hopf lemma.
        Therefore, the maximum principle holds for the operator $\Delta+\lambda$.

       To safely apply the maximum principle,
		we first analyze the behavior of $\nu$ near the vertices.
		\textcolor{black}{Remark~\ref{rem:corner-regularity} gives $\nabla\nu\to0$
		at the acute mixed corner $P_1$ and the Neumann corner $P_4$.
		The remaining vertex $P_3$ has mixed boundary conditions and
		obtuse angle $\pi-\beta$. Applying
		Lemma~\ref{lem:corner-bessel-expansions}(ii), with the polar angle
		{\color{black}$\vartheta=0$ on $P_3P_4$, so that
		$x=1+r\cos(\vartheta+\beta)$, $y=r\sin(\vartheta+\beta)$ and
		$0\le\vartheta\le\pi-\beta$, gives}}
		\begin{equation}\label{eq:kite-mixed-corner}
			\nu=A r^\delta{\color{black}\cos(\delta\vartheta)}
			+{\color{black}\mathcal E},
			\qquad
			\delta=\frac{\pi}{2(\pi-\beta)}\in(1/2,1),\quad A>0,
		\end{equation}
		\textcolor{black}{where \eqref{eq:corner-remainder-derivatives} gives}
		\[
		|\nabla{\color{black}\mathcal E}|
		=O\bigl(r^{\min\{\delta+1,3\delta-1\}}\bigr).
		\]
		Hence, the leading term of $\nu_x$ is
		$A\delta r^{\delta-1}{\color{black}\cos\bigl((1-\delta)\vartheta+\beta\bigr)}$.
		Consequently, $\nu_x$ may blow up near $P_3$.

        We now introduce the perturbation which removes the difficulty at the obtuse corner.
        Since $\lambda<\lambda_1$, the Dirichlet problem
        \begin{equation}\left\{
            \begin{aligned}
                &(\Delta+\lambda)P=0, &&\text{ in } T,\\
                &P=1,    &&\text{ on }\partial T.
            \end{aligned}\right.
        \end{equation}
        has a unique solution. Moreover, $P$ is positive in $T$,
        is continuous up to the boundary, and satisfies $P(P_3)=1$.
        Set $w=\nu_x+\varepsilon P$, then the corner expansion implies
        that $w>0$ in a neighborhood of $P_3$.
        Clearly, $w$ satisfies $(\Delta + \lambda)w = 0$.
		We may therefore apply the maximum principle on $T$ with a small disk removed around vertex $P_3$,
		and then pass to the limit to obtain $\nu_x \ge 0$ in $T$.
		Moreover, the strong maximum principle gives strict positivity.
		
		Finally, on the open edge $P_1P_3$, both $\nu$ and $\nu_x$ vanish,
		whereas both are positive in $T$. The Hopf lemma applied 
		to $\nu$ and  $\nu_x$, therefore yields
		\[
		\nu_y(x,0)>0,\quad \nu_{xy}(x,0)>0,
		\quad 0<x<1.
		\]
		It follows that
		\[
		\frac{d}{dx}\bigl(\nu_y(x,0)^2\bigr)
		=2\nu_y(x,0)\nu_{xy}(x,0)>0.
		\]
		\textcolor{black}{Finally, differentiating \eqref{eq:kite-mixed-corner}
		with \eqref{eq:corner-remainder-derivatives} gives
		$\nu_y(x,0)=O((1-x)^{\delta-1})$ as $x\to1$,}
		hence, $\nu_y(x,0)^2$ is integrable on $(0,1)$. 
	\end{proof}
	
	\begin{Lemma}\label{lem:kite-mixed-flux}
		For every fixed $a>1$, let $X=\displaystyle\int_T\nu_x\nu_ydxdy$. Then we have
		\begin{align}
			\frac h2\int_0^1\nu_y(x,0)^2\,dx&=\lambda,
			\label{eq:kite-flux-mass}\\
			\frac h2\int_0^1x\nu_y(x,0)^2\,dx
			&=aE_D-c(\lambda-E_D)+\frac{ac-h^2}{h}X.
			\label{eq:kite-flux-moment}
		\end{align}
	\end{Lemma}
	\begin{proof}
		For an affine vector field $V$, the eigenvalue equation gives the
		Rellich identity
		\begin{equation}
			\begin{aligned}\label{eq:kite-rellich}
				&\operatorname{div}\left[(V\cdot\nabla\nu)\nabla\nu
				-\frac12|\nabla\nu|^2V+\frac\lambda2\nu^2V\right] \\
				&\hspace{10mm}=\nabla\nu\cdot DV\nabla\nu
				-\frac12\operatorname{div}V
				\bigl(|\nabla\nu|^2-\lambda\nu^2\bigr).
			\end{aligned}
		\end{equation}
		Choose
		\[
		V_1=\left(x-1-\frac ch y\right)(a,h),\qquad
		V_2=\left(-x+\frac ah y\right)(c,h).
		\]
		The field $V_1$ is tangent to $P_1P_4$ and vanishes on $P_3P_4$,
		$V_2$ vanishes on $P_1P_4$ and is tangent to $P_3P_4$.
		Integrate \eqref{eq:kite-rellich} on $T$ after deleting a disk of radius $\varepsilon$ about $P_3$,
		and apply the divergence theorem.
		\textcolor{black}{The gradient bound \eqref{eq:corner-gradient-bounds}
		at the mixed corner $P_3$ gives a circular boundary contribution
		of $O(\varepsilon^{2\delta-1})$, which tends to zero since $\delta>1/2$.}
		The Neumann condition shows that two Neumann edges contribute zero.
		Last, on the Dirichlet edge, $\nu=0$ and $\nabla\nu=(0,\nu_y(x,0))$.
		Thus the left side of \eqref{eq:kite-rellich} is
		\[
		\frac12\int_0^1 (V_i\cdot n)\nu_y(x,0)^2  dx.
		\]

		For $V_2$, $V_2\cdot n=hx$ on $P_1P_3$, and the derivative matrix is
		\[
		DV_2=\begin{pmatrix}-c&ac/h\\-h&a\end{pmatrix}.
		\]
		Since $\operatorname{div} V_2=1$, the divergence term on the right of \eqref{eq:kite-rellich} integrates to zero,
		thus
		\begin{align*}
			\frac h2\int_0^1x\nu_y(x,0)^2\,dx
			&=\int_T\left(-c\nu_x^2+a\nu_y^2
			+\left(\frac{ac}{h}-h\right)\nu_x\nu_y\right)dxdy\\
			&=-c(\lambda-E_D)+aE_D+\frac{ac-h^2}{h}X.
		\end{align*}
		Finally,
		$V_1+V_2=(x-a,y-h)$ has derivative matrix equal to the identity.
		Adding the two Rellich formulas gives
		\[
		\frac h2\int_0^1\nu_y(x,0)^2\,dx
		=\int_T|\nabla\nu|^2dxdy=\lambda.
		\]
		This proves the lemma.
	\end{proof}

	{\color{black}The preceding monotonicity and flux identities yield a quantitative lower bound for the mixed vertical energy. This also restricts the possible location of a crossing with \(E_D\le E_N\).}
	\begin{Proposition}\label{prop:kite-mixed-energy}
		For every fixed $a>1$, we have
		\begin{equation}\label{eq:kite-mixed-theta}
			\frac{E_D}{\lambda}> \left\{
			\begin{aligned}
				&\Theta(a,h), && \text{ if } h^2< ac, \\
				&\frac{1}{2}, && \text{ if } h^2\ge ac.
			\end{aligned}
			\right.
		\end{equation}
		where $c=a-1$ and $\Theta(a,h)=\sin^2\bigl((\alpha+\beta)/2\bigr)$,
        as defined in \eqref{eq:kite-angles}.In particular, any crossing with $E_D\le E_N$ must satisfy $\Theta(a,h)<\frac18$ and  $h<h_b(a)$, where
		\begin{equation}\label{eq:kite-failing-window}
			h_b(a)=\frac{-3(2a-1)+\sqrt{9(2a-1)^2+28a(a-1)}}{2\sqrt7}.
		\end{equation}
	\end{Proposition}
	\begin{proof}
		Strict monotonicity of $\nu_y(x,0)$ in Lemma \ref{lem:kite-mixed-monotonicity} gives a strict inequality
		\begin{align*}
			\int_0^1\left(x-\frac12\right) \nu_y(x,0)^2dx
			&=\int_{1/2}^1\left(x-\frac12\right)
			\bigl(\nu_y(x,0)^2- \nu_y(1-x,0)^2 \bigr) dx
			>0.
		\end{align*}
		Subtracting one half of \eqref{eq:kite-flux-mass} from
		\eqref{eq:kite-flux-moment}, and writing
		\[
		e=\frac{E_D}{\lambda},\qquad
		\kappa=\frac{ac-h^2}{h},\qquad S=2a-1,
		\]
		we obtain
		\begin{equation}\label{eq:kite-mixed-quadratic}
			\kappa X > S(\lambda/2 - E_D).
		\end{equation}
		Meanwhile, by the Cauchy--Schwarz inequality and the identity $\displaystyle\int_T \nu_x^2dxdy = \lambda - E_D$, we have
		\begin{align} \label{eq:kite-X-cauchy}
			0 < X \le \lambda\sqrt{e(1-e)},
		\end{align}
		where the positivity of $X$ follows from Lemma \ref{lem:kite-mixed-monotonicity}.

		We first note the identity
		\begin{align}\label{eq:kite-kappa}
			\cos(\alpha+\beta)
			=\frac{ac-h^2}{\sqrt{(a^2+h^2)(c^2+h^2)}}
			=\frac{\kappa}{\sqrt{\kappa^2+S^2}}.
		\end{align}
		In particular, $\kappa>0$ is equivalent to $\alpha+\beta<\pi/2$.
		If $\kappa\le0$, the inequality \eqref{eq:kite-mixed-quadratic} implies $e>1/2$.
		Suppose now that $\kappa>0$. Then $\alpha+\beta<\pi/2$, so $\Theta<1/2$.
		If $e\ge1/2$, the desired inequality follows immediately.
		If $e<1/2$, the inequalities \eqref{eq:kite-mixed-quadratic} and \eqref{eq:kite-X-cauchy} give
		\[
		\kappa\sqrt{e(1-e)}>S(1/2-e).
		\]
		Hence
		\[
		e>\frac12\left(1-\frac{\kappa}{\sqrt{\kappa^2+S^2}}\right)=\frac12\left(1- \cos(\alpha+\beta)\right)=\Theta(a,h).
		\]

		At a crossing with $E_D\le E_N$, Lemma~\ref{lem:kite-eighth}
		implies $e\le1/8$. Thus $h^2< ac$ and
		$\Theta<1/8$. Both $\alpha$ and $\beta$ increase strictly with
		$h$, so $\Theta(a,h)$ also increases strictly.
		At $\Theta=1/8$ one has
		$\tan(\alpha+\beta)=\sqrt7/3$.
		Since
		\[
		\tan(\alpha+\beta)=\frac{(2a-1)h}{ac-h^2},
		\]
		the threshold is the positive root of
		\begin{align}
			\sqrt7h^2+3(2a-1)h-\sqrt7ac=0.
		\end{align}
		This is exactly the value $h_b(a)$ in
		\eqref{eq:kite-failing-window} and $h<h_b(a)$ holds.
	\end{proof}

    {\color{black}To deal with the remaining small-height regime, we need two further quantitative estimates for the mixed eigenvalue, namely a lower bound and a comparison under vertical compression.}
	\begin{Lemma}\label{lem:kite-robin}
		Let $a>1$, $h>0$ and $0<k=h/a\le1/2$. Set
		\[
		b_*(k)=\frac{\pi/(2k)-1}{1+k^2}-1>0.
		\]
		Then
		\begin{equation}\label{eq:kite-robin-lower}
			\lambda(a,h)\ge
			\min\left\{\frac{\pi/k-1}{1+k^2},\,
			\left(\frac{c^2+h^2}{j_{0,1}^2}+\frac{c}{2b_*(k)}\right)^{-1}
			\right\}.
		\end{equation}
	\end{Lemma}
	\begin{proof}
		Let $D=(1,k)$ and split $T$ along $P_3D$ into $T_-=\triangle P_1P_3D$ and $T_+=\triangle P_3P_4D$.
		On $T_-$ use polar coordinates $(r,\vartheta)$ centered at $P_1$.
		Set
		\[
		m_\alpha=\frac\pi{2\alpha},\quad s=m_\alpha-1>0,
		\quad P(r,\vartheta)=r^s\sin(m_\alpha\vartheta).
		\]
		The function $P$ is positive in $T_-$, vanishes on $P_1P_3$, and
		satisfies $\partial_nP=0$ on $P_1D$. A calculation in polar
		coordinates gives
		\[
		\Delta P=(s^2-m_\alpha^2)r^{s-2}\sin(m_\alpha\vartheta),
		\quad
		-\frac{\Delta P}{P}=\frac{2m_\alpha-1}{r^2}.
		\]
		Since $\alpha=\arctan k\le k$ and $r^2\le1+k^2$,
		\begin{equation}\label{eq:kite-Picone-interior}
			-\frac{\Delta P}{P}
			\ge\frac{\pi/k-1}{1+k^2}=:C_-.
		\end{equation}
		On that interface
		$r=1/\cos\vartheta$, and the identity
		$\partial_x=\cos\vartheta\,\partial_r
		-r^{-1}\sin\vartheta\,\partial_\vartheta$ gives
		\begin{align}\label{eq:kite-Picone-interface}
			\frac{\partial_nP}{P}=\frac{P_x}{P}
			&=s\cos^2\vartheta
			-m_\alpha\sin\vartheta\cos\vartheta
			\cot(m_\alpha\vartheta)\notag\\
			&\ge\frac{s}{1+k^2}-1\ge b_*(k).
		\end{align}
		Here we used $0\le t\cot t\le1$ for $0<t\le\pi/2$ and
		$\sin\vartheta\cos\vartheta\le\vartheta$.
		
		{\color{black}For smooth $f$ vanishing in a neighborhood of the closed Dirichlet edge $P_1P_3$, Picone's identity gives}
		\begin{equation}\label{eq:kite-Picone-pointwise}
			|\nabla f|^2-\nabla P\cdot\nabla\left(\frac{f^2}{P}\right)
			=\left|\nabla f-\frac fP\nabla P\right|^2
			=P^2\left|\nabla\left(\frac fP\right)\right|^2\ge0.
		\end{equation}
		Integrating this identity by parts, we have
		\begin{align}\label{eq:kite-Picone-integrated}
			\int_{T_-}|\nabla f|^2d\sigma
			&=\int_{T_-}P^2\left|\nabla(f/P)\right|^2d\sigma
			+\int_{T_-}\left(-\frac{\Delta P}{P}\right)f^2d\sigma
			+\int_{\partial T_-}\frac{\partial_nP}{P}f^2 ds \notag\\
			&\ge C_-\int_{T_-}f^2d\sigma+b_*(k)\int_{P_3D}f^2 ds.
		\end{align}
        {\color{black}Since such smooth functions are dense in $\mathcal S(T_-)$ in the $H^1$ norm, \eqref{eq:kite-Picone-integrated} holds for all $f\in \mathcal S(T_-)$.}
		
		On an interval $(0,R)$,
		by using the variational characterization of the first eigenvalue of
		$-(rw_0')'=\lambda rw_0$, with finite weighted energy and boundary condition $w_0(R)=0$,
		whose first eigenvalue is
		${j_{0,1}^2}/{R^2}$, we obtain the weighted Poincar\'e inequality
		\begin{equation}\label{eq:kite-radial-Dirichlet}
			\int_0^R w_0(r)^2 rdr
			\le\frac{R^2}{j_{0,1}^2}\int_0^R w_0'(r)^2 rdr,
			\qquad w_0(R)=0.
		\end{equation}
		Apply \eqref{eq:kite-radial-Dirichlet} to
		$w_0=w-w(R)$. Since
		$\|w(R)\|_{L^2((0,R),rdr)}=R|w(R)|/\sqrt2$, the triangle
		inequality gives
		\[
		\|w\|_{L^2(rdr)}
		\le\frac R{j_{0,1}}\|w'\|_{L^2(rdr)}
		+\frac R{\sqrt2}|w(R)|.
		\]
		For any $\beta_R>0$, the Cauchy--Schwarz inequality then yields
		\begin{equation}\label{eq:kite-ray-Robin}
			\int_0^R w^2 rdr
			\le\left(\frac{R^2}{j_{0,1}^2}+\frac R{2\beta_R}\right)
			\left(\int_0^R w'^2rdr+\beta_RR w(R)^2\right).
		\end{equation}
		Choose $\beta_R=b_*(k)R/c$ and set
		\[
		C_+=\frac{c^2+h^2}{j_{0,1}^2}+\frac c{2b_*(k)},
		\quad R(\theta)=\frac{c}{\cos\theta}.
		\]
		For $w_\theta(r)=f(a-r\cos\theta,h-r\sin\theta)$, $r\in(0,R(\theta))$,
		\eqref{eq:kite-ray-Robin} give
		\begin{align}\label{eq:kite-ray-to-domain}
			\int_{T_+}f^2
			&=\int_{\alpha}^{\beta}\int_0^{R(\theta)}
			r w_\theta(r)^2 drd\theta\notag\\*
			&\le C_+\left[
			\int_{\alpha}^{\beta}\int_0^{R(\theta)}
			r w'_\theta(r)^2 drd\theta
			+\frac{b_*(k)}c\int_{\alpha}^{\beta}
			R(\theta)^2 w_\theta(R(\theta))^2\,d\theta\right]\notag\\*
			&\le C_+\left(\int_{T_+}|\nabla f|^2d\sigma
			+b_*(k)\int_{P_3D}f^2\,ds\right),
		\end{align}
		where the last inequality uses $w'_\theta(r)^2 \le |\nabla f|^2$ and
		the measure $ds=\frac{R(\theta)^2}{c}d\theta$ on $P_3D$. Thus
		\begin{equation}\label{eq:kite-robin-right}
			\int_{T_+}|\nabla f|^2d\sigma+b_*(k)\int_{P_3D}f^2 ds
			\ge C_+^{-1}\int_{T_+}f^2d\sigma.
		\end{equation}
		
		Finally, adding
		\eqref{eq:kite-Picone-integrated} and \eqref{eq:kite-robin-right} and
		canceling the common interface term gives
		\[
		\int_{T}|\nabla f|^2d\sigma
		\ge C_-\int_{T_-}f^2d\sigma+C_+^{-1}\int_{T_+}f^2d\sigma
		\ge\min\{C_-,C_+^{-1}\}\int_{T}f^2d\sigma.
		\]
		Taking the infimum of the mixed Rayleigh quotient proves
		\eqref{eq:kite-robin-lower}.
	\end{proof}
	
	\begin{Lemma}\label{lem:kite-compression}
		For $a,h>0$ and $\tau\ge1$,
		\begin{equation}\label{eq:kite-compression}
			\lambda(a,h/\sqrt\tau)\le\lambda(a,h)+(\tau-1)E_D(a,h).
		\end{equation}
	\end{Lemma}
	\begin{proof}
		On $T_{a,h/\sqrt\tau}$, use the test function
		\[
		v_\tau(x,y)=\tau^{1/4}\nu(x,\sqrt\tau\,y;a,h).
		\]
		It has zero trace on $P_1P_3$. The substitution
		$Y=\sqrt\tau\,y$ gives
		\begin{align*}
			\int_{T_{a,h/\sqrt\tau}}v_\tau^2d\sigma=1,\quad
			\int_{T_{a,h/\sqrt\tau}}(v_\tau)_x^2d\sigma=\lambda(a,h)-E_D(a,h),\quad
			\int_{T_{a,h/\sqrt\tau}}(v_\tau)_y^2d\sigma =\tau E_D(a,h).
		\end{align*}
		Thus, the claimed upper bound follows from the variational characterization.
	\end{proof}

    {\color{black}Combining the mixed-energy estimate with the lower bound and compression argument above yields the desired transversality for \(1<a\le 3/2\).}
	\begin{Lemma}\label{prop:kite-small-a-transverse}
		For $1<a\le3/2$, every crossing satisfies $E_D>E_N$.
	\end{Lemma}
	\begin{proof}
		Suppose, to the contrary, that
		$\lambda(a,h)=\mu(a,h)=\Lambda$ and $E_D\le E_N$.
		Lemma~\ref{lem:kite-eighth} and
		Proposition~\ref{prop:kite-mixed-energy} give
		\[
		E_D\le\frac\Lambda8,\qquad h<h_b(a).
		\]
		A simple analysis shows that $h_b(a)$ and $h_b(a)/a$ are increasing
		for $1<a\le 3/2$, then we have
		\begin{equation}\label{eq:kite-hb-small}
			\frac{h_b(a)}a\le\frac{h_b(3/2)}{3/2}<0.2,\quad h_b(a)\le h_b(3/2)<0.3,
			\quad 1<a\le3/2.
		\end{equation}
		
		We next give a simple upper bound for the Neumann eigenvalue.
		Take $v=x-(1+a)/3$ as a test function. A direct computation gives
		\[
		\int_T|\nabla v|^2d\sigma=\frac h2,\quad
		\int_Tv^2d\sigma=\frac{h(a^2-a+1)}{36}.
		\]
		The Rayleigh principle implies
		\begin{equation}\label{eq:kite-affine-upper}
			\mu(a,h)\le \frac{18}{a^2-a+1}
			=\frac{18}{1+c+c^2}.
		\end{equation}
		Taking $h'=h/\sqrt3$ in Lemma~\ref{lem:kite-compression}, we obtain
		\begin{equation}\label{eq:kite-compressed-upper}
			\lambda(a,h')\le\Lambda+2E_D
			\le\frac54\Lambda\le \frac{22.5}{(1+c+c^2)}.
		\end{equation}
		
		On the other hand, \eqref{eq:kite-hb-small} gives
		$h'^2<0.03$, and $k=\frac{h'}a<\frac{0.2}{\sqrt3}<\frac18$, then
		\[
		\frac{\pi/k-1}{1+k^2}>22.5,\qquad
		b_*(k)>b_*(\frac 18)>10.
		\]
		and
		\begin{equation}\label{eq:kite-robin-denominator}
			\frac{c^2+h'^2}{j_{0,1}^2}+\frac c{2b_*(k)}
			<\frac 1{j_{0,1}^2}\left(c^2+0.03\right)+\frac c{20}
			<\frac{1+c+c^2}{22.5},\quad \text{ for }0<c\le\frac12.
		\end{equation}
		It follows that each term in the minimum in
		\eqref{eq:kite-robin-lower} is strictly greater than
		$22.5/(1+c+c^2)$, so Lemma \ref{lem:kite-robin} contradicts
		\eqref{eq:kite-compressed-upper}, proving the Lemma.
	\end{proof}
	
	\subsection[Concave kites: \texorpdfstring{$3/2<a<2$}{3/2 < a < 2}]{Concave kites: $3/2<a<2$}\label{subsec:kite-residual}
	
	In this range we estimate the Neumann vertical energy by separating
	the first positive weighted one-dimensional mode from its
	orthogonal remainder. Set
	\begin{equation}\label{eq:kite-normalized-parameters}
		q=\frac{a-1}{a},\qquad b=1-q=\frac1a,\qquad p=\frac ha,
		\qquad \frac13\le q<\frac12,
	\end{equation}
	and the triangle
	$\Omega_{q,p}=\operatorname{int}\operatorname{conv}\{(0,0),(b,0),(1,p)\}$.
	The map $(x,y)\mapsto(ax,ay)$ takes $\Omega_{q,p}$ to
	$T_{a,h}$. Define $u(x,y)=a\varphi(ax,ay)$, the factor $a$
	compensates for the area factor $a^2$ in the change of variables.
	The eigenvalue and the vertical energy scale by the same factor,
	so
	\begin{equation}\label{eq:kite-similarity}
		\int_{\Omega_{q,p}}u^2d\sigma=1,\quad
		\widehat\mu=a^2\mu(a,h),\quad
		E:=\int_{\Omega_{q,p}}u_y^2d\sigma=a^2E_N,\quad
		\frac E{\widehat\mu}=\frac{E_N}{\mu(a,h)}.
	\end{equation}
	Moreover, the necessary condition for the crossing height in Lemma \ref{lem:kite-large}, $h^2\le a(2-a)$, becomes $p^2\le1-2q$. With the above notation, $\alpha=\arctan p$ and $\beta=\arctan(p/q)$.

	\begin{Lemma}\label{lem:kite-geometric-window}
		Suppose $1/3\le q<1/2$, $p>0$, $p^2\le1-2q$, and
		$\Theta\le1/8$. Then
		\begin{equation}\label{eq:kite-geometric-bounds}
			p<\frac14,\quad
			p^2<\frac{5q^2}{8(1+q)^2},\quad
			\Theta>\frac{33}{16}p^2.
		\end{equation}
	\end{Lemma}
	\begin{proof}
		Since $p^2\le1-2q$, we have
		$p/q\ge2p/(1-p^2)$ and hence $\beta\ge2\alpha$.
		Thus $\Theta\ge\sin^2(3\alpha/2)$.
		Therefore the assumption $\Theta\le 1/8$ yields $\cos(3\alpha)\ge 3/4$.
		On the other hand, 	at $p_0=1/4$, $\cos(3\alpha)=\cos(3\arctan(p_0))<3/4$.
		Since $\cos(3\arctan p)$ is strictly decreasing in $p$ in the relevant range, it follows that $p<1/4$.
		The condition $\Theta\le1/8$ also gives
		\[
		F(p):=p^2+\frac{3(1+q)}{\sqrt7}p-q\le0.
		\]
		At $p_0=\frac{\sqrt5q}{\sqrt8(1+q)}$, we have
		\begin{align}
			\frac{F(p_0)}{q}=\frac58\frac q{(1+q)^2}+\frac{3\sqrt{70}}{28}-1>0,
		\end{align}
		since $q/(1+q)^2\ge3/16$. Therefore $F(p_0)>0$.
		Consequently, $p<p_0$, which proves the second bound.
		Finally, for $z=\cos\alpha$,
		\[
		\frac{\sin^2(3\alpha/2)}{\tan^2\alpha}
		=\frac{z^2(2z+1)^2}{2(1+z)}>\frac{33}{16},
		\]
		since the right-hand side is increasing for $z>0$ and
		$z=(1+p^2)^{-\frac{1}{2}}>\frac{4}{\sqrt{17}}$.
		This proves the lemma.
	\end{proof}
	
	\paragraph{The weighted comparison problem.}
	Set
	\begin{equation}\label{eq:kite-fibres}
		\ell(x)=\begin{cases}0,&0<x<b,\\(x-b)/q,&b<x<1,\end{cases}
		\quad
		g(x)=x-\ell(x)=\begin{cases}x,&0<x<b,\\b(1-x)/q,&b<x<1.\end{cases}
	\end{equation}
	Denote $\mathcal D_q=\Omega_{q,1}=\{(x,y)|0<x<1,\ell(x)<y<x\}$,
	thus $g(x)$ is the vertical thickness, and the maximum of $g$ is $b$.
	It is also convenient to use the weight normalized to have maximum one
	\begin{equation}\label{eq:kite-tent-weight}
		w_q(x)=\frac{g(x)}b
		=\begin{cases}x/b,&0<x<b,\\(1-x)/q,&b<x<1.\end{cases}
	\end{equation}
	
	Let $0=\zeta_0(q)<\zeta_1(q)<\zeta_2(q)<\cdots$ be the eigenvalues
	of the weighted one-dimensional Neumann problem $-(w_q\Psi')'=\zeta w_q\Psi$.
	%The two weights $w_q$ and $g$ give the same eigenvalues but different normalizations.
	Choose an eigenfunction $\Psi_1$ associated with the first positive eigenvalue so that
	\begin{equation}\label{eq:kite-weighted-normalization}
		-(w_q\Psi_1')'=\zeta_1 w_q\Psi_1,\qquad
		\int_0^1w_q\Psi_1dx=0,\qquad\int_0^1w_q\Psi_1^2dx=1,
		\qquad \psi=\frac{\Psi_1}{\sqrt b}.
	\end{equation}
	Then
	\[
	\int_0^1g\psi^2dx=1,\qquad \int_0^1g\psi dx=0,\qquad
	\int_0^1g\psi'^2dx=\zeta_1,
	\]
	and the lift of $\psi$ to $\mathcal D_q$ has the same normalization.
	The function $p^{-1/2}\psi(x)$ on $\Omega_{q,p}$ has unit
	$L^2$ norm and zero mean. Its vertical derivative vanishes, and its
	horizontal energy is $\int_0^1g\psi'^2=\zeta_1$. Thus the Rayleigh
	principle gives, at every positive height,
	\begin{equation}\label{eq:kite-mu-zeta}
		\widehat\mu\le\zeta_1.
	\end{equation}
	
	For an eigenvalue $\zeta$,
	by an argument similar to that in Lemma \ref{lem:kite-neumann-thin-limit}, we have
	\begin{equation}\label{eq:kite-radial-pieces}
		\Psi(x)=
		\begin{cases}
			A J_0(kx),&0<x<b,\\
			BJ_0(k(1-x)),&b<x<1,
		\end{cases}
	\end{equation}
	where $k=\sqrt{\zeta}$. Once $A$ is fixed, the matching conditions determine $B$ uniquely from $A$.
	Hence all eigenvalues $\zeta$ are simple.
	The matching conditions also imply that the eigenvalue satisfies
	\begin{equation}\label{eq:kite-normalized-matching}
		J_0(kb)J_1(kq)
		+J_1(kb)J_0(kq)=0.
	\end{equation}
	Applying the root argument of
	Lemma \ref{lem:kite-limit-comparison}, we obtain
	that the first positive eigenvalue satisfies
	\begin{equation}\label{eq:kite-first-root-location}
		j_{0,1}<\sqrt{\zeta_1}b<j_{1,1},\quad 0<\sqrt{\zeta_1}q<j_{0,1},\quad
		R(\sqrt{\zeta_1}b)+R(\sqrt{\zeta_1}q)=0.
	\end{equation}
	where $R(t)=J_0(t)/J_1(t)$.

	\begin{Lemma}\label{lem:kite-weighted-identities}
		For the positive weighted eigenvalue $\zeta_j(q)$, let $\Psi_j$ be
		normalized by $\displaystyle\int_0^1w_q\Psi_j^2dx=1$. Set
		\[
		f_j=\Psi_j(b),\qquad d_j=\Psi_j'(b),
		\]
		and define its left and right masses and energies by
		\begin{align*}
			M_L&=\int_0^bw_q\Psi_j^2dx,&
			M_R&=\int_b^1w_q\Psi_j^2dx,\\
			\mathcal E_L&=\int_0^bw_q\Psi_j'^2dx,&
			\mathcal E_R&=\int_b^1w_q\Psi_j'^2dx.
		\end{align*}
		Then
		\begin{align}
			&M_L=b,\quad M_R=q,\quad f_j^2+d_j^2/\zeta_j=2, \label{eq:kite-mass-split}\\
			&\mathcal E_L=\zeta_jb+f_jd_j,\qquad
			\mathcal E_R=\zeta_jq-f_jd_j, \label{eq:kite-energy-split}\\
			&\zeta_j'(q)=\frac{f_jd_j}{bq}. \label{eq:kite-weighted-derivative}
		\end{align}
	\end{Lemma}
	
	\begin{proof}
		The radial equation $-(rU')'=\zeta_jrU$ implies
		\begin{equation}\label{eq:kite-radial-mass-identity}
			\frac{d}{dr}\left[r^2(U'^2+\zeta_jU^2)\right]
			=2\zeta_jrU^2.
		\end{equation}
		We first apply this identity on the left interval.
		Set $r=x$ and write $U(r)=\Psi_j(r)$. Then $w_q(x)=r/b$.
		Applying \eqref{eq:kite-radial-mass-identity} on $(0,b)$ therefore yields
		\begin{align}
			M_L=\frac{1}{b}\int_0^b rU^2(r) dr=\frac{1}{2b\zeta_j}\int_0^b
			\frac{d}{dr}\left[r^2(U'^2+\zeta_jU^2)\right]dr=
			\frac b{2\zeta_j}(d_j^2+\zeta_jf_j^2).
		\end{align}
		Here the terms at $r=0$ vanish since the finite-energy solution is the regular Bessel branch.
		On the right interval, set $r=1-x$ and write $U(r)=\Psi_j(1-r)$.
		Then $w_q(x)=r/q$. Applying \eqref{eq:kite-radial-mass-identity} on $(0,q)$ therefore yields
		\begin{align}
			M_R=\frac{1}{q}\int_0^q rU^2(r) dr=\frac{1}{2q\zeta_j}\int_0^q
			\frac{d}{dr}\left[r^2(U'^2+\zeta_jU^2)\right]dr=
			\frac q{2\zeta_j}(d_j^2+\zeta_jf_j^2).
		\end{align}
		Since $M_L+M_R=1$ and $b+q=1$, the preceding two identities imply
		$f_j^2+d_j^2/\zeta_j=2$. Substitution back into the formulas for $M_L$
		and $M_R$ proves \eqref{eq:kite-mass-split}.
		
		We next derive the corresponding energy identities.
		Multiplying $-(rU')'=\zeta_j rU$ by $U$ and integrating separately
		over two intervals gives
		\begin{align*}
			\mathcal E_L
			&=\zeta_jM_L+[w_q\Psi_j\Psi_j']_0^b
			=\zeta_jb+f_jd_j,\\
			\mathcal E_R
			&=\zeta_jM_R+[w_q\Psi_j\Psi_j']_b^1
			=\zeta_jq-f_jd_j.
		\end{align*}

		It remains to differentiate the eigenvalue with respect to $q$.
		Since the weight and the interface point $b=1-q$
		vary with $q$, it is convenient to keep the domain fixed. Put
		\[
		U_L(t)=\Psi_j(bt),\qquad U_R(t)=\Psi_j(1-qt),\qquad 0<t<1.
		\]
		The space of such pairs consists of finite-energy functions with
		$U_L(1)=U_R(1)$ and is independent of $q$.
		The energy and mass forms on this space are
		\begin{align*}
			e_q(U,V)
			&=\frac1b\int_0^1tU_L'V_L'dt
			+\frac1q\int_0^1tU_R'V_R'dt,\\
			m_q(U,V)
			&=b\int_0^1tU_LV_Ldt+q\int_0^1tU_RV_Rdt.
		\end{align*}
		Hence the normalized eigenfunction satisfies
		$m_q(U,U)=1$ and
		$e_q(U,V)=\zeta_jm_q(U,V)$ for every admissible $V$.
		Since $\zeta_j$ are simple, the corresponding normalized 
		eigenfunction can be chosen locally differentiably in $q$.
		Differentiating the energy and the normalization gives
		\begin{align*}
			\zeta_j'
			&=(\partial_qe_q)(U,U)+2e_q(U,\partial_qU),\\
			0&=(\partial_qm_q)(U,U)+2m_q(U,\partial_qU).
		\end{align*}
		where $\partial_qe_q$ and $\partial_qm_q$ mean differentiation of
		the explicit coefficients of the forms while keeping $U$ fixed.
		Using the weak eigenvalue equation with $V=\partial_qU$,
		the terms involving $\partial_qU$ cancel. Therefore
		\begin{equation}\label{eq:kite-weighted-HF}
			\zeta_j'=(\partial_qe_q)(U,U)-\zeta_j(\partial_qm_q)(U,U).
		\end{equation}
		Finally, the explicit derivatives of the mass and energy forms are
		\begin{align*}
			(\partial_qm_q)(U,U)
			&=-\frac{M_L}{b}+\frac{M_R}{q}=0,\\
			(\partial_qe_q)(U,U)
			&=\frac{\mathcal E_L}{b}-\frac{\mathcal E_R}{q}
			=f_jd_j\left(\frac1b+\frac1q\right)
			=\frac{f_jd_j}{bq}.
		\end{align*}
		Substituting these two identities into
		\eqref{eq:kite-weighted-HF} proves
		\eqref{eq:kite-weighted-derivative}.
	\end{proof}
	
	\begin{Lemma}\label{lem:kite-weighted-spectrum}
		For $1/3\le q\le1/2$, the first two weighted eigenvalues
		satisfy
		\begin{equation}\label{eq:kite-weighted-gap}
			\zeta_1(q)\le4j_{0,1}^2,\qquad
			\zeta_2(q)\ge4j_{1,1}^2.
		\end{equation}
	\end{Lemma}
	\begin{proof}
		For the first weighted eigenvalues, let $k_1(q)=\sqrt{\zeta_1(q)}$,
		then \eqref{eq:kite-first-root-location} gives
		$J_0(k_1b)<0$ and $J_1(k_1b)>0$. Hence
		\begin{align}\label{eq:kite-xi-1-increase}
			f_1d_1=-A^2k_1J_0(k_1b)J_1(k_1b)>0.
		\end{align}
		By \eqref{eq:kite-weighted-derivative}, $\zeta_1(q)$ increases
		strictly with $q$.
		At $q=1/2$ the weight is symmetric about $1/2$. Reflection splits
		the problem into a radial Dirichlet problem at $r=1/2$ for the odd
		modes and a radial Neumann problem there for the even modes.
		Interlacing of the Bessel zeros therefore gives
		\[
		\zeta_1(1/2)=4j_{0,1}^2,\qquad
		\zeta_2(1/2)=4j_{1,1}^2.
		\]
		This proves the first bound.
		
		For the second weighted eigenvalues, let
		$k_2(q)=\sqrt{\zeta_2(q)}$ and $s(q)=bk_2(q)$.
		At $q=1/2$ this mode is even, so $d_2=0$ and
		$\zeta_2'(1/2)=0$. Consequently
		\[
		s(1/2)=j_{1,1},\qquad s'(1/2)=-2j_{1,1}.
		\]
		It follows that
		$j_{1,1}<s(q)<j_{0,2}$ for $q<1/2$ sufficiently close to $1/2$.
		We now show that this location holds throughout $[1/3,1/2)$.
		With $c_0=q/b\in[1/2,1)$, write the matching determinant as
		\[
		\mathscr M_{c_0}(s)
		=J_0(s)J_1(c_0s)+J_1(s)J_0(c_0s).
		\]
		At the two proposed endpoints,
		\begin{align*}
			\mathscr M_{c_0}(j_{1,1})
			&=J_0(j_{1,1})J_1(c_0j_{1,1})<0,\\
			\mathscr M_{c_0}(j_{0,2})
			&=J_1(j_{0,2})J_0(c_0j_{0,2})>0.
		\end{align*}
		For the last sign, $j_{0,2}>2j_{0,1}$ implies
		$j_{0,1}<c_0j_{0,2}<j_{0,2}$, so the two factors are both negative.
		The continuous simple branch $s(q)$ cannot leave
		$(j_{1,1},j_{0,2})$ without meeting an endpoint at which the
		matching determinant is nonzero. Thus it stays in that interval.
		There both $J_0(s(q))$ and $J_1(s(q))$ are negative, and hence
		\[
		f_2d_2=-A^2k_2J_0(s(q))J_1(s(q))<0.
		\]
		Equation~\eqref{eq:kite-weighted-derivative} shows that
		$\zeta_2(q)$ decreases as $q$ increases to $1/2$. Therefore
		$\zeta_2(q)\ge4j_{1,1}^2$.
	\end{proof}

	\paragraph{The effective slope.}
	In the fibre coordinates $t=(y-\ell(x))/g(x)\in(0,1)$, set
	\begin{equation}\label{eq:kite-fibre-slope}
		\beta_0(x,t)=\ell'(x)+g'(x)t
		=\begin{cases}t,&x<b,\\(1-bt)/q,&x>b.\end{cases}
	\end{equation}
	The mean square of this coefficient on a fibre is
	\begin{align}\label{eq:kite-beta-0}
		\int_0^1\beta_0(x,t)^2dt
		=\begin{cases}
			\displaystyle\int_0^1t^2\,dt=\frac13,&0<x<b,\\[2mm]
			\displaystyle\frac1{q^2}\int_0^1(1-bt)^2\,dt
			=\frac{1+q^{-1}+q^{-2}}3=:C_R,&b<x<1.
		\end{cases}
	\end{align}
	The coefficient $\beta_0$ describes the change of the vertical
	coordinate when $x$ varies along a fixed fibre parameter $t$.
	Define
	\begin{equation}\label{eq:kite-effective-slope}
		C_{\rm eff}=\frac1{\zeta_1}\int_{\mathcal D_q}\psi'(x)^2\beta_0(x,t)^2dxdy.
	\end{equation}
	
	\begin{Lemma}\label{lem:kite-effective-slope}
		For $1/3\le q\le1/2$, one has $C_{\rm eff}\le4/3$.
	\end{Lemma}
	\begin{proof}
		Let $\mathcal E_L$ and $\mathcal E_R$ be the left and right weighted energies of $\Psi_1$.
		Since $g=bw_q$ and $\psi=\Psi_1/\sqrt b$, one has
		$g(x)\psi'(x)^2=w_q(x)\Psi'_1(x)^2$. 
		For fixed $x$, the change of variables $y=\ell(x)+g(x)t$ gives
		\begin{align}
			\int_{\ell(x)}^x \psi'(x)^2\beta_0(x,t)^2dy=g(x)\psi'(x)^2\int_0^1\beta_0(x,t)^2dt.
		\end{align}
		Together with \eqref{eq:kite-energy-split} and \eqref{eq:kite-beta-0},
		splitting the integral \eqref{eq:kite-effective-slope} at $b$ therefore gives
		\[
		C_{\rm eff}=\frac1{\zeta_1}
		\left(\frac13\mathcal E_L+C_R\mathcal E_R\right).
		\]
		Using $\mathcal E_L+\mathcal E_R=\zeta_1$ and
		Lemma \ref{lem:kite-weighted-identities}, we obtain
		\begin{equation}\label{eq:kite-effective-energy}
			C_{\rm eff}=\frac13+\frac{q^{-1}+q^{-2}}3
			\frac{\mathcal E_R}{\zeta_1},\qquad
			\frac{\mathcal E_R}{\zeta_1}=q-\frac{fd}{\zeta_1},
		\end{equation}
		where $f=\Psi_1(b)$ and $d=\Psi'_1(b)$.

		For $q<1/2$, write $k=\sqrt{\zeta_1}$ and
		$\xi=R(kq)=-R(kb)>0$.
		The interface conditions and the left-hand representation $\Psi_1=AJ_0(kx)$
		give $\frac{f}{d}=-\frac{1}{k}R(kb)=\frac \xi k$,
		so \eqref{eq:kite-mass-split} yields
		\begin{equation}\label{eq:kite-fd-xi}
			\frac{fd}{\zeta_1}=\frac{2\xi}{k(1+\xi^2)}.
		\end{equation}
		In particular, Lemma \ref{lem:kite-weighted-identities} gives
		$k'(q)=\frac{fd}{2kbq}>0$, {\color{black}hence both $k$ and $kq$ increase with $q$.}

		At $q=1/3$, one has $b=2/3$, putting $s=kb$ gives
		$R(s)+R(s/2)=0$. Since $R$ is strictly decreasing and $R(3)+R(3/2)>0$, its root satisfies
		$s>3$, and hence $kq=s/2>3/2$ at $q=1/3$,
		then the monotonicity of $kq$ yields $kq>3/2$ for $1/3\le q \le 1/2$.
		Moreover, since $kq<j_{0,1}$ and $R$ is decreasing, we have
		$0<\xi=R(kq)<R(3/2)<1$.
		We now estimate $\xi$ from below. The Riccati identity for $R$ is
		\begin{align}
			R'(t)=-1-R(t)^2+\frac{R(t)}{t}=-\left( R(t)-\frac{1}{2t}\right)^2-1+\frac{1}{4t^2}.
		\end{align}
		For every $t\in[kq,kb]$, since $t\ge kq>3/2$, we have
		\[
		-R'(t)=\left(R(t)-\frac1{2t}\right)^2+1-\frac1{4t^2}>
		\frac89.
		\]
		Integrating over $[kq,kb]$ yields
		\begin{align}
			2\xi=R(kq)-R(kb)=\int_{kq}^{kb} -R'(t)dt>\frac{8}{9}k(b-q).
		\end{align}
		Therefore
		\begin{align}
			\frac{fd}{\zeta_1}=\frac{2\xi}{k(1+\xi^2)}>\frac{\xi}{k}>\frac{4}{9}(b-q)=\frac{4}{9}(1-2q)>\frac{1-2q}{3}.
		\end{align}
		At $q=1/2$, the weight $w_q$ is symmetric with respect to $x=1/2$.
		Hence $f=0$, and the last bound holds with equality.
		
		Finally, substituting this bound in \eqref{eq:kite-effective-energy} gives
		\[
		C_{\rm eff}\le\frac13+\frac{q^{-1}+q^{-2}}3
		\left(q-\frac{1-2q}{3}\right)
		=\frac43-\frac{(1-2q)^2}{9q^2}\le\frac43.
		\]
	\end{proof}
	
	\paragraph{Elimination of the spectral residual.}
	Pull the normalized Neumann eigenfunction back to $\mathcal D_q$ by
	$v(x,y)=\sqrt p\,u(x,py)$ and define
	\begin{equation}\label{eq:kite-fixed-form}
		\mathfrak a_p(f_1,f_2)=\int_{\mathcal D_q}
		\bigl((f_1)_x(f_2)_x+p^{-2}(f_1)_y(f_2)_y\bigr)dxdy.
	\end{equation}
	Then $\displaystyle\int_{\mathcal D_q}v^2d
	\sigma=1$, $\displaystyle\int_{\mathcal D_q}vd\sigma=0$, and
	\begin{equation}\label{eq:kite-fixed-weak}
		\mathfrak a_p(v,\chi)=\widehat\mu\int_{\mathcal D_q}v\chi d\sigma
		\quad(\chi\in H^1(\mathcal D_q)),\qquad
		E=p^{-2}\|v_y\|_2^2.
	\end{equation}
	Here and below an unweighted $L^2$ norm is over $\mathcal D_q$,
	and $\|f\|_g^2=\displaystyle\int_0^1gf^2dx$.

	The following lemma quantifies the deviation of the two-dimensional
	Neumann eigenfunction from the first positive weighted eigenfunction.
	The key is to control the orthogonal remainder and its coupling with
	the first weighted eigenfunction, and then eliminate the spectral
	residual $\zeta_1-\widehat\mu$.
	
	\begin{Lemma}\label{lem:kite-residual}
		Let $\zeta_1<G\le\zeta_2(q)$ and suppose
		\begin{equation}\label{eq:kite-coercivity-constant}
			\mathscr D=1-\widehat\mu p^2
			\left(\frac{b^2}{\pi^2}+\frac{C_R}{G-\widehat\mu}\right)>0.
		\end{equation}
		With $\rho=\widehat\mu/\zeta_1\in(0,1]$, one has
		\begin{equation}\label{eq:kite-residual-estimates}
			1-\rho\le\frac{C_{\rm eff}p^2}{\mathscr D},\qquad
			\frac E{\widehat\mu}\le
			\frac{C_{\rm eff}p^2}{\rho\mathscr D^2}.
		\end{equation}
	\end{Lemma}
	\begin{proof}
		We decompose the pulled-back Neumann eigenfunction into its projection onto the
		first positive weighted eigenfunction and an orthogonal remainder:
		\begin{equation}\label{eq:kite-mode-decomposition}
			v=t_0\psi+z,\qquad t_0=\int_{\mathcal D_q}v\psi d\sigma, \qquad\int_{\mathcal D_q}z\psi d\sigma=0.
		\end{equation}
		The normalization and mean-zero conditions imply
		\[
		\int_{\mathcal D_q}\psi^2 d\sigma=1,
		\qquad
		\int_{\mathcal D_q}\psi d\sigma=0,
		\qquad
		\int_{\mathcal D_q}\psi_x^2 d\sigma=\zeta_1,
		\]
		and $\psi$ is independent of the vertical variable $y$.
		Since $v$ is normalized and has zero mean, it follows that
		\begin{equation}\label{eq:kite-z-basic}
			\int_{\mathcal D_q}z d\sigma=0,
			\qquad
			t_0^2+\|z\|_2^2=1,
			\qquad
			E=p^{-2}\|z_y\|_2^2.
		\end{equation}
		Let
		$\varepsilon=\zeta_1-\widehat\mu\ge0$,
		where the nonnegativity follows from
		$\widehat\mu\le\zeta_1$ in \eqref{eq:kite-mu-zeta}.

		We first derive two identities for the projection. Testing the weak equation
		\eqref{eq:kite-fixed-weak} with $\psi$ gives
		\[
		\mathfrak a_p(v,\psi)
		=\widehat\mu\int_{\mathcal D_q}v\psi d\sigma
		=\widehat\mu t_0.
		\]
		Since $\psi_y=0$ and
		$\mathfrak a_p(\psi,\psi)=\zeta_1$, the decomposition
		\eqref{eq:kite-mode-decomposition} gives
		\begin{equation}\label{eq:kite-residual-first}
			t_0\zeta_1+\mathfrak a_p(z,\psi)=\widehat\mu t_0.
		\end{equation}
		Hence $\mathfrak a_p(\psi,z)=-t_0\varepsilon$.
		Testing \eqref{eq:kite-fixed-weak} instead with $z$, and using
		$\displaystyle\int_{\mathcal D_q}\psi z d\sigma=0$, yields
		\[
		t_0\mathfrak a_p(\psi,z)
		+\mathfrak a_p(z,z)
		=
		\widehat\mu\|z\|_2^2.
		\]
		Consequently,
		\begin{equation}\label{eq:kite-residual-identities}
			\mathfrak a_p(\psi,z)=-t_0\varepsilon,
			\qquad
			\mathfrak a_p(z,z)-\widehat\mu\|z\|_2^2
			=t_0^2\varepsilon.
		\end{equation}

		We next estimate the remainder $z$ by splitting it into its fibre average
		and its zero-mean fluctuation. Define
		\begin{equation}\label{eq:kite-mean-fluctuation}
			m(x)=\frac1{g(x)}\int_{\ell(x)}^xz(x,y)dy,
			\qquad r(x,y)=z(x,y)-m(x).
		\end{equation}
		Thus $\displaystyle\int_{\ell(x)}^xr(x,y)dy=0$ for almost every $x$.
		Applying the one-dimensional Poincar\'e inequality on the interval
		$(\ell(x),x)$, whose length is $g(x)$, gives
		\[
		\int_{\ell(x)}^x r(x,y)^2dy\le\frac{g(x)^2}{\pi^2}\int_{\ell(x)}^x z_y(x,y)^2 dy.
		\]
		Since $g(x)\le b$, integration in $x$ gives
		\begin{equation}\label{eq:kite-fibre-fluctuation}
			\|r\|_2^2\le\frac{b^2}{\pi^2}\|z_y\|_2^2
			=\frac{b^2p^2}{\pi^2}E.
		\end{equation}
		
		%		Now we will show that the average $m$ can be controlled by the one-dimensional spectral gap.
		Because the endpoints of the fibres depend on $x$, differentiating
		$m$ produces an additional term. To keep track of this contribution, set
		\[
		Z(x,t)=z(x,\ell(x)+g(x)t),\qquad 0<t<1.
		\]
		Then $m(x)=\displaystyle\int_0^1Z(x,t)dt$, and differentiation gives
		$\partial_x Z=z_x+\beta_0(x,t)z_y$.
		Consequently,
		\begin{equation}\label{eq:kite-average-derivative}
			m'(x)=\frac1{g(x)}\int_{\ell(x)}^x z_x(x,y)dy+B_m(x),
		\end{equation}
		where
		\begin{equation}\label{eq:kite-Bm}
			B_m(x)=\frac1{g(x)}\int_{\ell(x)}^x\beta_0(x,t)z_y(x,y)dy,
			\qquad t=\frac{y-\ell(x)}{g(x)}.
		\end{equation}

		Next we estimate $B_m$. By Cauchy--Schwarz inequality on each fibre,
		we have
		\begin{align*}
			g(x)|B_m(x)|^2
			&\le\left(\frac1{g(x)}\int_{\ell(x)}^x\beta_0(x,t)^2dy\right)
			\int_{\ell(x)}^xz_y(x,y)^2dy \\
			&=\left(\int_{0}^1 \beta_0(x,t)^2dt\right)
			\int_{\ell(x)}^xz_y(x,y)^2dy
			\le C_R\int_{\ell(x)}^xz_y(x,y)^2 dy.
		\end{align*}
		Consequently,
		\begin{equation}\label{eq:kite-fibre-correction}
			\|B_m\|_g^2\le C_R\|z_y\|_2^2=C_Rp^2E.
		\end{equation}
		By \eqref{eq:kite-mean-fluctuation} and Jensen’s inequality on each fibre,
		we have $g(x)|m(x)|^2\le\displaystyle\int_{\ell(x)}^x |z(x,y)|^2dy$,
		and hence
		$\|m\|_g\le \|z\|_2$.
		Similarly, from \eqref{eq:kite-average-derivative},
		so another application of Jensen's inequality gives
		$\|m'-B_m\|_g\le \|z_x\|_2$.
		Therefore, the triangle inequality gives 
		$\|m'\|_g\le\|z_x\|_2+\|B_m\|_g$.
		Thus $m\in H^1_g(0,1)$, and its lift also belongs to
		$H^1(\mathcal D_q)$.
		
		The two orthogonality conditions on $z$ pass to its average:
		\begin{equation}\label{eq:kite-mean-orthogonality}
			\int_0^1gmdx=\int_{\mathcal D_q}zdxdy=0,\qquad
			\int_0^1gm\psi dx=\int_{\mathcal D_q}z\psi dxdy=0.
		\end{equation}
		Hence $m$ is orthogonal, in $L^2(g\,dx)$, both to the constant
		eigenfunction and to the first positive weighted eigenfunction
		$\psi$. The variational characterization of the second
		positive weighted eigenvalue gives
		\begin{equation}\label{eq:kite-mean-gap}
			\|m\|_g^2\le\frac1{\zeta_2(q)}\|m'\|_g^2
			\le\frac1G\|m'\|_g^2.
		\end{equation}

		Combining these facts, we have
		\begin{align}\label{eq:kite-remainder-lower}
			\mathfrak a_p(z,z)-\widehat\mu\|z\|_2^2
			&=(\|z_x\|_2^2+E)-\widehat\mu(\|m\|_g^2+\|r\|_2^2)\notag\\*
			&\ge\|m'-B_m\|_g^2-\frac{\widehat\mu}{G}\|m'\|_g^2
			+\left(1-\frac{\widehat\mu b^2p^2}{\pi^2}\right)E.
		\end{align}
		Since $\widehat\mu\le\zeta_1<G$, completing the square gives
		\begin{align*}
			\|m'-B_m\|_g^2-\frac{\widehat\mu}{G}\|m'\|_g^2
			&=\left(1-\frac{\widehat\mu}{G}\right)
			\left\|m'-\frac{B_m}{1-\widehat\mu/G}\right\|_g^2
			-\frac{\widehat\mu}{G-\widehat\mu}\|B_m\|_g^2\\
			&\ge-\frac{\widehat\mu}{G-\widehat\mu}\|B_m\|_g^2.
		\end{align*}
		Using \eqref{eq:kite-fibre-correction} in
		\eqref{eq:kite-remainder-lower}, we obtain the coercivity estimate
		\begin{equation}\label{eq:kite-remainder-coercivity}
			\mathfrak a_p(z,z)-\widehat\mu\|z\|_2^2
			\ge\left[1-\widehat\mu p^2
			\left(\frac{b^2}{\pi^2}+\frac{C_R}{G-\widehat\mu}\right)\right]E
			=\mathscr D E.
		\end{equation}
		
		We now estimate the coupling term in the first identity of
		\eqref{eq:kite-residual-identities}. Since $\psi$ has no vertical
		dependence, \eqref{eq:kite-average-derivative} gives
		\begin{align}\label{eq:kite-coupling}
			\mathfrak a_p(\psi,z)
			&=\int_0^1\psi'(x) dx
			\int_{\ell(x)}^xz_x(x,y)dy\notag\\*
			&=\int_0^1g\psi'(m'-B_m)dx
			=-\int_0^1g\psi'B_m dx\\*
			&=-\int_{\mathcal D_q}\psi'(x)\beta_0(x,t)z_y(x,y)dxdy.
		\end{align}
		Here $\displaystyle\int_{0}^{1}g\psi'(x)m'dx=\zeta_1\int_0^1 g\psi mdx=0$
		follows by testing the weighted weak equation with $m\in H^1_g(0,1)$.
		Thus Cauchy--Schwarz and the definition of $C_{\rm eff}$ give
		\begin{equation}\label{eq:kite-coupling-bound}
			|\mathfrak a_p(\psi,z)|^2
			\le\left(\int_{\mathcal D_q}\psi'(x)^2\beta_0(x,t)^2 dxdy \right)
			\left(\int_{\mathcal D_q} z_y^2dxdy\right)
			=C_{\rm eff}\zeta_1 p^2 E.
		\end{equation}
		
		It remains to eliminate the spectral residual $\varepsilon$.
		First, $E>0$. Indeed, otherwise $u$ would depend only on $x$, and the
		Neumann condition on $y=px$ would force $u_x=0$, contradicting
		its positive eigenvalue and normalization.
		Combining \eqref{eq:kite-remainder-coercivity} with the second
		identity in \eqref{eq:kite-residual-identities} gives
		\begin{equation}\label{eq:kite-residual-coercive-comparison}
			0<\mathscr D E\le	t_0^2\varepsilon.
		\end{equation}
		In particular, $t_0\ne0$ and $\varepsilon>0$ follow from the
		assumed positivity of $\mathscr D$. The first identity in \eqref{eq:kite-residual-identities}
		and \eqref{eq:kite-coupling-bound} give
		\begin{align}\label{eq:kite-residual-coupling-comparison}
			t_0^2\varepsilon^2=|\mathfrak a_p(\psi,z)|^2 \le C_{\rm eff}\zeta_1 p^2E.
		\end{align}
		Multiplying \eqref{eq:kite-residual-coercive-comparison} by $\varepsilon$
		and combining it with \eqref{eq:kite-residual-coupling-comparison}, we obtain
		\[
		\mathscr D E\varepsilon
		\le t_0^2\varepsilon^2\le C_{\rm eff}\zeta_1 p^2E.
		\]
		Canceling $E>0$ yields
		\[
		1-\rho=\frac{\varepsilon}{\zeta_1}
		\le\frac{C_{\rm eff}p^2}{\mathscr D}.
		\]
		Finally, for the second estimate, using $t_0^2\le1$ in
		\eqref{eq:kite-residual-coercive-comparison} gives
		\[
		E\le\frac{t_0^2\varepsilon}{\mathscr D}
		\le\frac\varepsilon{\mathscr D}
		\le\frac{C_{\rm eff}\zeta_1 p^2}{\mathscr D^2}.
		\]
		Dividing by $\widehat\mu=\rho\zeta_1>0$ gives the second estimate
		in \eqref{eq:kite-residual-estimates}.
	\end{proof}

	\begin{Lemma}\label{prop:kite-large-a-transverse}
		Let $3/2\le a<2$, $h^2\le a(2-a)$ and $\Theta(a,h)\le1/8$.
		Then
		\begin{equation}\label{eq:kite-neumann-separation}
			\frac{E_N}{\mu(a,h)}<\Theta(a,h).
		\end{equation}
		In particular, every crossing in this range satisfies $E_D>E_N$.
	\end{Lemma}
	\begin{proof}
		By Lemma \ref{lem:kite-weighted-spectrum}, 
		we choose $G=56$. Set $\rho=\widehat\mu/\zeta_1$, by
		\eqref{eq:kite-mu-zeta}, 
		$0<\rho\le1$. We first verify that the constant $\mathscr D$ in
		Lemma \ref{lem:kite-residual} is uniformly positive. Recall that
		\[
		\mathscr D=1-\widehat\mu p^2
		\left(\frac{b^2}{\pi^2}+\frac{C_R}{G-\widehat\mu}\right),
		\qquad
		C_R=\frac{1+q^{-1}+q^{-2}}3.
		\]
		Using $\widehat\mu=\rho\zeta_1$ and Lemma \ref{lem:kite-weighted-spectrum},  we have
		\[
		\frac{\widehat\mu b^2}{\pi^2}
		=\frac{\zeta_1 }{\pi^2} \rho b^2<
		\frac{140}{57}\rho b^2.
		\]
		Furthermore, Lemma \ref{lem:kite-weighted-spectrum},
		\[
		\frac{\widehat\mu}{G-\widehat\mu}
		=
		\frac{\rho\zeta_1}{G-\widehat\mu}
		\le\frac{\rho\zeta_1}{G-\zeta_1}<
		\frac57\rho.
		\]
		Consequently,
		\begin{equation}\label{eq:kite-D-lower}
			\mathscr D\ge1-\rho K,
		\end{equation}
		where
		\begin{equation}\label{eq:kite-K-definition}
			K=p^2\left[\frac{140}{57}(1-q)^2
			+\frac5{21}(1+q^{-1}+q^{-2})\right].
		\end{equation}

		We now estimate $K$. By
		Lemma \ref{lem:kite-geometric-window},
		\begin{equation}\label{eq:kite-p-refined-use}
			p^2<
			\frac{5q^2}{8(1+q)^2}.
		\end{equation}
		Substituting this into \eqref{eq:kite-K-definition} gives
		\begin{align}
			K
			&<
			\frac{175}{114}
			\left(
			\frac{q(1-q)}{1+q}
			\right)^2
			+
			\frac{25}{168}
			\frac{1+q+q^2}{(1+q)^2}.
			\label{eq:kite-K-first-bound}
		\end{align}
		For $1/3\le q\le1/2$, one has $\frac{q(1-q)}{1+q}<\frac{11}{64}$ and $\frac{1+q+q^2}{(1+q)^2}\le\frac{13}{16}$.
		Hence
		\begin{align}\label{eq:kite-K-bound}
			K&< \frac{175}{114}\left(\frac{11}{64}\right)^2 +\frac{25}{168}\frac{13}{16}<\frac16.
		\end{align}
		Combining \eqref{eq:kite-D-lower} and
		\eqref{eq:kite-K-bound}, we obtain
		\begin{equation}\label{eq:kite-D-positive}
			\mathscr D
			\ge1-\rho K
			>
			1-\frac{\rho}{6}
			\ge\frac56>0.
		\end{equation}

		We next improve the crude bound $0<\rho\le1$. By
		Lemma \ref{lem:kite-effective-slope},
		\[
		C_{\rm eff}\le\frac43,
		\]
		while Lemma \ref{lem:kite-geometric-window} gives $p<1/4$ and hence
		$p^2<1/16$. The first estimate in
		\eqref{eq:kite-residual-estimates}, together with
		\eqref{eq:kite-D-positive}, therefore yields
		\[
		1-\rho\le\frac{C_{\rm eff}p^2}{\mathscr D}<\frac1{10},
		\qquad \rho>\frac9{10}.
		\]
		
		We now apply the second estimate in
		\eqref{eq:kite-residual-estimates}. From \eqref{eq:kite-D-lower} and $K<1/6$,
		we have $\mathscr D>1-\rho/6$, then
		\[
		\frac E{\widehat\mu} \le \frac{C_{\rm eff}p^2}
		{\rho\mathscr D^2}
		\le\frac{(4/3)p^2}{\rho(1-\rho/6)^2}
		<\frac{(4/3)p^2}{(9/10)(17/20)^2}
		=\frac{16000}{7803}p^2< \frac{33}{16}p^2.
		\]
		where we used that $r(1-r/6)^2$ is increasing on $[9/10,1]$.
		Lemma~\ref{lem:kite-geometric-window}
		and \eqref{eq:kite-similarity} give
		\eqref{eq:kite-neumann-separation}.
		
		Finally, suppose a crossing had $E_D\le E_N$.
		Lemmas~\ref{lem:kite-large}, \ref{lem:kite-eighth} and
		Proposition~\ref{prop:kite-mixed-energy} give
		$h^2\le a(2-a)$, $h^2<ac$ and $\Theta<1/8$.
		At that crossing the two strict estimates imply
		\[
		\frac{E_D}{\Lambda}>\Theta(a,h)>\frac{E_N}{\Lambda},
		\]
		a contradiction.
	\end{proof}

	We now prove Theorem~\ref{thm1.2}.
	\begin{proof} [Proof of Theorem \ref{thm1.2}]
		Let $T=T_{a,h}=K_{a,h}\cap\{y>0\}$ be the upper half-domain.
		Its second Neumann eigenvalue $\mu(a,h)$ and first mixed eigenvalue
		$\lambda(a,h)$ are defined in \eqref{eq:kite-spectra}.

		(1) If $0<a<1$  and $a\ne \frac{1}{2}$, Lemma \ref{lem:kite-large} shows that
		$\lambda(a,h)<\mu(a,h)$ holds for sufficiently large $h$,
		and \cite[Lemma 1]{Siudeja2015MZ} shows that
		$\mu(a,h)<\lambda(a,h)$ holds for sufficiently small $h$.
		Hence, there exists at least one $\widetilde h(a)$ such that
		$\mu(a,\widetilde h(a))=\lambda(a,\widetilde h(a))$. Moreover,
		Lemma \ref{lem:kite-convex-transverse} shows that $\lambda(a,h)$
		crosses $\mu(a,h)$ from above to below at the intersection point.
		Thus there exists exactly one critical height $\widetilde h(a)$
		such that $\mu(a,\widetilde h(a))=\lambda(a,\widetilde h(a))$.
		
		If $a=1/2$, then $K_{a,h}$ reduces to a rhombus whose diagonals have lengths $1$ and $2h$.
		Proposition \ref{rhombus} gives symmetry about the longer
		diagonal and antisymmetry about the shorter one. Thus the second Neumann eigenfunction is
		symmetric about the $x$-axis for $h<1/2$ and antisymmetric about the $x$-axis for $h>1/2$.
		At $h=1/2$, the domain is a square.
		Consequently, $\widetilde h(1/2)=1/2$.
		
		(2) If $a=1$, then $K_{a,h}$ reduces to an isosceles triangle.
		The symmetry classification for
		subequilateral and superequilateral triangles in
		\cite[Theorems 3.1 and 3.2]{Laugesen2010JDE} shows that
		if $h<1/\sqrt3$, $\mu_2$ is simple, $u$ is symmetric about the
		$x$-axis; if $h>1/\sqrt3$,
		$\mu_2$ is simple and $u$ is antisymmetric; if $h=1/\sqrt3$, the triangle is
		equilateral and $\mu_2$ has multiplicity exactly two, with one symmetric
		and one antisymmetric basis eigenfunction. Thus $\widetilde h(1)=1/\sqrt3$.
		
		(3) If $1<a\le \frac32$, Lemma \ref{lem:kite-large} shows that
		$\lambda(a,h)<\mu(a,h)$ holds for sufficiently large $h$,
		and Lemma \ref{lem:kite-thin-limit} shows that
		$\mu(a,h)<\lambda(a,h)$ holds for sufficiently small $h$.
		Hence, there exists at least one $\widetilde h(a)$ such that
		$\mu(a,\widetilde h(a))=\lambda(a,\widetilde h(a))$. Moreover,
		Lemma \ref{prop:kite-small-a-transverse} shows that $\lambda(a,h)$
		crosses $\mu(a,h)$ from above to below at the intersection point.
		Thus there exists exactly one critical height $\widetilde h(a)$
		such that $\mu(a,\widetilde h(a))=\lambda(a,\widetilde h(a))$.
		
		(4) If $\frac32<a<2$, Lemma \ref{lem:kite-large} shows that
		$\lambda(a,h)<\mu(a,h)$ holds for sufficiently large $h$,
		and Lemma \ref{lem:kite-thin-limit} shows that
		$\mu(a,h)<\lambda(a,h)$ holds for sufficiently small $h$.
		Hence, there exists at least one $\widetilde h(a)$ such that
		$\mu(a,\widetilde h(a))=\lambda(a,\widetilde h(a))$. Moreover,
		Lemma \ref{prop:kite-large-a-transverse} shows that $\lambda(a,h)$
		crosses $\mu(a,h)$ from above to below at the intersection point.
		Thus there exists exactly one critical height $\widetilde h(a)$
		such that $\mu(a,\widetilde h(a))=\lambda(a,\widetilde h(a))$.

		(5) If $a\ge2$, Lemma \ref{lem:kite-large} shows that
		$\lambda(a,h)<\mu(a,h)$ holds for all $h>0$,
		hence the second Neumann eigenfunctions are antisymmetric about the $x$-axis.
		
	\end{proof}

\section{Distribution of critical points}\label{sec:critical-points}

	\textcolor{black}{Sections~\ref{sec:trapezoids} and \ref{sec:kites} establish
	the spectral and symmetry classifications in Theorems~\ref{thm1.1}
	and \ref{thm1.2}. This section proves the absence of interior critical
	points asserted in Theorem~\ref{thm1.3} and justifies the non-vertex
	boundary classification in Corollary~\ref{cor1.4}.} We first treat the simple
	eigenvalue case, separating interior points from non-vertex boundary
	points. At a critical height, the symmetric and antisymmetric eigenfunctions span the second
	eigenspace, so the remaining task is to exclude interior critical
	points of their nontrivial linear combinations. This is done for
	isosceles trapezoids in Subsection~\ref{subsec:trapezoid-multiplicity-two}
	and for convex and concave kites in
	Subsection~\ref{subsec:kite-multiplicity-two}.
	We use the geometric notation of Figure~\ref{fig:review-geometry}
	throughout. Boundary classifications below concern non-vertex points;
	no differentiability at a reentrant vertex is assumed.

	The following comparison lemma follows from the nodal-domain argument of
	Siudeja \cite[Section~3.2, proof of Lemma~2]{Siudeja2015MZ}.
	It excludes nodal critical points and will be used for both
	isosceles trapezoids and convex kites at the critical height.
	\begin{Lemma}\label{lem.criticalpoint-comparison}
		Let $\Omega$ be a bounded convex polygon and
		$\Lambda=\mu_2(\Delta,\Omega)<\lambda_1(\Delta,\Omega)$.
		Let $\psi\in H^1(\Omega)$ be a nonzero real weak solution of
		$(\Delta+\Lambda)\psi=0$ in $\Omega$.
		Assume that $\psi$ is continuously differentiable up to each open side
		and satisfies
		\begin{equation}\label{eq:nodal-comparison-boundary}
			\partial_n\psi\le0\quad\text{on each open side of }\Omega.
		\end{equation}
		Then there is no $p\in\Omega$ such that $\psi(p)=0$ and $\nabla \psi(p)=0$.
	\end{Lemma}
	\begin{proof}
		We first introduce cutoff functions that vanish near the vertices.
		Choose $0<\varepsilon<1$ sufficiently small that the disks
		$B_\varepsilon(P)$, where $P$ ranges over all vertices of $\Omega$,
		are pairwise disjoint.
		For each vertex $P$, set $\chi_\varepsilon(z)=0$ when
		$|z-P|\le\varepsilon^2$, and define
		\begin{equation}\label{eq:nodal-cutoff-definition}
			\chi_\varepsilon(z)=\frac{\log(|z-P|/\varepsilon^2)}{\log(1/\varepsilon)}
			\quad\text{if }\varepsilon^2<|z-P|<\varepsilon.
		\end{equation}
		Set $\chi_\varepsilon=1$ outside these disks. Then $0\le\chi_\varepsilon\le1$,
		and \eqref{eq:nodal-cutoff-definition} gives
		\begin{equation}\label{eq:nodal-cutoff-energy}
			\int_\Omega|\nabla\chi_\varepsilon|^2\,d\sigma
			\le\frac{C}{\log^2(1/\varepsilon)}
			\int_{\varepsilon^2}^{\varepsilon}\frac{dr}{r}
			=\frac{C}{\log(1/\varepsilon)}\longrightarrow0.
		\end{equation}
		For any smooth nonnegative $\eta$ on $\overline\Omega$,
		$\eta\chi_\varepsilon$ vanishes near all vertices, so integration by parts gives
		\begin{equation}\label{eq:nodal-cutoff-integration}
			\int_\Omega\bigl(\nabla\psi\cdot\nabla(\eta\chi_\varepsilon)
			-\Lambda\psi\eta\chi_\varepsilon\bigr)\,d\sigma
			=\int_{\partial\Omega}\eta\chi_\varepsilon\partial_n\psi\,ds\le0.
		\end{equation}
		Since $\chi_\varepsilon\to1$ almost everywhere and $0\le\chi_\varepsilon\le1$,
		\eqref{eq:nodal-cutoff-energy} and dominated convergence give
		$\eta\chi_\varepsilon\to\eta$ in $H^1(\Omega)$.
		Since $\psi\in H^1(\Omega)$, letting $\varepsilon\downarrow0$ in
		\eqref{eq:nodal-cutoff-integration} and then approximating nonnegative
		$H^1(\Omega)$ functions by smooth nonnegative functions yield
		\begin{equation}\label{eq:weak-nodal-comparison}
			\int_\Omega(\nabla\psi\cdot\nabla\eta-\Lambda\psi\eta)\,d\sigma
			\le0
			\qquad(0\le\eta\in H^1(\Omega)).
		\end{equation}
		Suppose such a point $p$ exists. Interior analyticity gives at least four nodal arcs issuing from $p$, with alternating positive and negative sectors. We explain why two positive sectors cannot belong to the same global nodal domain, following \cite[Lemmas~6--7 and the proof of Lemma~8]{Siudeja2015MZ}. If they were connected in $\{\psi>0\}$, a simple connecting arc, together with two arcs approaching $p$ through these sectors, would form a Jordan curve enclosing a negative sector. Since $\Omega$ is convex, its interior lies in $\Omega$. The negative nodal domain $G$ in that sector cannot cross this curve, on which $\psi$ is positive except at $p$. Thus $G$ is compactly contained in $\Omega$, and the restriction of $-\psi$ to $G$, extended by zero, is a nonzero element of $H_0^1(\Omega)$ with Rayleigh quotient $\Lambda$. This contradicts $\Lambda<\lambda_1(\Delta,\Omega)$. Hence there are at least two distinct positive nodal domains $G_1,G_2$. This argument uses the interior equation and the Dirichlet spectral bound, not Courant's theorem for $\psi$, which need not satisfy Neumann boundary conditions.

		Set $\psi_i=\psi\,\mathbf1_{G_i}$ and
		$M_i=\int_\Omega \psi_i\,d\sigma>0$ for $i=1,2$.
		The zero extensions $\psi_i$ belong to $H^1(\Omega)$, with
		$\nabla\psi_i=\mathbf1_{G_i}\nabla\psi$ almost everywhere;
		this follows by approximating them with
		$(\psi-\varepsilon)^+\mathbf1_{G_i}$ as $\varepsilon\downarrow0$.
		Testing \textcolor{black}{\eqref{eq:weak-nodal-comparison}} with $\psi_i$ gives
		\begin{equation}\label{eq:trap-nodal-energy}
			Q(\psi_i):=\int_\Omega(|\nabla \psi_i|^2-\Lambda \psi_i^2)\,d\sigma\le0.
		\end{equation}
		Set $z=M_2\psi_1-M_1\psi_2$. The disjoint supports give
		$z\not\equiv0$ and $\int_\Omega z\,d\sigma=0$.
		By \textcolor{black}{\eqref{eq:trap-nodal-energy}} and the variational characterization of
		$\mu_2=\Lambda$ (cf.\ \cite[Lemma~9]{Siudeja2015MZ}),
		\begin{equation}\label{eq:trap-nodal-zero-energy}
			0\le Q(z)=M_2^2Q(\psi_1)+M_1^2Q(\psi_2)\le0,
		\end{equation}
		As in \cite[Section~3.2, proof of Lemma~2]{Siudeja2015MZ},
		\eqref{eq:trap-nodal-zero-energy} makes $z$ a Neumann eigenfunction
		with eigenvalue $\Lambda$. Since $z$ vanishes on the nonempty open set
		$\{\psi<0\}$, unique continuation contradicts $z\not\equiv0$.
	\end{proof}

	\subsection{The simple eigenvalue case}\label{subsec:simple-critical-points}

	When the second Neumann eigenvalue is simple, reflection symmetry
	forces its eigenfunction to be either symmetric or antisymmetric.
	Its restriction to a half-domain is therefore a second Neumann
	eigenfunction or a first mixed eigenfunction, respectively. We record
	the interior conclusion for either symmetry type, including at a
	critical height. This will also handle the cases in
	which one coefficient vanishes in the subsequent multiplicity-two
	proofs.
\subsubsection{Absence of interior critical points}\label{subsubsec:simple-interior}

	\begin{Proposition}\label{prop:pure-modes-interior}
		Let $u$ be a nonzero second Neumann eigenfunction on an isosceles
		trapezoid or a kite in Theorems~\ref{thm1.1} and \ref{thm1.2}.
		If $u$ is symmetric or antisymmetric with respect to the indicated
		axis, then $u$ has no interior critical point.
	\end{Proposition}
	\begin{proof}
		For an isosceles trapezoid, let $R=R_{\alpha,h}$ be its right half.
		If $u$ is symmetric, Lemma~\ref{lem.y}, after a sign choice, gives
		$u_y<0$ in $R$ and on the open symmetry segment $OC$.
		If $u$ is antisymmetric, Lemma~\ref{lem.x} gives $u_x>0$
		in $R$ and on $OC$. Thus the gradient cannot vanish in the right
		half or on the symmetry segment. Reflection gives the same
		conclusion in the left half.

		For a kite, let $T=T_{a,h}$ be its upper half. Suppose first that
		$u$ is symmetric. Its restriction is a second Neumann eigenfunction
		on $T$. For a non-equilateral $T$, the directional monotonicity in
		\cite[Theorem~1.2(5)]{Chen2026Invent} excludes interior critical
		points of this restriction. To exclude points on the open symmetry
		segment $P_1P_3$, use
		\cite[Theorem~1.2(1)--(2)]{Chen2026Invent}: any non-vertex critical
		point of the restriction must be on a shortest side.
		By Lemma~\ref{lem:kite-large}, $P_1P_3$ cannot be a shortest side:
		otherwise $\lambda(a,h)<\mu(a,h)=\mu_2(K_{a,h})$.
		This also excludes an equilateral half-triangle from
		this symmetric case. Hence the gradient is nonzero throughout
		the interior of the full kite.

		If $u$ is antisymmetric, choose its restriction $\nu$ positive in $T$.
		For $a>1$, Lemma~\ref{lem:kite-mixed-monotonicity} gives
		$\nu_x>0$ and $\nu_y>0$ in $T$. For $0<a\le1$, when the Neumann
		vertex $P_4$ is non-obtuse, the inward-normal monotonicity
		of the positive mixed eigenfunction gives $\nu_y>0$ in $T$
		\cite{Hatcher2024PAMS,Li2026JFA}.
		When that vertex is obtuse, the mixed critical-point classification
		in \cite{Li2026JFA} still excludes interior critical points:
		any non-vertex critical point must lie on a Neumann side.
		On the open Dirichlet side $P_1P_3$, the Hopf lemma gives
		$\nu_y>0$. Thus neither the upper half nor the open symmetry
		segment contains a critical point; odd reflection treats the
		lower half.
	\end{proof}

	\subsubsection{Non-vertex boundary critical points}\label{subsubsec:simple-boundary}

	\begin{proof}[\textcolor{black}{Proof of Corollary~\ref{cor1.4}}]
		For an isosceles trapezoid and an antisymmetric $u$,
		Lemma~\ref{lem.x} excludes critical points
		on the open edges of $R_{\alpha,h}$. Reflection and
		Proposition~\ref{prop:pure-modes-interior} give
		$\mathcal C_{nv}(u)=\varnothing$.

		If $u$ is symmetric, Lemma~\ref{lem.y} excludes critical points
		on those open edges. The only additional non-vertex boundary
		points of the full trapezoid are the base midpoints $O$ and $C$.
		At each of them, symmetry gives $u_x=0$ and the Neumann
		condition gives $u_y=0$. Hence
		$\mathcal C_{nv}(u)=\{O,C\}$.

		For a kite, first let $a\ne1$.
		If $u$ is symmetric, its restriction to $T$ is a second Neumann
		eigenfunction. By \cite[Theorem~1.2(1)--(2)]{Chen2026Invent},
		it has a non-vertex critical point precisely when $T$ is acute
		and non-superequilateral, and that point lies on a shortest side.
		By Lemma~\ref{lem:kite-large}, $P_1P_3$ cannot be a shortest side
		when $u$ is a symmetric second eigenfunction of the kite.
		Thus the point lies on a shorter outer edge. Reflection gives
		exactly the two asserted non-vertex critical points of the kite.

		If $u$ is antisymmetric, its restriction is a positive first mixed
		eigenfunction with Dirichlet side $P_1P_3$ and Neumann vertex $P_4$.
		The mixed critical-point classification
		\cite{Hatcher2024PAMS,Li2026JFA} gives no non-vertex critical
		point when the angle at $P_4$ is non-obtuse. When that angle is
		obtuse and the Neumann sides have unequal lengths, there is
		exactly one such point, on the longer Neumann side; when the
		sides have equal lengths, there is none.
		The Hopf lemma excludes the open Dirichlet side.
		Reflecting the point, if present, across the $x$-axis gives exactly
		the two asserted critical points on the longer outer edges.

	\end{proof}

	\textcolor{black}{This proves Theorem~\ref{thm1.3} in the simple-eigenvalue case,
	as well as the boundary classification in Corollary~\ref{cor1.4}.}
	At a critical height,
	Proposition~\ref{prop:pure-modes-interior} already treats symmetric
	and antisymmetric eigenfunctions, but does not exclude cancellation
	between their gradients in a linear combination.
	The following subsections address that additional issue.

	\subsection{Isosceles trapezoids at the critical height}\label{subsec:trapezoid-multiplicity-two}

	We use the notation in Figure~\ref{fig:review-geometry}(a) and write
	$\Omega=Q_{\alpha,h}$ and $R=R_{\alpha,h}=OP_2P_3C$.
	\textcolor{black}{For the critical-height trapezoids in Theorem~\ref{thm1.1}(2)(c),}
	we show that the gradients of the symmetric and antisymmetric
	eigenfunctions are linearly independent at every interior point.

	\begingroup\color{black}
	Let $h={\color{black}\hat h}(\alpha)$, where $\pi/3<\alpha<\pi/2$, and put
	$\Lambda=\mu_2(\Omega)=\mu(\alpha,h)=\lambda(\alpha,h)$.
	By Theorem~\ref{thm1.1}, the second Neumann eigenspace on $\Omega$
	has a basis $\varphi,\nu$ that is symmetric and antisymmetric,
	respectively, with respect to the $y$-axis.

	We use the same symbols for their restrictions to the right half $R$.
	On $R$, $\varphi$ is a second Neumann eigenfunction, whereas $\nu$
	is a first mixed eigenfunction with Dirichlet condition on $OC$
	and Neumann conditions on the other three sides.
	Normalize both restrictions in $L^2(R)$ and choose $\nu>0$ in $R$.
	Choose the sign of $\varphi$ as in Lemma~\ref{lem.y}; then
	Lemmas~\ref{lem.y} and \ref{lem.x} give
	$\varphi_x,\nu_x>0$ and $\varphi_y,\nu_y<0$ in $R$.

	Every nonzero second Neumann eigenfunction on $\Omega$ can therefore
	be written as $u=a_1\varphi+a_2\nu$, with $(a_1,a_2)\ne(0,0)$.
	\endgroup

	Set
	\begin{equation}\label{eq.add7.trap-J}
		J:=\det(\nabla\varphi,\nabla\nu)
		=\varphi_x\nu_y-\varphi_y\nu_x.
	\end{equation}
	It suffices to prove $J>0$ in $\Omega$.
	Indeed, this makes $\nabla\varphi$ and $\nabla\nu$ linearly independent
	at every interior point. Since $\varphi,\nu$ form a basis of the second
	Neumann eigenspace, every nonzero eigenfunction $u=a_1\varphi+a_2\nu$
	has $\nabla u=a_1\nabla\varphi+a_2\nabla\nu\ne0$ in $\Omega$.

	Throughout this subsection, $n$ denotes the outward unit normal.
	On $P_2P_3$, let $\tau=(-\cos\alpha,\sin\alpha)$ be the unit tangent
	pointing from $P_2$ toward $P_3$; then $n=(\sin\alpha,\cos\alpha)$.
	\begingroup\color{black}
	The corner expansions and regularity estimates used below are recorded
	in Lemma~\textcolor{black}{\ref{lem:corner-bessel-expansions}} and
	Remark~\ref{rem:corner-regularity}.
	\endgroup

	We first establish a second-derivative inequality for $\nu$.
	\phantomsection\label{sub.ext.9}
	\begin{Lemma}\label{lem:trap-second-derivative}
		At $h={\color{black}\hat h}(\alpha)$, the positive mixed eigenfunction $\nu$ satisfies
		\[
			\nu_{xx}+\frac{\Lambda}{2}\nu<0\qquad\text{in }R\cup P_2P_3.
		\]
	\end{Lemma}
	\begin{proof}
		Lemmas~\ref{lem.y} and \ref{lem.x} give $\nu_x>0>\nu_y$ for the positive mixed
		eigenfunction $\nu$ at the critical height. Put $F=\nu_{xx}+\Lambda\nu/2$.
		Then $F$ satisfies
		\begin{equation}\label{eq:trap-F-boundary-problem}
			(\Delta+\Lambda)F=0\quad\text{in }R,\qquad
			F=0\quad\text{on }OC,\qquad
			\partial_nF=0\quad\text{on }OP_2\cup CP_3.
		\end{equation}
		On $P_2P_3$, differentiating the boundary condition
		$\nu_n=0$ once and twice in the tangential direction $\tau$, we obtain
		$\nu_{n\tau}=\nu_{n\tau\tau}=0$.
		The equation and its normal derivative then yield
		$\nu_{nn}=-\Lambda\nu-\nu_{\tau\tau}$ and $\nu_{nnn}=0$.
		Using $\partial_x=-\cos\alpha\,\partial_\tau+\sin\alpha\,\partial_n$,
		we differentiate in the interior and then take boundary values on $P_2P_3$
		to obtain
		\begin{align}
			F&=\cos^2\alpha\,\nu_{\tau\tau}
			+\sin^2\alpha\,\nu_{nn}+\Lambda\nu/2
			=-(\sin^2\alpha-\cos^2\alpha)(\nu_{\tau\tau}+\Lambda\nu/2),
			\label{eq.add7.trap-boundary-F}\\
			F_n&=-2\sin\alpha\cos\alpha\,\nu_{\tau nn}
			=2\sin\alpha\cos\alpha(\nu_{\tau\tau\tau}+\Lambda\nu_\tau),\label{eq:trap-F-normal}\\
			F_\tau&=-(\sin^2\alpha-\cos^2\alpha)
			(\nu_{\tau\tau\tau}+\Lambda\nu_\tau/2).\label{eq:trap-F-tangential}
		\end{align}
		Set
		\begin{equation}\label{eq:trap-oblique-beta}
			{\color{black}\chi}=\frac{2\sin\alpha\cos\alpha}{\sin^2\alpha-\cos^2\alpha}>0,
		\end{equation}
		since $\pi/3<\alpha<\pi/2$.
		Combining \textcolor{black}{\eqref{eq:trap-F-normal}} and
		\textcolor{black}{\eqref{eq:trap-F-tangential}} cancels
		$\nu_{\tau\tau\tau}$ and gives
		\begin{equation}\label{eq:trap-F-oblique}
			F_n+{\color{black}\chi} F_\tau=\Lambda\sin\alpha\cos\alpha\,\nu_\tau<0,
		\end{equation}
		because $\nu_\tau=-\cos\alpha\,\nu_x+\sin\alpha\,\nu_y<0$.

		We estimate $F$ near $P_2,P_3$ before testing
		\eqref{eq:trap-F-boundary-problem} with $F_+=\max\{F,0\}$.
		At $P_2$, introduce polar coordinates
		$(r,\vartheta)$ centered at $P_2$,
		with $\vartheta=0$ on the ray $P_2O$ and $\vartheta$ increasing clockwise
		toward $P_2P_3$, so that
		$0\le\vartheta\le\alpha$ and
		$x=1-r\cos\vartheta$, $y=r\sin\vartheta$.
		Put ${\color{black}\delta}=\pi/\alpha\in(2,3)$. Applying
		\textcolor{black}{\eqref{eq.add7.corner-first}} with $w=\nu$ gives
		\begin{equation}\label{eq:trap-P2-nu-expansion}
			\nu=c_0\left(1-\frac{\Lambda r^2}{4}\right)
			+c_1r^{\color{black}\delta}\cos({\color{black}\delta}\vartheta)+O(r^4),
		\end{equation}
		where $c_0=\nu(P_2)$ and $c_1$ is the first angular coefficient at $P_2$.
		All corner remainders used in this subsection are uniform in the angle
		and may be differentiated according to
		\textcolor{black}{\eqref{eq:corner-remainder-derivatives}}.
		Differentiating \eqref{eq:trap-P2-nu-expansion}, the constant terms in
		$\nu_{xx}+\Lambda\nu/2$ cancel, giving
		\begin{equation}\label{eq:trap-P2-second-expansion}
			F=c_1{\color{black}\delta}({\color{black}\delta}-1)r^{{\color{black}\delta}-2}\cos(({\color{black}\delta}-2)\vartheta)+O(r^2),
			\qquad |\nabla F|=O(r^{{\color{black}\delta}-3}).
		\end{equation}
		In particular, $F\to0$ uniformly for $0\le\vartheta\le\alpha$.
		At $P_3$, introduce polar coordinates $(r,\vartheta)$ centered at $P_3$,
		with $\vartheta=0$ on the ray $P_3C$ and $\vartheta$ increasing
		counterclockwise toward $P_3P_2$, so that
		$0\le\vartheta\le\pi-\alpha$ and
		$x=s-r\cos\vartheta$, $y=h-r\sin\vartheta$.
		Put ${\color{black}\delta}=\pi/(\pi-\alpha)\in(3/2,2)$. The same formula
		\textcolor{black}{\eqref{eq.add7.corner-first}}, with $w=\nu$, now gives
		\begin{equation}\label{eq:trap-P3-nu-expansion}
			\nu=c_0\left(1-\frac{\Lambda r^2}{4}\right)
			+c_1r^{\color{black}\delta}\cos({\color{black}\delta}\vartheta)+O(r^{2{\color{black}\delta}}),
		\end{equation}
		where $c_0,c_1$ are the corresponding coefficients at $P_3$.
		Here $c_0=\nu(P_3)>0$, since $\nu>0$ on the open edge $CP_3$ and
		$\nu$ is strictly increasing along this edge from $C$ toward $P_3$
		by Lemma~\ref{lem.x}. On the top edge $\vartheta=0$,
		\[
			0<\nu_x=-\delta c_1r^{\delta-1}+O(r),\qquad
			r^{1-\delta}\nu_x\longrightarrow-\delta c_1.
		\]
		Thus $c_1\le0$. If $c_1=0$, the truncated expansion
		\eqref{eq:trap-P3-nu-expansion} gives
		$\nu=c_0(1-\Lambda r^2/4)+O(r^{2{\color{black}\delta}})$, since ${\color{black}\delta}<2$.
		In these polar coordinates,
		$\partial_y=-\sin\vartheta\,\partial_r
		-r^{-1}\cos\vartheta\,\partial_\vartheta$.
		Applying this operator to the expansion, with its differentiable
		remainder, yields along any fixed ray $0<\vartheta<\pi-\alpha$
		\[
			\nu_y=\frac{\Lambda c_0}{2}r\sin\vartheta
			+O(r^{2{\color{black}\delta}-1})>0
			\quad\text{for sufficiently small }r,
		\]
		since $2{\color{black}\delta}-1>1$. This contradicts $\nu_y<0$, so $c_1<0$.
		The same differentiation of \eqref{eq:trap-P3-nu-expansion}, now with
		$3/2<{\color{black}\delta}<2$, gives
		\begin{equation}\label{eq.add7.trap-upper-coefficient}
			c_1<0,\qquad
			F=c_1{\color{black}\delta}({\color{black}\delta}-1)r^{{\color{black}\delta}-2}\cos(({\color{black}\delta}-2)\vartheta)+O(r^{2{\color{black}\delta}-2}).
		\end{equation}
		Here $r^{2{\color{black}\delta}-2}=o(r^{{\color{black}\delta}-2})$ and
		$|({\color{black}\delta}-2)\vartheta|\le\pi-2\alpha<\pi/3$.
		Thus $F\to-\infty$ as $r\downarrow0$, uniformly for
		$0\le\vartheta\le\pi-\alpha$.
		Thus $F_+$ vanishes near $P_3$ and at $P_2$.
		At $P_2$, \textcolor{black}{\eqref{eq:trap-P2-second-expansion}} and the weak chain rule give
		$|\nabla F_+|=O(r^{{\color{black}\delta}-3})$, with ${\color{black}\delta}=\pi/\alpha>2$, so
		\begin{equation}\label{eq:trap-F-positive-energy-bound}
			\int_{R\cap B_\varepsilon(P_2)}|\nabla F_+|^2\,d\sigma
			\le C\int_0^\varepsilon r^{2{\color{black}\delta}-5}\,dr
			=\frac{C\varepsilon^{2{\color{black}\delta}-4}}{2{\color{black}\delta}-4}<\infty.
		\end{equation}
		By \eqref{eq:trap-F-positive-energy-bound} and $F_+=0$ near $P_3$ and on $OC$,
		$F_+$ belongs to the mixed form domain.
		Let $\Gamma_\varepsilon=R\cap\partial B_\varepsilon(P_2)$,
		with $n$ the outward unit normal to $R\setminus\overline{B_\varepsilon(P_2)}$
		on this arc. Then
		\begin{equation}\label{eq:trap-F-corner-flux-limit}
			\left|\int_{\Gamma_\varepsilon}F_+\partial_nF\,ds\right|
			\le C\varepsilon^{2{\color{black}\delta}-4}\longrightarrow0
			\qquad(\varepsilon\downarrow0).
		\end{equation}
		Multiply $(\Delta+\Lambda)F=0$ by $F_+$ and integrate by parts,
		using \eqref{eq:trap-F-corner-flux-limit} at $P_2$ and $F_+=0$ near $P_3$.
		The conditions in \eqref{eq:trap-F-boundary-problem} leave only
		$\int_{P_2P_3}F_+F_n\,ds$.
		Substituting \eqref{eq:trap-F-oblique} and integrating
		$-{\color{black}\chi} F_+F_\tau=-\frac{\color{black}\chi}2\partial_\tau(F_+^2)$ along $P_2P_3$ gives
		\begin{equation}\label{eq:trap-F-energy}
			\int_R|\nabla F_+|^2\,d\sigma
			-\Lambda\int_R F_+^2\,d\sigma
			=-\frac{\color{black}\chi}2 F_+^2\Big|_{P_2}^{P_3}
			+\Lambda\sin\alpha\cos\alpha
			\int_{P_2P_3}F_+\nu_\tau\,ds\le0.
		\end{equation}
		The first mixed Rayleigh principle makes the left side of
		\textcolor{black}{\eqref{eq:trap-F-energy}} nonnegative, so equality holds.
		If $F_+\not\equiv0$, it attains the minimum
		Rayleigh quotient and hence is a multiple of $\nu$, by simplicity of
		the first mixed eigenvalue. Its vanishing near $P_3$ and $\nu>0$ in
		$R$ force that multiple to be zero. Thus $F_+\equiv0$ and $F\le0$.
		By \textcolor{black}{\eqref{eq:trap-F-oblique}}, we have $F\not\equiv0$.
		Lemma~\ref{lem2.1} now gives
		\begin{equation}\label{eq.add7.trap-second-derivative}
			F=\nu_{xx}+\Lambda\nu/2<0\quad\hbox{in }R.
		\end{equation}
		The inequality also holds on $P_2P_3$. Indeed, if $F(z)=0$ there,
		then $F_\tau(z)=0$ and Hopf's lemma gives $F_n(z)>0$, contrary to
		\textcolor{black}{\eqref{eq:trap-F-oblique}}.
	\end{proof}

	Set $g=\nu|_{P_2P_3}$, with primes denoting tangential derivatives.
	By \textcolor{black}{\eqref{eq.add7.trap-boundary-F}} and
	Lemma~\textcolor{black}{\ref{lem:trap-second-derivative}},
	\[
		g''+\frac{\Lambda g}{2}
		=-\frac{F}{\sin^2\alpha-\cos^2\alpha}>0.
	\]
	Put $w=\nu_x>0$. On the open edge $OC$, tangential differentiation
	of $\nu=0$ gives $\nu_{yy}=0$. The equation then yields
	$\nu_{xx}=-\nu_{yy}-\Lambda\nu=0$, so
	$\partial_nw=-\nu_{xx}=0$, since $n=(-1,0)$ there.
	Thus $w$ satisfies
	\[
		(\Delta+\Lambda)w=0\quad\text{in }R,\qquad
		\partial_nw=0\quad\text{on }OC\cup OP_2\cup CP_3,
	\]
	while on $P_2P_3$,
	\begin{equation}\label{eq.add7.trap-auxiliary-flux}
		\partial_n w=-\sin\alpha\,(g''+\Lambda g)
		<-\tfrac12\Lambda\sin\alpha\,g<0.
	\end{equation}
	We use this boundary inequality and
	Lemma~\textcolor{black}{\ref{lem.criticalpoint-comparison}}
	to prove the following determinant inequality.
	\begin{Proposition}\label{prop.trap.separation}
		On the full trapezoid $\Omega$, put $w=\nu_x$. Then
		\begin{equation}\label{eq:trap-W-positive}
			w>0,\qquad
			W:=\det\begin{pmatrix}
				\varphi&\nu&w\\
				\varphi_x&\nu_x&w_x\\
				\varphi_y&\nu_y&w_y
			\end{pmatrix}>0
			\quad\hbox{in }\Omega.
		\end{equation}
	\end{Proposition}
	\begin{proof}
		Since $\nu$ is odd in $x$, $w=\nu_x$ is even in $x$.
		Thus $\nu_x>0$ in $R$ gives $w>0$ in both halves of $\Omega$.
		On the open segment $OC$, Hopf's lemma applied to $\nu>0$ in $R$
		and $\nu=0$ on $OC$ gives $w=\nu_x>0$ as well.

		The equation for $w$ and its symmetry about the $y$-axis give
		\begin{equation}\label{eq:trap-w-boundary-problem}
			\begin{cases}
				(\Delta+\Lambda)w=0 &\hbox{in }\Omega,\\
				\partial_n w=0 &\hbox{on }P_1P_2\cup P_4P_3,\\
				\partial_n w<0 &\hbox{on }P_2P_3\cup P_1P_4.
			\end{cases}
		\end{equation}
		We also verify that $w\in H^1(\Omega)$, as required by
		Lemma~\ref{lem.criticalpoint-comparison}.
		At a vertex of angle $\omega$, set ${\color{black}\delta}=\pi/\omega$.
		Applying \textcolor{black}{\eqref{eq.add7.corner-first}} and
		\textcolor{black}{\eqref{eq:corner-remainder-derivatives}} to $\nu$ gives
		\begin{equation}\label{eq:trap-w-corner-bounds}
			w=O(r^{\min\{1,{\color{black}\delta}-1\}}),\qquad
			\nabla w=O(r^{\min\{0,{\color{black}\delta}-2\}}),\qquad
			\partial_n w=O(r^{\min\{0,{\color{black}\delta}-2\}}).
		\end{equation}
		At $P_1,P_2$ we have ${\color{black}\delta}=\pi/\alpha>2$, whereas at $P_3,P_4$ we have
		${\color{black}\delta}=\pi/(\pi-\alpha)\in(3/2,2)$. Thus
		\textcolor{black}{\eqref{eq:trap-w-corner-bounds}} yields
		$w\in H^1(\Omega)\cap C(\overline\Omega)$: at each vertex the
		squared gradient is integrable, since
		$\int_0^\varepsilon r^{2\min\{0,{\color{black}\delta}-2\}+1}\,dr<\infty$ for ${\color{black}\delta}>1$.

		Suppose, for a contradiction, that $W(p)=0$ at some $p\in\Omega$.
		By the definition of $W$ in \eqref{eq:trap-W-positive}, the matrix
		below has determinant zero. Its columns are therefore linearly
		dependent, so there are real coefficients $(a,b,c)\ne(0,0,0)$ such that
		\[
			\begin{pmatrix}
				\varphi(p)&\nu(p)&w(p)\\
				\varphi_x(p)&\nu_x(p)&w_x(p)\\
				\varphi_y(p)&\nu_y(p)&w_y(p)
			\end{pmatrix}
			\begin{pmatrix}a\\b\\-c\end{pmatrix}=0.
		\]
		Set $\psi=a\varphi+b\nu-cw$. The first row gives $\psi(p)=0$,
		and the second and third rows give $\psi_x(p)=\psi_y(p)=0$. Thus
		\[
			\psi(p)=0,\qquad \nabla\psi(p)=0.
		\]
		Since $w(p)>0$, we have $(a,b)\ne(0,0)$.
		If $c=0$, the different reflection symmetries of $\varphi,\nu$ give
		$\psi\not\equiv0$; if $c\ne0$, this follows from
		$\partial_n \psi=-c\partial_n w\not\equiv0$.
		After multiplying all three coefficients by $-1$ if needed, take $c\le0$.
		The Neumann conditions for $\varphi,\nu$ give
		\[
			\partial_n\psi=-c\partial_nw\le0
			\quad\text{on each open side of }\Omega.
		\]
		Moreover, $(\Delta+\Lambda)\psi=0$, and \eqref{eq:trap-w-corner-bounds}
		give $\psi\in H^1(\Omega)$.
		This contradicts Lemma~\textcolor{black}{\ref{lem.criticalpoint-comparison}}.

		Thus $W$ has one strict sign on the connected domain $\Omega$.
		Since $\Delta\varphi=-\Lambda\varphi$ and
		$\Delta\nu=-\Lambda\nu$, Green's second identity on $R$ gives
		\begin{equation}\label{eq:trap-green-axis}
			0=\int_R(\varphi\Delta\nu-\nu\Delta\varphi)\,d\sigma
			=\int_{\partial R}(\varphi\partial_n\nu-\nu\partial_n\varphi)\,ds
			=-\int_{OC}\varphi\nu_x\,ds.
		\end{equation}
		By \eqref{eq:trap-green-axis}, $\varphi$ cannot have one strict sign
		throughout $OC$. Since $\nu_x>0$ and $\varphi_y<0$
		on $OC$, there is an interior point $z_0\in OC$ with $\varphi(z_0)=0$.
		At $z_0$, the boundary conditions also give
		$\nu=\varphi_x=\nu_y=0$.
		Expanding the determinant at $z_0$ gives
		\begin{equation}\label{eq:trap-sign-at-zero}
			W(z_0)=-\varphi_y(z_0)\nu_x(z_0)w(z_0)
			=-\varphi_y(z_0)w(z_0)^2>0.
		\end{equation}
		The positive value in \textcolor{black}{\eqref{eq:trap-sign-at-zero}} determines the sign
		of $W$ throughout $\Omega$, proving
		\textcolor{black}{\eqref{eq:trap-W-positive}}.
	\end{proof}

	\phantomsection\label{sub.ext.17}
	For $\nu$ put
	\[
		e=\frac{\nabla\nu}{|\nabla\nu|}=(e_1,e_2),\qquad
		e^\perp=(-e_2,e_1),\qquad \mathcal C=D^2\nu(e,e^\perp).
	\]
	Here $D^2\nu$ is the Hessian.
	By Lemma~\ref{lem.x}, the gradient of $\nu$ is nonzero in $R$
	and on $OP_2\cup P_2P_3\cup P_3C$.
	On each of these open straight sides, let $\tau$ and $n$ be the
	unit tangent and normal. The Neumann condition $\nu_n=0$ makes
	$e$ tangential and $e^\perp$ normal. Since $n$ is constant along
	the side, tangential differentiation gives
	$D^2\nu(\tau,n)=\partial_\tau(\nu_n)=0$.
	Hence $\mathcal C=D^2\nu(e,e^\perp)=0$ on each such side.
	We next prove $\nu_{xy}<0$ and a sign inequality for
	$\mathcal C$.

	\begin{Lemma}\label{ext.lem:trapstrictturning}
		We have
		\begin{equation}
			\nu_{xy}<0,\qquad \mathcal C<0\quad\hbox{in }R.
			\label{eq.ext.manual.17.3}
		\end{equation}
	\end{Lemma}
	\begin{proof}
		We first prove $\nu_{xy}<0$ by checking its boundary signs and
		corner limits and applying the maximum principle.
		By \textcolor{black}{\eqref{eq.add7.trap-second-derivative}},
		$F=\nu_{xx}+\Lambda\nu/2<0$ in $R$.
		On $P_2P_3$, tangential differentiation of $\nu_n=0$ gives
		$\nu_{\tau n}=0$. Using
		$\partial_x=-\cos\alpha\,\partial_\tau+\sin\alpha\,\partial_n$
		and $\partial_y=\sin\alpha\,\partial_\tau+\cos\alpha\,\partial_n$,
		we obtain
		\begin{equation}\label{eq:trap-xy-slanted}
			\nu_{xy}=\sin\alpha\cos\alpha(\nu_{nn}-\nu_{\tau\tau})
			=-2\sin\alpha\cos\alpha
			\left(\nu_{\tau\tau}+\frac{\Lambda\nu}{2}\right)
			={\color{black}\chi} F<0.
		\end{equation}
		Here we used $\nu_{nn}=-\Lambda\nu-\nu_{\tau\tau}$,
		\textcolor{black}{\eqref{eq.add7.trap-boundary-F}}, and the definition of ${\color{black}\chi}$ in
		\textcolor{black}{\eqref{eq:trap-oblique-beta}}.
		On $OC$, Hopf's lemma applied to $\nu_y<0$ in $R$ and
		$\nu_y=0$ on $OC$ gives $\partial_n\nu_y=-\nu_{xy}>0$.
		On $OP_2$ and $CP_3$, tangential differentiation of $\nu_y=0$
		gives $\nu_{xy}=0$. Thus
		\begin{equation}\label{eq:trap-xy-boundary-problem}
			(\Delta+\Lambda)\nu_{xy}=0\ \text{in }R,\qquad
			\nu_{xy}=0\ \text{on }OP_2\cup CP_3,\qquad
			\nu_{xy}<0\ \text{on }OC.
		\end{equation}
		We use the polar coordinates introduced in the proof of
		Lemma~\ref{lem:trap-second-derivative} to estimate $\nu_{xy}$ at $P_2,P_3$.
		At $P_2$, set ${\color{black}\delta}=\pi/\alpha\in(2,3)$.
		Differentiating \eqref{eq:trap-P2-nu-expansion} gives
		$\nu_{xy}=O(r^{{\color{black}\delta}-2})\to0$ uniformly as $z\to P_2$ within $R$.

		At $P_3$, set ${\color{black}\delta}=\pi/(\pi-\alpha)\in(3/2,2)$, and let
		$c_0=\nu(P_3)$ and $c_1$ be the first angular coefficient in
		\eqref{eq:trap-P3-nu-expansion}. Write $P_3=(s,h)$, where $s=1-h\cot\alpha$.
		Since $r^2=(x-s)^2+(y-h)^2$, the quadratic term
		$-\Lambda c_0r^2/4$ has zero $xy$-derivative.
		Differentiating the remaining terms in \eqref{eq:trap-P3-nu-expansion} yields
		\begin{equation}\label{eq:trap-P3-xy-expansion}
			\nu_{xy}=c_1{\color{black}\delta}({\color{black}\delta}-1)r^{{\color{black}\delta}-2}\sin((2-{\color{black}\delta})\vartheta)
			+O(r^{2{\color{black}\delta}-2}).
		\end{equation}
		The remainder is uniform for $0\le\vartheta\le\pi-\alpha$.
		By \eqref{eq.add7.trap-upper-coefficient}, $c_1<0$; moreover,
		$0\le(2-{\color{black}\delta})\vartheta\le\pi-2\alpha<\pi/3$.
		Thus the leading term in \eqref{eq:trap-P3-xy-expansion} is nonpositive.
		Since $2{\color{black}\delta}-2>0$, it follows that
		$\limsup_{R\ni z\to P_3}\nu_{xy}(z)\le0$.

		The points $O,C$ lie in the interiors of the straight sides
		$P_1P_2,P_4P_3$ of $\Omega$, respectively. Boundary regularity
		there and tangential differentiation of $\nu_y=0$ give
		$\nu_{xy}(O)=\nu_{xy}(C)=0$.

		To pass from these boundary inequalities to the interior, we use
		a positive comparison function on domains with small corner disks removed.
		Since $\Lambda<\lambda_1(\Delta,R)$, let $P$ solve
		$(\Delta+\Lambda)P=0$ in $R$ with $P=1$ on $\partial R$;
		the maximum principle gives $P>0$, and $P$ is continuous on
		$\overline R$. At each vertex $V\in\{O,P_2,P_3,C\}$, we have
		\[
			\limsup_{\substack{z\to V\\z\in R}}\nu_{xy}(z)\le0,
			\qquad \lim_{\substack{z\to V\\z\in R}}P(z)=P(V)=1.
		\]
		Fix $\varepsilon>0$. For each $V$, these limits give a radius
		${\color{black}r_V}>0$ such that $\nu_{xy}(z)<\varepsilon/2$ and $P(z)>1/2$
		whenever $z\in R$ and $0<|z-V|<{\color{black}r_V}$.
		Thus $\nu_{xy}-\varepsilon P<0$ on the new circular arcs after
		corner disks of radii smaller than ${\color{black}r_V}$ are removed.
		The same inequality holds on the remaining original boundary by
		the boundary signs above. The truncated domain has first
		Dirichlet eigenvalue at least $\lambda_1(\Delta,R)>\Lambda$.
		Its maximum principle therefore gives $\nu_{xy}\le\varepsilon P$.
		For any fixed $z\in R$, shrinking the corner disks places $z$ in
		the truncated domain, so $\nu_{xy}(z)\le\varepsilon P(z)$.
		Letting $\varepsilon\downarrow0$ then gives $\nu_{xy}(z)\le0$.
		The strict boundary sign in \textcolor{black}{\eqref{eq:trap-xy-slanted}} and the strong
		maximum principle give $\nu_{xy}<0$ throughout $R$.

		We next establish $\mathcal C<0$ near $P_2,P_3$, to control the
		corners in the maximum-principle argument below.
		\begin{equation}
			\mathcal C=
			\frac{(\nu_x^2-\nu_y^2)\nu_{xy}
			+\nu_x\nu_y(\nu_{yy}-\nu_{xx})}
			{\nu_x^2+\nu_y^2}.                                \label{eq.ext.manual.17.4}
		\end{equation}
		At each vertex, $c_0,c_1$ denote the corresponding expansion coefficients in
		\eqref{eq:trap-P2-nu-expansion} or \eqref{eq:trap-P3-nu-expansion}.
		At $P_2$, put ${\color{black}\delta}=\pi/\alpha\in(2,3)$.
		Here $c_0=\nu(P_2)>0$, since $\nu=0$ at $O$ and
		$\nu_x>0$ on $OP_2$. Since $F<0$ in $R$, continuity gives
		$F\le0$ on $OP_2$.
		On its ray $\vartheta=0$, \textcolor{black}{\eqref{eq:trap-P2-second-expansion}} gives
		\[
			F=c_1{\color{black}\delta}({\color{black}\delta}-1)r^{{\color{black}\delta}-2}+O(r^2)\le0.
		\]
		Since ${\color{black}\delta}-2<2$, this implies $c_1\le0$. If $c_1=0$, retaining the
		fourth-order radial term in the full series \eqref{eq.add7.neumann} gives
		\[
			\nu=c_0\left(1-\frac{\Lambda r^2}{4}
			+\frac{\Lambda^2r^4}{64}\right)+O(r^{2{\color{black}\delta}}).
		\]
		Differentiating this refined expansion gives, on the same ray,
		\[
			F=\frac{\Lambda^2c_0}{16}r^2+O(r^{2{\color{black}\delta}-2})>0
			\quad\text{for sufficiently small }r,
		\]
		because $2{\color{black}\delta}-2>2$, a contradiction. Thus $c_1<0$ at $P_2$.

		At $P_2$, differentiating \eqref{eq:trap-P2-nu-expansion} gives
		\begin{equation}\label{eq:trap-P2-gradient}
			(\nu_x,\nu_y)=\frac{\Lambda c_0}{2}r
			(\cos\vartheta,-\sin\vartheta)+O(r^{{\color{black}\delta}-1}),
		\end{equation}
		Here $x=1-r\cos\vartheta$ and $y=r\sin\vartheta$.
		Since $c_0>0$ and ${\color{black}\delta}>2$, \eqref{eq:trap-P2-gradient} gives
		$e=(\cos\vartheta,-\sin\vartheta)+O(r^{{\color{black}\delta}-2})$ and
		$e^\perp=(\sin\vartheta,\cos\vartheta)+O(r^{{\color{black}\delta}-2})$.
		Differentiating \eqref{eq:trap-P2-nu-expansion} twice gives
		\[
			D^2\nu=-\frac{\Lambda c_0}{2}I
			+c_1{\color{black}\delta}({\color{black}\delta}-1)r^{{\color{black}\delta}-2}
			\begin{pmatrix}
				\cos(({\color{black}\delta}-2)\vartheta)&\sin(({\color{black}\delta}-2)\vartheta)\\
				\sin(({\color{black}\delta}-2)\vartheta)&-\cos(({\color{black}\delta}-2)\vartheta)
			\end{pmatrix}+O(r^2),
		\]
		where $I$ is the identity matrix. The first matrix is the Hessian of
		$-\Lambda c_0r^2/4$. Since $e\cdot e^\perp=0$, its value on
		$(e,e^\perp)$ is zero. Substituting the remaining matrix and the
		above estimates for $e,e^\perp$ into $\mathcal C=D^2\nu(e,e^\perp)$ gives
		the following expansion, where $E_2=E_2(r,\vartheta)$ denotes the remainder
		at $P_2$:
		\begin{equation}\label{eq.ext.manual.17.5}
			\mathcal C=c_1{\color{black}\delta}({\color{black}\delta}-1)r^{{\color{black}\delta}-2}\sin({\color{black}\delta}\vartheta)+E_2,
			\qquad |E_2|+|\partial_\vartheta E_2|\le Kr^{2{\color{black}\delta}-4}.
		\end{equation}

		At $P_3$, put ${\color{black}\delta}=\pi/(\pi-\alpha)\in(3/2,2)$.
		Here $c_1<0$ by
		\textcolor{black}{\eqref{eq.add7.trap-upper-coefficient}}. Differentiating
		\eqref{eq:trap-P3-nu-expansion} gives
		\begin{equation}\label{eq:trap-P3-gradient}
			(\nu_x,\nu_y)=c_1{\color{black}\delta} r^{{\color{black}\delta}-1}
			\bigl(-\cos(({\color{black}\delta}-1)\vartheta),\sin(({\color{black}\delta}-1)\vartheta)\bigr)+O(r),
		\end{equation}
		Here $x=s-r\cos\vartheta$ and $y=h-r\sin\vartheta$.
		In the same Cartesian coordinates, the Hessian has the expansion
		\[
			D^2\nu=\begin{pmatrix}\nu_{xx}&\nu_{xy}\\\nu_{xy}&\nu_{yy}\end{pmatrix}
			=c_1{\color{black}\delta}({\color{black}\delta}-1)r^{{\color{black}\delta}-2}
			\begin{pmatrix}
				\cos(({\color{black}\delta}-2)\vartheta)&-\sin(({\color{black}\delta}-2)\vartheta)\\
				-\sin(({\color{black}\delta}-2)\vartheta)&-\cos(({\color{black}\delta}-2)\vartheta)
			\end{pmatrix}+O(1).
		\]
		Equation~\textcolor{black}{\eqref{eq:trap-P3-gradient}} gives
		$|\nabla\nu|^2=c_1^2{\color{black}\delta}^2r^{2{\color{black}\delta}-2}[1+O(r^{2-{\color{black}\delta}})]$.
		Substitution into $\mathcal C=D^2\nu(e,e^\perp)$ yields the following
		expansion, with remainder $E_3=E_3(r,\vartheta)$ at $P_3$:
		\begin{equation}\label{eq:trap-P3-turning}
			\mathcal C=c_1{\color{black}\delta}({\color{black}\delta}-1)r^{{\color{black}\delta}-2}\sin({\color{black}\delta}\vartheta)+E_3,
			\qquad |E_3|+|\partial_\vartheta E_3|\le K.
		\end{equation}
		The bounds on $\partial_\vartheta E_2$ and $\partial_\vartheta E_3$
		follow from the differentiable remainder estimates and hold uniformly
		in the angle. At $P_2$, $\pi/{\color{black}\delta}=\alpha$; at $P_3$, $\pi/{\color{black}\delta}=\pi-\alpha$.
		Thus at either vertex the two boundary rays are $\vartheta=0$ and
		$\vartheta=\pi/{\color{black}\delta}$, where $\sin({\color{black}\delta}\vartheta)=0$.
		The main term therefore vanishes on both rays and becomes arbitrarily
		small near them. To determine the sign of $\mathcal C$ uniformly in the
		angle, we need an error bound containing the same factor $\sin({\color{black}\delta}\vartheta)$.
		The Neumann condition gives $\mathcal C=0$ at both endpoints, hence
		$E_i(r,0)=E_i(r,\pi/{\color{black}\delta})=0$. Integrating from the nearer endpoint gives
		\[
			|E_i(r,\vartheta)|\le
			\sup_{0\le t\le\pi/{\color{black}\delta}}|\partial_tE_i(r,t)|
			\min\{\vartheta,\pi/{\color{black}\delta}-\vartheta\}.
		\]
		The inequality $\sin t\ge2t/\pi$ for $0\le t\le\pi/2$, applied to
		$t={\color{black}\delta}\min\{\vartheta,\pi/{\color{black}\delta}-\vartheta\}$, gives
		\[
			\min\{\vartheta,\pi/{\color{black}\delta}-\vartheta\}
			\le\frac{\pi}{2{\color{black}\delta}}\sin({\color{black}\delta}\vartheta).
		\]
		Combining this with the preceding derivative bounds, we obtain
		$|E_2|\le Kr^{2{\color{black}\delta}-4}\sin({\color{black}\delta}\vartheta)$ at $P_2$ and
		$|E_3|\le K\sin({\color{black}\delta}\vartheta)$ at $P_3$.
		Consequently,
		\begin{equation}\label{eq:trap-corner-turning-error}
			\mathcal C=c_1{\color{black}\delta}({\color{black}\delta}-1)r^{{\color{black}\delta}-2}\sin({\color{black}\delta}\vartheta)
			\begin{cases}1+O(r^{{\color{black}\delta}-2})&\text{at }P_2,\\
			1+O(r^{2-{\color{black}\delta}})&\text{at }P_3.\end{cases}
		\end{equation}
		The error bounds are uniform in the angle. Since $c_1<0$ and both
		powers of $r$ are positive, $\mathcal C<0$ in the interior sector
		near each vertex.

		It remains to prove $\mathcal C<0$ away from the corners.
		Since $F<0$ and $\nu_{xy}<0$, \eqref{eq.ext.manual.17.4} already
		gives this sign wherever $-\nu_y\le\nu_x$.
		On $E=\{-\nu_y>\nu_x\}$, write $\rho=|\nabla\nu|$ and set
		\begin{equation}\label{eq:trap-turning-comparison}
			N=\rho^2\mathcal C
			=(\nu_x^2-\nu_y^2)\nu_{xy}-2\nu_x\nu_yF,\qquad
			H=\nu_x\nu_y(\nu_x^2-\nu_y^2)>0.
		\end{equation}
		We will exclude a positive maximum of $N/H$ by comparing their
		differential equations. Define
		\begin{equation}\label{eq:trap-comparison-operator}
			\mathscr L f=\Delta f
			-\frac4{\rho^2}\bigl[(2D^2\nu+\Lambda\nu I)\nabla\nu\bigr]\cdot\nabla f
			+\left(\Lambda+\frac{16(F^2+\nu_{xy}^2)}{\rho^2}\right)f,
		\end{equation}
		where $I$ is the identity matrix. We claim that
		\begin{equation}\label{eq:trap-comparison-identities}
			\begin{gathered}
				\mathscr LN=0,\\
				\mathscr LH
				=-\left(3\Lambda+\frac{16(F^2+\nu_{xy}^2)}{\rho^2}\right)H
				-2\Lambda\nu\bigl[(\nu_x^2-\nu_y^2)\nu_{xy}
				+2\nu_x\nu_yF\bigr]<0\quad\text{in }E.
			\end{gathered}
		\end{equation}
		The first identity holds throughout $R$.

		To prove the claim, apply the product rule
		and the differentiated eigenvalue equation. In
		$\Delta N$, the terms containing $\Lambda\nu F\nu_{xy}$ cancel, leaving
		\[
			\Delta N=-3\Lambda N
			+2\nabla(\nu_x^2-\nu_y^2)\cdot\nabla\nu_{xy}
			-4\nabla(\nu_x\nu_y)\cdot\nabla F.
		\]
		Expand $\nabla N$ and use
		\[
			F_x+(\nu_{xy})_y=-\frac{\Lambda}{2}\nu_x,\qquad
			F_y-(\nu_{xy})_x=\frac{\Lambda}{2}\nu_y
		\]
		to eliminate the third derivatives. The preceding expression becomes
		\[
			\Delta N=
			\frac4{\rho^2}\bigl[(2D^2\nu+\Lambda\nu I)\nabla\nu\bigr]\cdot\nabla N
			-\left(\Lambda+\frac{16(F^2+\nu_{xy}^2)}{\rho^2}\right)N,
		\]
		which is $\mathscr LN=0$. For $H$, apply the product rule to its
		two factors, using
		\[
			(\Delta+2\Lambda)(\nu_x\nu_y)=-2\Lambda\nu\nu_{xy},\qquad
			(\Delta+2\Lambda)(\nu_x^2-\nu_y^2)=-4\Lambda\nu F.
		\]
		Substitution in \eqref{eq:trap-comparison-operator} gives the
		second identity in \eqref{eq:trap-comparison-identities}.
		To verify its strict sign, recall that on $E$ we have
		$\nu_x>0$, $\nu_y<0$, $\nu_x^2-\nu_y^2<0$,
		$\nu_{xy}<0$ and $F<0$. Hence
		\[
			(\nu_x^2-\nu_y^2)\nu_{xy}>0,
			\qquad 2\nu_x\nu_yF>0.
		\]
		Since $H>0$, $\nu>0$ and $\Lambda>0$ on $E$, both terms
		on the right-hand side of the identity for $\mathscr LH$ are negative.
		Thus $\mathscr LH<0$ on $E$, proving the claim.

		We now check where a positive maximum of $N/H$ could occur.
		At a point of $\partial E\cap R$, $\nu_x^2-\nu_y^2=0$,
		so $N=-2\nu_x\nu_yF<0$, whereas $H\to0$ from above.
		Thus $N/H\to-\infty$ there.
		On the open edge $P_2P_3$, $\nu_y=-\tan\alpha\,\nu_x$ and
		$\mathcal C=0$, so $H>0$ and $N/H=0$.
		The open edges $OP_2$, $CP_3$ and $OC$ cannot meet $\overline E$, since
		$\nu_y=0$ and $\nu_x>0$ there. Finally,
		\eqref{eq:trap-corner-turning-error} gives $N/H<0$ near $P_2,P_3$.

		Consequently, any positive value of $N/H$ would yield a positive
		interior maximum $m$. At that point $N-mH=0$,
		$\nabla(N-mH)=0$ and $\Delta(N-mH)\le0$, but
		\[
			\mathscr L(N-mH)=-m\mathscr LH>0,
		\]
		a contradiction. Hence $N\le0$ in $E$, and therefore throughout $R$.
		Since the zeroth-order coefficient of $\mathscr L$ is positive,
		$\mathscr LN=0$ and $N\le0$ imply that $N$ is a subsolution of the
		same operator with its zeroth-order term omitted.
		The strong maximum principle \cite[Theorem~3.5]{Gilbarg1983}
		gives $N<0$ in $R$, because $N<0$ where $-\nu_y\le\nu_x$.
		Thus $\mathcal C=N/\rho^2<0$ throughout $R$.
	\end{proof}

	On the full trapezoid $\Omega$, retain $w=\nu_x$ and the notation
	$e,e^\perp,\mathcal C$ introduced above.
	Set
	\begin{equation}\label{eq:trap-normalizing-determinants}
		\rho=|\nabla\nu|,\qquad
		D=\det(\nabla w,\nabla\nu)
		=\nu_{xx}\nu_y-\nu_{xy}\nu_x.
	\end{equation}
	Proposition~\textcolor{black}{\ref{prop.trap.separation}} gives $W>0$.
	Together with $\mathcal C<0$, this will give a strictly negative
	right-hand side in the equation for ${\color{black}\mathcal T}=J/D$.

	\begin{Lemma}\label{lem:trap-normalized-jacobian}
		Let $D$ be defined by \textcolor{black}{\eqref{eq:trap-normalizing-determinants}}
		and $W$ by \textcolor{black}{\eqref{eq:trap-W-positive}}.
		Then $D>0$ throughout $\Omega$, and the quotient ${\color{black}\mathcal T}=J/D$ is well defined
		and satisfies
		\begin{equation}\label{eq:trap-normalized-jacobian}
			\Delta {\color{black}\mathcal T}+b\cdot\nabla {\color{black}\mathcal T}=\frac{2\Lambda\mathcal C W}{D^2}<0
			\quad\text{in }\Omega,\qquad
			b=2\nabla\log\frac{D}{\rho^2}
			-\frac{2\Lambda\nu}{\rho^2}\nabla\nu.
		\end{equation}
	\end{Lemma}
	\begin{proof}
		In $R$, the inequalities $\nu_{xx}<0$, $\nu_y<0$, $\nu_{xy}<0$ and
		$\nu_x>0$ give $D>0$. On $OC$, \eqref{eq:trap-xy-boundary-problem}
		and $\nu_y=0$ give $D=-\nu_{xy}\nu_x>0$ and
		$\mathcal C=\nu_{xy}<0$.
		The antisymmetry of $\nu$ makes $D$ and $\mathcal C$ even in $x$.
		Lemma~\textcolor{black}{\ref{ext.lem:trapstrictturning}} therefore gives the stated signs
		throughout $\Omega$.

		Fix $p\in\Omega$ and put $t={\color{black}\mathcal T}(p)$. Recall that $w=\nu_x$.
		By the definitions of $J,D$,
		\[
			\det\bigl(\nabla\varphi(p)-t\nabla w(p),\nabla\nu(p)\bigr)
			=J(p)-tD(p)=0.
		\]
		Since $\nabla\nu(p)\ne0$, the two vectors are parallel, so there is
		a constant $a$ with $\nabla\varphi(p)-t\nabla w(p)=a\nabla\nu(p)$.
		Keeping $t,a$ fixed, define
		\[
			\psi=\varphi-tw-a\nu,\qquad \nabla\psi(p)=0.
		\]
		In the determinant defining $W(p)$, replace the first column by
		the first column minus $t$ times the third and $a$ times the second.
		Since $\nabla\psi(p)=0$, this gives
		\[
			W(p)=\det\begin{pmatrix}
				\psi(p)&\nu(p)&w(p)\\
				0&\nu_x(p)&w_x(p)\\
				0&\nu_y(p)&w_y(p)
			\end{pmatrix}
			=\psi(p)\bigl(\nu_x(p)w_y(p)-\nu_y(p)w_x(p)\bigr)
			=-\psi(p)D(p).
		\]
		Since $W(p)>0$ and $D(p)>0$, we obtain
		\begin{equation}\label{eq:trap-contact-value}
			\psi(p)=-\frac{W(p)}{D(p)}<0.
		\end{equation}

		Set $J_\psi=\det(\nabla\psi,\nabla\nu)=J-tD=D({\color{black}\mathcal T}-t)$.
		In the following calculation all quantities are evaluated at $p$.
		Writing $J_\psi=\psi_x\nu_y-\psi_y\nu_x$ and using
		$\psi_x(p)=\psi_y(p)=0$, the product rule gives
		\[
			\nabla J_\psi
			=\nu_y\nabla\psi_x-\nu_x\nabla\psi_y
			=D^2\psi\begin{pmatrix}\nu_y\\-\nu_x\end{pmatrix}
			=-\rho D^2\psi\,e^\perp.
		\]
		Here $(\nu_y,-\nu_x)^{\mathsf T}=-\rho e^\perp$ by the definition of
		$e^\perp$. The same gradient can also be computed from
		$J_\psi=D({\color{black}\mathcal T}-t)$. With $t$ fixed and ${\color{black}\mathcal T}(p)=t$, the product rule gives
		\[
			\nabla J_\psi=({\color{black}\mathcal T}-t)\nabla D+D\nabla {\color{black}\mathcal T}=D\nabla {\color{black}\mathcal T}.
		\]
		Combining these two expressions yields
		\[
			\nabla J_\psi=-\rho D^2\psi\,e^\perp=D\nabla {\color{black}\mathcal T}.
		\]
		Write ${\color{black}\mathcal T}_e=\nabla {\color{black}\mathcal T}\cdot e$, ${\color{black}\mathcal T}_\perp=\nabla {\color{black}\mathcal T}\cdot e^\perp$,
		$A=D^2\nu(e,e)$ and $B=D^2\nu(e^\perp,e^\perp)$.
		The preceding identity and $\Delta\psi=-\Lambda\psi$ give
		\begin{equation}\label{eq:trap-contact-hessian}
			D^2\psi(e,e^\perp)=-\frac{D}{\rho}{\color{black}\mathcal T}_e,\qquad
			D^2\psi(e^\perp,e^\perp)=-\frac{D}{\rho}{\color{black}\mathcal T}_\perp,\qquad
			D^2\psi(e,e)=-\Lambda\psi+\frac{D}{\rho}{\color{black}\mathcal T}_\perp.
		\end{equation}
		The product rule for the determinant gives
		\[
			\begin{aligned}
				\Delta J_\psi
				&=\det(\nabla\Delta\psi,\nabla\nu)
				+\det(\nabla\psi,\nabla\Delta\nu)
				+2\sum_{v=e,e^\perp}
				\det(D^2\psi\,v,D^2\nu\,v)\\
				&=-2\Lambda J_\psi
				+2\sum_{v=e,e^\perp}\det(D^2\psi\,v,D^2\nu\,v).
			\end{aligned}
		\]
		Here the directions $e(p),e^\perp(p)$ are fixed; no derivatives of
		these directions are taken. Substituting
		\eqref{eq:trap-contact-hessian} and $J_\psi(p)=0$ yields
		\[
			\Delta J_\psi=-2\Lambda\psi\mathcal C
			+\frac{2D}{\rho}(A-B){\color{black}\mathcal T}_e
			+\frac{4D}{\rho}\mathcal C {\color{black}\mathcal T}_\perp.
		\]
		Since $A+B=-\Lambda\nu$ and
		$\nabla\log\rho=(Ae+\mathcal C e^\perp)/\rho$, this becomes
		\[
			\Delta J_\psi=-2\Lambda\psi\mathcal C
			+D\left(4\nabla\log\rho+
			\frac{2\Lambda\nu}{\rho^2}\nabla\nu\right)\cdot\nabla {\color{black}\mathcal T}.
		\]
		On the other hand, $J_\psi=D({\color{black}\mathcal T}-t)$ and ${\color{black}\mathcal T}(p)=t$ give
		\[
			\Delta J_\psi=D\Delta {\color{black}\mathcal T}+2\nabla D\cdot\nabla {\color{black}\mathcal T}.
		\]
		Combining these identities with \eqref{eq:trap-contact-value}
		proves \eqref{eq:trap-normalized-jacobian} at $p$.
		Since $p$ was arbitrary, the equation holds throughout $\Omega$;
		its strict sign follows from $\mathcal C<0$ and $W>0$.
	\end{proof}

	We next establish the boundary behavior of ${\color{black}\mathcal T}$.

	\begin{Lemma}\label{lem:trap-normalized-boundary}
		Let ${\color{black}\mathcal T}=J/D$ be the quotient defined in
		Lemma~\ref{lem:trap-normalized-jacobian}. Then
		\[
			{\color{black}\mathcal T}\in C(\overline\Omega),\qquad
			{\color{black}\mathcal T}=0\quad\text{on }\overline{P_2P_3}\cup\overline{P_1P_4}.
		\]
		Moreover, ${\color{black}\mathcal T}$ is smooth up to the relative interiors of $P_1P_2$ and
		$P_4P_3$, where $\partial_n\mathcal T=0$ and $\partial_nD<0$;
		here $n$ is the outward unit normal.
	\end{Lemma}
	\begin{proof}
		On $P_2P_3$, take the outward normal $n$ and the unit tangent
		$\tau$ with $\det(n,\tau)=1$. The Neumann conditions give $J=0$,
		while \textcolor{black}{\eqref{eq.add7.trap-auxiliary-flux}} and Lemma~\ref{lem.x}
		give $w_n<0$ and $\nu_\tau<0$. Hence $D=w_n\nu_\tau>0$
		and ${\color{black}\mathcal T}=0$ on the relative interior of $P_2P_3$.
		Since $J,D$ are even in $x$, ${\color{black}\mathcal T}=0$ also on the relative interior of $P_1P_4$.

		To prove ${\color{black}\mathcal T}(z)\to0$ as $z$ approaches $P_2$ or $P_3$ from within $\Omega$,
		we will bound $|{\color{black}\mathcal T}|=|J|/D$ by a positive power of the distance to the vertex.
		At $P_2,P_3$, use the polar coordinates of
		Lemma~\ref{lem:trap-second-derivative}, with angle ${\color{black}\vartheta}$ and
		${\color{black}\delta}=\pi/\alpha$ or $\pi/(\pi-\alpha)$, respectively.
		The estimates below use \eqref{eq:trap-angular-remainder-bound}:
		If $E(r,0)=0$ and $|\partial_{\color{black}\vartheta} E|\le Cr^q$, then
		\begin{equation}\label{eq:trap-angular-remainder-bound}
			|E(r,{\color{black}\vartheta})|\le\int_0^{\color{black}\vartheta}|\partial_{\color{black}\vartheta} E(r,t)|\,dt
			\le Cr^q\sin{\color{black}\vartheta}.
		\end{equation}
		The last inequality holds because ${\color{black}\vartheta}/\sin{\color{black}\vartheta}$ has a bounded
		continuous extension to $[0,\pi/{\color{black}\delta}]$ and $\pi/{\color{black}\delta}<\pi$.
		We first estimate $J$. The corner expansion used in
		\eqref{eq:trap-P2-nu-expansion} and \eqref{eq:trap-P3-nu-expansion}
		also applies to $\varphi$, with its own coefficients $a_0,a_1$.
		Writing $f=1-\Lambda r^2/4$ and $g=r^{\color{black}\delta}\cos({\color{black}\delta}{\color{black}\vartheta})$, we have
		\begin{equation}\label{eq:trap-paired-corner-expansions}
			\varphi=a_0f+a_1g+O(r^m),\qquad
			\nu=c_0f+c_1g+O(r^m),\qquad
			m=\begin{cases}4&\text{at }P_2,\\2{\color{black}\delta}&\text{at }P_3.\end{cases}
		\end{equation}
		Differentiating \eqref{eq:trap-paired-corner-expansions}, with the
		remainder bounds stated above, and using bilinearity of the determinant and
		$\det(\nabla f,\nabla f)=\det(\nabla g,\nabla g)=0$,
		\[
			J=(a_0c_1-a_1c_0)\det(\nabla f,\nabla g)
			+O(r^m)+O(r^{{\color{black}\delta}+m-2})+O(r^{2m-2}).
		\]
		Indeed, $\nabla f=O(r)$, $\nabla g=O(r^{{\color{black}\delta}-1})$, and the
		gradients of the remainders are $O(r^{m-1})$.
		Thus the displayed determinant is $O(r^{\color{black}\delta})$; all three remainder
		exponents are at least ${\color{black}\delta}$ at both vertices.
		The same bounds hold after one ${\color{black}\vartheta}$-derivative, so
		$|\partial_{\color{black}\vartheta} J|\le Cr^{\color{black}\delta}$. Since $J(r,0)=0$,
		\eqref{eq:trap-angular-remainder-bound} gives
		\begin{equation}\label{eq:trap-J-corner-bound}
			|\partial_{\color{black}\vartheta} J|\le Cr^{\color{black}\delta},\qquad |J|\le Cr^{\color{black}\delta}\sin{\color{black}\vartheta}.
		\end{equation}
		For $D$, set $k=\Lambda\nu(P_2)/2>0$ and
		$\kappa=({\color{black}\delta}-1){\color{black}\delta}^2c_1^2>0$, where $c_1$ is the first angular
		coefficient of $\nu$ at $P_3$.
		The differentiated expansions already used in
		\eqref{eq:trap-P2-gradient} and \eqref{eq:trap-P3-gradient} give
		\[
			\begin{aligned}
				|\partial_{\color{black}\vartheta}(D-k^2r\sin{\color{black}\vartheta})|&\le Cr^{{\color{black}\delta}-1}
				&&\text{at }P_2,\\
				|\partial_{\color{black}\vartheta}(D-\kappa r^{2{\color{black}\delta}-3}\sin{\color{black}\vartheta})|&\le Cr^{{\color{black}\delta}-1}
				&&\text{at }P_3,
			\end{aligned}
		\]
		The term $k^2r\sin{\color{black}\vartheta}$ comes from the radial quadratic term at $P_2$;
		$\kappa r^{2{\color{black}\delta}-3}\sin{\color{black}\vartheta}$ comes from $c_1r^{\color{black}\delta}\cos({\color{black}\delta}{\color{black}\vartheta})$ at $P_3$.
		In both cases $D(r,0)=0$, so
		\eqref{eq:trap-angular-remainder-bound} controls the error by
		$Cr^{{\color{black}\delta}-1}\sin{\color{black}\vartheta}$. Dividing this error bound by the respective
		main terms gives $O(r^{{\color{black}\delta}-2})$ at $P_2$ and $O(r^{2-{\color{black}\delta}})$ at $P_3$.
		Both tend to zero. Thus, for sufficiently small $r$,
		\begin{equation}\label{eq:trap-P2-normalized-bound}
			D\ge\tfrac12 k^2r\sin{\color{black}\vartheta},\qquad |{\color{black}\mathcal T}|\le Cr^{{\color{black}\delta}-1}
			\quad\text{at }P_2,
		\end{equation}
		\begin{equation}\label{eq:trap-P3-normalized-bound}
			D\ge\tfrac12\kappa r^{2{\color{black}\delta}-3}\sin{\color{black}\vartheta},\qquad |{\color{black}\mathcal T}|\le Cr^{3-{\color{black}\delta}}
			\quad\text{at }P_3.
		\end{equation}
		Equations~\eqref{eq:trap-P2-normalized-bound} and
		\eqref{eq:trap-P3-normalized-bound} give ${\color{black}\mathcal T}\to0$ uniformly at $P_2,P_3$.
		Since $J,D$ are even in $x$, the same holds at $P_1,P_4$.

		On the relative interiors of $P_1P_2,P_4P_3$, we have $J=D=0$, since
		$\varphi_y=\nu_y=\nu_{xy}=0$. We show that $D$ has a simple zero
		in $y$ at $y=0$ and in $h-y$ at $y=h$, so these common factors
		can be removed from $J/D$.
		By Lemma~\ref{ext.lem:trapstrictturning}, symmetry, and
		\eqref{eq:trap-xy-boundary-problem}, $\nu_{xy}<0$ in $\Omega$.
		Since $\nu_{xy}=0$ on $P_1P_2\cup P_4P_3$,
		Hopf's lemma gives
		$\nu_{xyy}<0$ on $P_1P_2$ and $\nu_{xyy}>0$ on $P_4P_3$.
		Also $w=\nu_x>0$ in $\Omega$ by \eqref{eq:trap-W-positive}.
		If $w$ vanished at an interior point of either edge, Hopf's lemma
		would give $\partial_nw<0$, contrary to \eqref{eq:trap-w-boundary-problem}.
		Hence $\nu_x>0$ on the relative interiors of both edges.
		On the right halves of these edges, $\nu_{xx}<0$ follows from
		\textcolor{black}{\eqref{eq.add7.trap-second-derivative}} and $\nu>0$;
		Hopf's lemma applied to $\nu_y<0$ gives $\nu_{yy}<0$ on $OP_2$
		and $\nu_{yy}>0$ on $CP_3$.
		Since $\nu$ is odd in $x$, both $\nu_{xx}$ and $\nu_{yy}$ are odd,
		and their product is even. Continuity at $O,C$ therefore gives
		$\nu_{xx}\nu_{yy}\ge0$ on $P_1P_2$ and
		$\nu_{xx}\nu_{yy}\le0$ on $P_4P_3$.
		Using $\nu_y=\nu_{xy}=0$ on these edges, we obtain
		\begin{equation}\label{eq:trap-D-base-derivative}
			\begin{aligned}
				D_y&=\nu_{xx}\nu_{yy}-\nu_x\nu_{xyy}>0
				&&\text{on }P_1P_2,\\
				-D_y&=-\nu_{xx}\nu_{yy}+\nu_x\nu_{xyy}>0
				&&\text{on }P_4P_3.
			\end{aligned}
		\end{equation}
		By Neumann reflection, $\varphi,\nu,w$ extend smoothly and evenly
		across each horizontal edge, so $J,D$ extend oddly. Dividing both by
		$y$ at $P_1P_2$, or by $h-y$ at $P_4P_3$, gives smooth even functions.
		By \eqref{eq:trap-D-base-derivative}, the denominator is strictly positive
		on the edge:
		\begin{equation}\label{eq:trap-base-factorization}
			\left.\frac{D}{y}\right|_{y=0}=D_y(x,0)>0,\qquad
			\left.\frac{D}{h-y}\right|_{y=h}=-D_y(x,h)>0.
		\end{equation}
		Hence ${\color{black}\mathcal T}$ extends smoothly and evenly, with ${\color{black}\mathcal T}_y=0$ on both open edges.
		These extensions, together with its vanishing on $P_2P_3,P_1P_4$ and the corner limits in
		\textcolor{black}{\eqref{eq:trap-P2-normalized-bound}} and
		\textcolor{black}{\eqref{eq:trap-P3-normalized-bound}} give ${\color{black}\mathcal T}\in C(\overline\Omega)$.
	\end{proof}

	\begin{proof}[\textcolor{black}{Proof of Theorem~\ref{thm1.3} for critical-height trapezoids}]
		By Lemma~\textcolor{black}{\ref{lem:trap-normalized-jacobian}}, $D>0$ in $\Omega$,
		so proving $J>0$ is equivalent to proving ${\color{black}\mathcal T}=J/D>0$.
		Let $L=\Delta+b\cdot\nabla$, with $b$ defined in
		\eqref{eq:trap-normalized-jacobian}, so that
		\eqref{eq:trap-normalized-jacobian} gives $L\mathcal T<0$ in $\Omega$.
		The principal coefficient matrix of $L$ is the identity, so $L$ is
		uniformly elliptic. The vector field $b$ is smooth in $\Omega$, but
		need not remain bounded near its boundary; in fact, the estimates
		below show that its normal component diverges at $P_1P_2,P_4P_3$.
		Thus the interior minimum principle applies locally, whereas the
		standard boundary-point lemma requiring bounded drift coefficients
		cannot be applied directly on these edges.
		Lemma~\ref{lem:trap-normalized-boundary} gives ${\color{black}\mathcal T}\in C(\overline\Omega)$,
		${\color{black}\mathcal T}=0$ on $P_2P_3,P_1P_4$ and at all vertices, and ${\color{black}\mathcal T}_y=0$ on
		the open edges $P_1P_2,P_4P_3$. The last boundary condition does not
		itself rule out a negative boundary minimum. We handle it by a small
		perturbation after estimating $b_y$.

		Define the auxiliary function $g(y)=y(h-y)$.
		We will use ${\color{black}\mathcal T}-\varepsilon g$ to move a possible boundary minimum
		into the interior. For the resulting contradiction, we first show
		that $Lg>0$ near the open edges $P_1P_2,P_4P_3$.
		By Lemma~\ref{lem:trap-normalized-boundary} and
		\eqref{eq:trap-base-factorization},
		$D/y$ and $D/(h-y)$ have smooth positive even extensions across
		$y=0$ and $y=h$, respectively, and ${\color{black}\mathcal T}_y=0$ on both open edges.
		At $y=0$, the smooth odd extension of $D$ gives
		\[
			D(x,y)=D_y(x,0)y+O(y^3),\qquad
			D_y(x,y)=D_y(x,0)+O(y^2).
		\]
		Consequently,
		\[
			\partial_y(D/y)=\frac{yD_y-D}{y^2}
			=\frac{O(y^3)}{y^2}=O(y).
		\]
		The Neumann condition and its tangential derivative give
		$\nu_y(x,0)=\nu_{xy}(x,0)=0$, so
		\[
			\partial_y(\rho^2)(x,0)
			=2\nu_x(x,0)\nu_{xy}(x,0)+2\nu_y(x,0)\nu_{yy}(x,0)=0.
		\]
		By smoothness, $\partial_{yy}(\rho^2)$ is locally bounded, so
		\[
			\partial_y(\rho^2)(x,y)
			=\int_0^y\partial_{yy}(\rho^2)(x,s)\,ds=O(y).
		\]
		Since $D/y$ and $\rho^2$ are strictly positive on the open edge,
		their denominators stay away from zero locally.
		Hence
		\[
			\partial_y\log D=\frac1y+\frac{\partial_y(D/y)}{D/y}
			=\frac1y+O(y),\qquad
			\partial_y\log\rho=\frac{\partial_y(\rho^2)}{2\rho^2}=O(y).
		\]
		Also $\nu_y(x,0)=0$ and smoothness give $\nu_y=O(y)$.
		At $y=h$, use the even extension of $D/(h-y)$ in the same way:
		\[
			\partial_y\log D=-\frac1{h-y}+O(h-y),\qquad
			\partial_y\log\rho=O(h-y),\qquad \nu_y=O(h-y).
		\]
		Substituting into $b_y=2\partial_y\log D-4\partial_y\log\rho
		-2\Lambda\nu\nu_y/\rho^2$ yields
		\begin{equation}\label{eq:trap-base-drift}
			b_y=\frac2y+O(y)\quad(y\to0),\qquad
			b_y=-\frac2{h-y}+O(h-y)\quad(y\to h).
		\end{equation}
		These estimates are uniform on compact subsegments of each open side.
		Since $g'(y)=h-2y$ and $g''(y)=-2$, we obtain
		\[
			Lg=-2+b_y(h-2y)
			=\begin{cases}
			\displaystyle\frac{2h}{y}-6+O(y)&(y\to0),\\
			\displaystyle\frac{2h}{h-y}-6+O(h-y)&(y\to h).
			\end{cases}
		\]
		In particular, $Lg>0$ near every interior point of either horizontal edge.

		We now prove that ${\color{black}\mathcal T}>0$ in $\Omega$.
		Suppose $m=\min_{\overline\Omega}{\color{black}\mathcal T}<0$; this minimum exists by
		Lemma~\ref{lem:trap-normalized-boundary}.
		Set
		\[
			{\color{black}\mathcal T}_\varepsilon={\color{black}\mathcal T}-\varepsilon g,
		\]
		and choose $\varepsilon>0$ so small that $-\varepsilon h^2/4>m$.
		The minimum of ${\color{black}\mathcal T}_\varepsilon$ is at most $m$.
		It cannot occur on $P_2P_3$ or $P_1P_4$, where
		${\color{black}\mathcal T}_\varepsilon=-\varepsilon g\ge-\varepsilon h^2/4>m$,
		or at a vertex, where ${\color{black}\mathcal T}_\varepsilon=0$.
		On the open edges $P_1P_2,P_4P_3$, respectively,
		\[
			\partial_y\mathcal T_\varepsilon(x,0)=-\varepsilon h<0,\qquad
			\partial_y\mathcal T_\varepsilon(x,h)=\varepsilon h>0,
		\]
		so its inward derivative is negative on both edges and no minimum
		occurs there either. Thus ${\color{black}\mathcal T}_\varepsilon$ attains
		its minimum at an interior point $p_\varepsilon$.

		Along a sequence $\varepsilon\downarrow0$, pass to a subsequence
		with $p_\varepsilon\to p_0\in\overline\Omega$.
		Since
		\[
			m\le {\color{black}\mathcal T}(p_\varepsilon)\le m+\varepsilon h^2/4,
		\]
		continuity gives ${\color{black}\mathcal T}(p_0)=m$.
		If $p_0\in\Omega$, the minimum conditions and $L\mathcal T<0$ give
		\begin{equation}\label{eq:trap-minimum-contradiction}
			0\le\Delta\mathcal T(p_0)=L\mathcal T(p_0)<0,
		\end{equation}
		a contradiction. Since ${\color{black}\mathcal T}=0$ on $P_2P_3,P_1P_4$ and at the vertices,
		$p_0$ must lie in the relative interior of $P_1P_2$ or $P_4P_3$.
		The boundary estimate above gives $Lg(p_\varepsilon)>0$ for sufficiently small $\varepsilon$,
		since $p_\varepsilon\to p_0$ and $p_0$ is not a vertex.
		At the interior minimum $p_\varepsilon$, we have
		$\nabla {\color{black}\mathcal T}=\varepsilon\nabla g$ and
		$\Delta {\color{black}\mathcal T}\ge\varepsilon\Delta g$. For sufficiently small
		$\varepsilon$, this implies
		\[
			L\mathcal T(p_\varepsilon)\ge\varepsilon Lg(p_\varepsilon)>0,
		\]
		contrary to $L\mathcal T<0$. Hence ${\color{black}\mathcal T}\ge0$ on $\overline\Omega$.

		An interior zero would now be a minimum and would give the same
		contradiction as \eqref{eq:trap-minimum-contradiction}.
		Hence ${\color{black}\mathcal T}>0$ throughout $\Omega$.

		Since $J=D\mathcal T>0$, the gradients $\nabla\varphi$ and $\nabla\nu$ are linearly independent
		at every interior point.
		Since $\varphi,\nu$ form a basis of the second Neumann eigenspace
		on $\Omega$, every nonzero second Neumann eigenfunction
		has the form $u=a_1\varphi+a_2\nu$ with $(a_1,a_2)\ne(0,0)$.
		Its gradient $a_1\nabla\varphi+a_2\nabla\nu$ cannot vanish in $\Omega$.
		Therefore $u$ has no interior critical point, and its global extrema
		can occur only on $\partial\Omega$.
	\end{proof}

	\subsection{Kites at the critical height}\label{subsec:kite-multiplicity-two}

	\textcolor{black}{We prove Theorem~\ref{thm1.3} for kites at the critical height,}
	with the notation of Figure~\ref{fig:review-geometry}(b), (d).

	\begingroup\color{black}
	Let $0<a<1$ or $1<a<2$, put $h=\widetilde h(a)$ and
	$\Lambda=\mu_2(K_{a,h})=\mu(a,h)=\lambda(a,h)$.
	By Theorem~\ref{thm1.2}, the second Neumann eigenspace on $K_{a,h}$
	has a basis $\varphi,\nu$ that is symmetric and antisymmetric,
	respectively, with respect to the $x$-axis.

	We use the same symbols for their restrictions to the upper triangle
	$T=P_1P_3P_4$.
	On $T$, $\varphi$ is a second Neumann eigenfunction, whereas $\nu$
	is a first mixed eigenfunction with Dirichlet condition on $P_1P_3$
	and Neumann conditions on $P_1P_4\cup P_3P_4$.
	Normalize these restrictions as in \eqref{eq:kite-normalization},
	with $\nu>0$ in $T$.
	\endgroup

	\textcolor{black}{Every nonzero second Neumann eigenfunction
	on $K_{a,h}$ can therefore be written as
	$u=a_1\varphi+a_2\nu$, with $(a_1,a_2)\ne(0,0)$.}
	The cases $a_1=0$ and $a_2=0$ follow from
	Proposition~\ref{prop:pure-modes-interior}.
	For $a_1a_2\ne0$, divide by $a_1$ and, if $a_2/a_1<0$,
	reflect across $P_1P_3$ using $S(x,y)=(x,-y)$.
	Scaling preserves critical points, and reflection preserves them
	up to $S$, since
	\begin{equation}\label{eq:kite-reflected-gradient}
		\nabla(u\circ S)(z)=
		\begin{pmatrix}1&0\\0&-1\end{pmatrix}\nabla u(Sz),
		\qquad |\nabla(u\circ S)(z)|=|\nabla u(Sz)|.
	\end{equation}
	Together with $\varphi\circ S=\varphi$ and $\nu\circ S=-\nu$,
	\textcolor{black}{\eqref{eq:kite-reflected-gradient}} allows us to take
	\begin{equation}\label{eq:kite-critical-combination}
		u=\varphi+t\nu,\qquad t=|a_2/a_1|>0.
	\end{equation}
	We use the rotational derivative
	$\mathcal R:=\mathcal R_{P_4}=(h-y)\partial_x+(x-a)\partial_y$ about $P_4$.
	For convex kites, directional and rotational signs first restrict
	where a critical point could occur; a radial comparison function
	then excludes it by Lemma~\textcolor{black}{\ref{lem.criticalpoint-comparison}}.
	For concave kites, directional and rotational derivatives alone
	suffice to exclude critical points in the upper and lower triangles.

	\subsubsection[Convex kites: \texorpdfstring{$0<a<1$}{0 < a < 1}]{Convex kites: $0<a<1$}\label{subsubsec:convex-kite-critical}

	\begingroup\color{black}
	The reflection $F(x,y)=(1-x,y)$ maps
	$K_{a,h}$ onto $K_{1-a,h}$ and preserves both boundary value problems
	and critical points. Thus $\mu(a,h)=\mu(1-a,h)$ and
	$\lambda(a,h)=\lambda(1-a,h)$; uniqueness in Theorem~\ref{thm1.2}
	gives $\widetilde h(a)=\widetilde h(1-a)$.
	Apart from the square case $a=1/2$, it therefore suffices to treat
	$1/2<a<1$.
	\endgroup

	When $a=1/2$, the critical-height assumption and Theorem~\ref{thm1.2}
	force $h=\widetilde h(1/2)=1/2$, so the kite is a square.
	Its second Neumann eigenvalue has multiplicity two, and separation of
	variables gives every second Neumann eigenfunction in the form
	\begin{equation}\label{eq:kite-square-eigenfunction-gradient}
		\begin{aligned}
			u&=a_1\cos\bigl(\pi(x+y)\bigr)+a_2\cos\bigl(\pi(x-y)\bigr),\\
			\nabla u&=-\pi a_1\sin\bigl(\pi(x+y)\bigr)(1,1)
			-\pi a_2\sin\bigl(\pi(x-y)\bigr)(1,-1).
		\end{aligned}
	\end{equation}
	Here $0<x+y,x-y<1$ in the interior. Since $(1,1)$ and $(1,-1)$ are
	linearly independent and both sine factors are positive there, the gradient
	cannot vanish unless $a_1=a_2=0$. Thus every nonzero second Neumann
	eigenfunction of the square has no interior critical point.

	For $1/2<a<1$, put $b=1-a$.
	The critical height $h$ satisfies the following bounds:
	\begin{equation}\label{ext.eq:height}
		b<\sqrt{(a^2-a+1)/3}\le h<a.
	\end{equation}
	The lower bound in \textcolor{black}{\eqref{ext.eq:height}} follows from
	\cite[Lemma~1]{Siudeja2015MZ}, which gives $h^2\ge(a^2-a+1)/3>b^2$.

	For the upper bound in \textcolor{black}{\eqref{ext.eq:height}}, note that
	if $h\ge a$, the boundary graphs $x=a|y|/h$ and $x=1-b|y|/h$
	have Lipschitz constants $a/h,b/h\le1$, so $K_{a,h}$ is a lip domain.
	Since $a>1/2$, it is not a square, and Lemma~\ref{lem.lip} gives
	simplicity of its second Neumann eigenvalue whenever $h\ge a$.
	Hence the critical height satisfies $h<a$.
	By \textcolor{black}{\eqref{ext.eq:height}}, $h^2\ge(a^2-a+1)/3>ab$, so $T$ is acute.
	Moreover, $|P_3P_4|^2=b^2+h^2<b^2+a^2<1$ and
	$|P_1P_4|^2-|P_3P_4|^2=2a-1>0$, so $P_3P_4$ is uniquely shortest.
	If $T$ is isosceles, its equal sides are $P_1P_3,P_1P_4$ and its apex
	angle is $\arctan(h/a)<\pi/4$; hence it is subequilateral.

	Put ${\color{black}\mathcal D}=h\partial_x+b\partial_y$.
	This direction is normal to the shortest side $P_3P_4$.
	\textcolor{black}{To exclude critical points of $u$ in the lower
	triangle $S(T)$, we estimate ${\color{black}\mathcal D}(u\circ S)={\color{black}\mathcal D}(\varphi-t\nu)$ in $T$.}

	\begin{Lemma}\label{ext.lem:signs}
		We have $\varphi(P_1)\ne0$. Choose the sign of $\varphi$ so that
		$\varphi(P_1)>0$. Then $\varphi(P_3),\varphi(P_4)<0$, and
		\[
			\begin{gathered}
				\varphi_x<0\quad\text{on }P_1P_3,\\
				{\color{black}\mathcal D}\varphi<0<{\color{black}\mathcal D}\nu,\qquad
				\mathcal R\varphi<0\quad\text{in }T.
			\end{gathered}
		\]
		Moreover, $u$ in \textcolor{black}{\eqref{eq:kite-critical-combination}} has no critical
		point in the lower triangle $S(T)$.
	\end{Lemma}
	\begin{proof}
		Since $P_3P_4$ is uniquely shortest, $P_1$ is an endpoint of a longest
		side. By \cite[Theorem~1.2(4)]{Chen2026Invent}, $\varphi$ attains a global
		extremum at $P_1$. Its zero mean gives $\varphi(P_1)\ne0$, so we may
		choose $\varphi(P_1)>0$.
		By \cite[Theorem~1.2(3)]{Chen2026Invent}, the nodal arc of $\varphi$ joins
		$P_1P_3$ to $P_1P_4$ and separates $P_1$ from $P_3$ and $P_4$.
		With the sign chosen so that $\varphi(P_1)>0$, we therefore have
		$\varphi(P_3)<0$ and $\varphi(P_4)<0$.
		The vector $(h,b)$ is normal to the uniquely shortest side $P_3P_4$.
		The strict directional monotonicity in
		\cite[Theorem~1.2(5)]{Chen2026Invent} shows that
		$h\varphi_x+b\varphi_y$ has a fixed strict sign in $T$.
		Let $H_0=P_1+s_0(h,b)$ be the foot of the altitude from $P_1$ to
		$P_3P_4$, where $s_0=h/(h^2+b^2)>0$. Since $T$ is acute,
		$H_0$ lies on the open side $P_3P_4$, which belongs to the negative
		nodal domain. Therefore
		\[
			0>\varphi(H_0)-\varphi(P_1)
			=\int_0^{s_0}(h\varphi_x+b\varphi_y)
			\bigl(P_1+s(h,b)\bigr)\,ds.
		\]
		Thus ${\color{black}\mathcal D}\varphi=h\varphi_x+b\varphi_y<0$ in $T$.
		Put $\tau=(a,h)/|P_1P_4|$, directed from $P_1$ to $P_4$.
		The Neumann conditions and continuity to the open sides give
		\[
			\begin{aligned}
				h\varphi_x+b\varphi_y&=h\varphi_x\le0
				&&\text{on }P_1P_3,\\
				h\varphi_x+b\varphi_y&=\frac{h}{|P_1P_4|}\partial_\tau\varphi\le0
				&&\text{on }P_1P_4.
			\end{aligned}
		\]
		Neither open side contains a critical point, by
		\cite[Theorem~1.2(1)]{Chen2026Invent}. Since the normal derivative is zero,
		the tangential derivative cannot vanish. Thus
		$\varphi_x<0$ on $P_1P_3$ and $\partial_\tau\varphi<0$ on $P_1P_4$.

		The directional derivative ${\color{black}\mathcal D}\nu=h\nu_x+b\nu_y$ satisfies
		$(\Delta+\Lambda)({\color{black}\mathcal D}\nu)=0$.
		Since the Neumann vertex $P_4$ is acute,
		\cite[Proposition~7.2]{Hatcher2024PAMS}, applied with Dirichlet side $P_1P_3$,
		gives $\nu_y>0$ in $T$ and excludes critical points on either open
		Neumann side. Its unique maximum is $P_4$, whereas $\nu(P_1)=\nu(P_3)=0$.
		Thus the tangential derivative toward $P_4$ is strictly positive on
		both open Neumann sides. Together with the Hopf lemma and the Neumann
		condition on $P_1P_4$, this yields
		\[
			{\color{black}\mathcal D}\nu=b\nu_y>0\ \text{on }P_1P_3,\qquad
			{\color{black}\mathcal D}\nu=\frac{h}{|P_1P_4|}\partial_\tau\nu>0\ \text{on }P_1P_4,\qquad
			{\color{black}\mathcal D}\nu=0\ \text{on }P_3P_4.
		\]
		Since $T$ is acute, \textcolor{black}{\eqref{eq:corner-regularity}} and
		\textcolor{black}{\eqref{eq:corner-gradient-bounds}} give
		$\nu\in H^2(T)\cap C^1(\overline T)$ and $\nabla\nu=0$ at the vertices.
		In particular, ${\color{black}\mathcal D}\nu=0$ there and ${\color{black}\mathcal D}\nu\in C(\overline T)$. Applying
		Lemmas~\ref{lem2.3} and
		\ref{lem2.1} to $-{\color{black}\mathcal D}\nu$ gives
		${\color{black}\mathcal D}\nu>0$ in $T$.

		Finally,
		\[
			\begin{cases}
				(\Delta+\Lambda)(\mathcal R\varphi)=0 &\text{in }T,\\
				\mathcal R\varphi=0 &\text{on }P_1P_4\cup P_3P_4,\\
				\mathcal R\varphi=h\varphi_x<0 &\text{on }P_1P_3.
			\end{cases}
		\]
		Equation~\textcolor{black}{\eqref{eq:corner-regularity}} gives continuity at the vertices.
		Since $\Lambda<\lambda_1(\Delta,T)$, Lemmas~\ref{lem2.3} and
		\ref{lem2.1} give $\mathcal R\varphi<0$ in $T$.

		For $u=\varphi+t\nu$ with $t>0$, we have
		$u\circ S=\varphi-t\nu$. The directional inequalities proved above give
		\begin{equation}\label{eq:kite-lower-half-directional-sign}
			\mathcal D(u\circ S)=\mathcal D(\varphi-t\nu)=\mathcal D\varphi-t\mathcal D\nu<0
			\quad\text{in }T.
		\end{equation}
		Thus $u\circ S$ has no critical point in $T$, and the gradient
		transformation \textcolor{black}{\eqref{eq:kite-reflected-gradient}} excludes critical points of $u$ in the lower half.
	\end{proof}

	Lemma~\textcolor{black}{\ref{ext.lem:signs}} excludes critical points of $u$ in $S(T)$.
	On the symmetry segment,
	$u_y(x,0)=\varphi_y(x,0)+t\nu_y(x,0)=t\nu_y(x,0)>0$.
	Thus any interior critical point must lie in $T$.
	If $p\in T$ is a critical point of $u$, then
	\begin{equation}\label{eq:kite-critical-point-rotation}
		0=\mathcal Ru(p)
		=\mathcal R\varphi(p)+t\mathcal R\nu(p),
		\qquad
		\mathcal R\nu(p)=-\frac{\mathcal R\varphi(p)}{t}>0.
	\end{equation}

	The next lemma uses reflection about rays through $P_4$ to
	restrict such a point to the right of the altitude $x=a$.
	\begin{Lemma}\label{ext.prop:verticalcut}
		If $p\in T$ satisfies $\mathcal R\nu(p)>0$, then $\nu_x(p)<0$.
	\end{Lemma}
	\begin{proof}
		We use the properties established in the proof of Lemma~\textcolor{black}{\ref{ext.lem:signs}}:
		$\nu\in H^2(T)\cap C^1(\overline T)$, its gradient vanishes at the
		vertices, $\nu_y>0$ in $T$ and on the open side $P_1P_3$, and its tangential
		derivative toward $P_4$ is positive on the open sides $P_1P_4$ and $P_3P_4$.
		We first prove $\mathcal R\nu\le0$ in $T\cap\{x\le a\}$,
		which provides the boundary inequality needed on $x=a$ to prove
		$\nu_x<0$ in $T\cap\{x>a\}$.
		Let $A=(a,0)$, ${\color{black}\eta}=\arctan(b/h)$ and
		$\gamma=\angle P_3P_4P_1={\color{black}\eta}+\arctan(a/h)<\pi/2$.
		Since $a>b$, we have $0<{\color{black}\eta}<\gamma/2$.
		Use polar coordinates centered at $P_4$, with $\theta=0$ on the ray
		$P_4P_3$ and $\theta$ increasing clockwise toward $P_4P_1$. Then
		\begin{equation}\label{eq:kite-clockwise-polar-coordinates}
			T=\{(r,\theta):0<\theta<\gamma,\ 0<r<R(\theta)\},\qquad
			R(\theta)=\frac{h}{\cos(\theta-{\color{black}\eta})}.
		\end{equation}
		In these coordinates,
		\begin{equation}\label{eq:kite-polar-rotation}
			\partial_\theta=-(h-y)\partial_x-(x-a)\partial_y=-\mathcal R,
			\qquad x\le a\ \Longleftrightarrow\ \theta\ge{\color{black}\eta}.
		\end{equation}
		The ray $\theta={\color{black}\eta}$ meets the side $P_1P_3$ at $A$.

Fix an angle $\theta_0\in({\color{black}\eta},\gamma/2]$, and let $D_{\theta_0}$ be the point where the ray
		$\theta=\theta_0$ meets $P_1P_3$.
		We reflect the triangle $U_{\theta_0}=\triangle P_3P_4D_{\theta_0}$ across the line
		$P_4D_{\theta_0}$; in polar coordinates, this reflection keeps $r$ fixed
		and replaces $\theta$ by $2\theta_0-\theta$:
		\[
			U_{\theta_0}=\{(r,\theta):0<\theta<\theta_0,\ 0<r<R(\theta)\},\qquad
			S_{\theta_0}(r,\theta)=(r,2\theta_0-\theta).
		\]
		We first verify that $S_{\theta_0}(U_{\theta_0})\subset T$.
		For $0\le\theta\le \theta_0$,
		\[
			(2\theta_0-\theta-{\color{black}\eta})^2-(\theta-{\color{black}\eta})^2
			=4(\theta_0-{\color{black}\eta})(\theta_0-\theta)\ge0.
		\]
		Hence
		\[
			R(2\theta_0-\theta)=\frac{h}{\cos(2\theta_0-\theta-{\color{black}\eta})}
			\ge\frac{h}{\cos(\theta-{\color{black}\eta})}=R(\theta).
		\]
		Also $2\theta_0-\theta\in[\theta_0,2\theta_0]\subset[0,\gamma]$. Hence
		$S_{\theta_0}(U_{\theta_0})\subset T$. Since $\theta_0>{\color{black}\eta}$, the radial inequality is strict
		for $0\le\theta<\theta_0$; thus the open segment $P_3D_{\theta_0}$ is reflected into $T$.

		We first prove the following conditional reflection statement:
		if $\mathcal R\nu\le0$ on the ray $\theta=2\theta_0$, then
		$\mathcal R\nu<0$ in the sector $\theta_0\le\theta<\gamma$.
		Suppose that $\mathcal R\nu\le0$ on this ray, and put
		$v_{\theta_0}=\nu\circ S_{\theta_0}-\nu$ in $U_{\theta_0}$. Reflection preserves the equation,
		$v_{\theta_0}=0$ on the reflecting ray, and $\nu=0$ on the side $P_1P_3$. On $P_3P_4$,
		the outward normal points toward decreasing $\theta$ and
		$\nu_\theta(r,0)=0$, giving
		\begin{equation}\label{eq:kite-moving-reflection-flux}
			\partial_n v_{\theta_0}(r,0)=\frac1r\nu_\theta(r,2\theta_0)\ge0.
		\end{equation}
		Thus
		\begin{equation}\label{eq:kite-moving-reflection-problem}
			\left\{
			\begin{aligned}
				&(\Delta+\Lambda)v_{\theta_0}=0\quad\text{in }U_{\theta_0},\\
				&v_{\theta_0}=0\ \text{on }P_4D_{\theta_0},\\
				&v_{\theta_0}>0\ \text{on the open segment }P_3D_{\theta_0},\\
				&\partial_n v_{\theta_0}\ge0\quad\text{on }P_3P_4.
			\end{aligned}
			\right.
		\end{equation}

		To obtain $v_{\theta_0}\ge0$, we use the mixed Rayleigh principle,
		as in Lemma~\textcolor{black}{\ref{lem:trap-second-derivative}}.
		The regularity established in the proof of Lemma~\ref{ext.lem:signs}
		is preserved under reflection, so testing $v_{\theta_0}$ with its negative part is valid.
		Put $v_{\theta_0}^-=\max\{-v_{\theta_0},0\}$.
		It vanishes on $P_4D_{\theta_0}\cup P_3D_{\theta_0}$, since $v_{\theta_0}=0$ on the former
		and $v_{\theta_0}\ge0$ on the latter.
		Testing \textcolor{black}{\eqref{eq:kite-moving-reflection-problem}} with $v_{\theta_0}^-$ gives
		\begin{equation}\label{eq:kite-moving-reflection-energy}
			\int_{U_{\theta_0}}(|\nabla v_{\theta_0}^-|^2-\Lambda(v_{\theta_0}^-)^2)\,d\sigma
			=-\int_{P_3P_4}v_{\theta_0}^-\partial_n v_{\theta_0}\,ds\le0.
		\end{equation}
		Extend $v_{\theta_0}^-$ by zero across $P_4D_{\theta_0}$ to the rest of $T$, and denote the
		extension by $\widetilde{v_{\theta_0}^-}$.
		Because $v_{\theta_0}^-=0$ on $P_4D_{\theta_0}$, this extension belongs to $H^1(T)$.
		It also vanishes on $P_1P_3$, so it is admissible in the variational
		characterization of $\lambda(a,h)$.
		Since $\lambda(a,h)=\Lambda$, the variational principle gives
		\[
			0\le\int_T\bigl(|\nabla\widetilde{v_{\theta_0}^-}|^2
			-\Lambda(\widetilde{v_{\theta_0}^-})^2\bigr)\,d\sigma
			=\int_{U_{\theta_0}}\bigl(|\nabla v_{\theta_0}^-|^2-\Lambda(v_{\theta_0}^-)^2\bigr)\,d\sigma
			\le0.
		\]
		If $\widetilde{v_{\theta_0}^-}\not\equiv0$, equality makes it a minimizer of the
		mixed Rayleigh quotient, hence a multiple of $\nu$ by simplicity.
		But it vanishes on $\triangle P_1P_4D_{\theta_0}$, whereas $\nu>0$ there.
		The multiple must therefore be zero, a contradiction.
		Thus $v_{\theta_0}^-=0$ and $v_{\theta_0}\ge0$.
		Since $v_{\theta_0}>0$ on the open segment $P_3D_{\theta_0}$,
		we have $v_{\theta_0}\not\equiv0$, so
		Lemma~\ref{lem2.1} yields $v_{\theta_0}>0$ in $U_{\theta_0}$.
		On the reflecting ray, the outward normal points toward increasing
		$\theta$. Lemma~\ref{lem2.4} consequently gives
		\[
			\partial_n v_{\theta_0}=-\frac2r\nu_\theta
			=\frac2r\mathcal R\nu<0
			\quad\text{on }P_4D_{\theta_0}.
		\]

		We next prove $\mathcal R\nu<0$ in $\triangle P_1P_4D_{\theta_0}$.
		Since $\theta_0>{\color{black}\eta}$, the $x$-coordinate of $D_{\theta_0}$ is less than $a$.
		On the open segment $P_1D_{\theta_0}$, we have $\nu_x=0$ because $\nu=0$
		on $P_1P_3$, and $\nu_y>0$ by the Hopf lemma. Thus
		$\mathcal R\nu=(x-a)\nu_y<0$ there.
		The preceding normal-derivative inequality gives $\mathcal R\nu<0$
		on the open segment $P_4D_{\theta_0}$, while the Neumann condition gives
		$\mathcal R\nu=0$ on $P_1P_4$. Consequently,
		\[
			\begin{cases}
				(\Delta+\Lambda)(\mathcal R\nu)=0
				&\text{in }\triangle P_1P_4D_{\theta_0},\\
				\mathcal R\nu<0&\text{on }P_1D_{\theta_0}\cup P_4D_{\theta_0},\\
				\mathcal R\nu=0&\text{on }P_1P_4.
			\end{cases}
		\]
		The $C^1$ regularity of $\nu$ ensures that $\mathcal R\nu$ is continuous
		on the closed triangle. Since
		$\Lambda<\lambda_1(\Delta,T)\le\lambda_1(\Delta,\triangle P_1P_4D_{\theta_0})$,
		Lemma~\ref{lem2.2} therefore gives
		$\mathcal R\nu<0$ for $\theta_0\le\theta<\gamma$.

		We now remove the assumption that $\mathcal R\nu\le0$ on the ray
		$\theta=2\theta_0$ and prove $\mathcal R\nu\le0$ for ${\color{black}\eta}\le\theta<\gamma$.
		First take $\theta_0=\gamma/2$. The ray $\theta=2\theta_0=\gamma$
		is $P_4P_1$, where the Neumann condition gives $\mathcal R\nu=0$.
		The conditional reflection statement therefore applies and yields
		$\mathcal R\nu<0$ for $\gamma/2\le\theta<\gamma$.
		We claim that $\mathcal R\nu\le0$ for ${\color{black}\eta}\le\theta<\gamma$.
		Suppose, for contradiction, that $\mathcal R\nu(r,\theta)>0$
		for some $\theta\in[{\color{black}\eta},\gamma)$ and $0<r<R(\theta)$. Define
		\[
			m=\sup\{\theta\in[{\color{black}\eta},\gamma):
			\mathcal R\nu(r,\theta)>0
			\text{ for some }0<r<R(\theta)\}.
		\]
		If $\mathcal R\nu(r,{\color{black}\eta})>0$, continuity gives some
		$\theta\in({\color{black}\eta},\gamma)$ with $r<R(\theta)$ and
		$\mathcal R\nu(r,\theta)>0$. Thus $m>{\color{black}\eta}$.
		Since $\mathcal R\nu<0$ for $\gamma/2\le\theta<\gamma$, we also have
		$m\le\gamma/2$.
		Choose $\max\{{\color{black}\eta},m/2\}<\theta_0<m$; this is possible because $m>{\color{black}\eta}$.
		Then $m<2\theta_0<\gamma$. By the definition of $m$, no positive value
		can occur at an angle greater than $m$, so $\mathcal R\nu\le0$
		on the ray $\theta=2\theta_0$.
		Applying the conditional reflection statement gives $\mathcal R\nu<0$
		for $\theta_0\le\theta<\gamma$.
		But $\theta_0<m$ and the definition of the supremum guarantee a positive
		value at some angle greater than $\theta_0$, a contradiction.
		This proves the claim, or equivalently
		$\mathcal R\nu\le0$ on $T\cap\{x\le a\}$.

		On $AP_4$, this gives $\mathcal R\nu=(h-y)\nu_x\le0$.
		Hence $\nu_x\le0$ on the open segment $AP_4$.
		On $AP_3$, differentiating $\nu(x,0)=0$ gives $\nu_x=0$.
		On $P_3P_4$, let $\tau=(-b,h)/|P_3P_4|$ point toward $P_4$.
		The proof of Lemma~\ref{ext.lem:signs} gives $\partial_\tau\nu>0$
		on this open side. Together with the Neumann condition, this yields
		$\nu_x=-b\,\partial_\tau\nu/|P_3P_4|<0$.
		Thus $\nu_x$ satisfies
		\[
			\begin{cases}
				(\Delta+\Lambda)\nu_x=0&\text{in }\triangle AP_3P_4,\\
				\nu_x=0&\text{on }AP_3,\\
				\nu_x<0&\text{on }P_3P_4,\\
				\nu_x\le0&\text{on }AP_4.
			\end{cases}
		\]
		By continuity at the vertices and
		$\Lambda<\lambda_1(\Delta,T)\le\lambda_1(\Delta,\triangle AP_3P_4)$,
		Lemma~\ref{lem2.2} yields $\nu_x<0$ in $\triangle AP_3P_4$.
		Thus $\mathcal R\nu(p)>0$ forces $x>a$, where $\nu_x(p)<0$.
	\end{proof}

	For a critical point $p\in T$ of $u$ in \textcolor{black}{\eqref{eq:kite-critical-combination}},
	\textcolor{black}{\eqref{eq:kite-critical-point-rotation}} and Lemma~\textcolor{black}{\ref{ext.prop:verticalcut}}
	give $\mathcal R\nu(p)>0$ and $\nu_x(p)<0$.

	To exclude these possible critical points, we need to construct the following comparison function.
	\begin{Lemma}
		\label{prop.kite.separation}
		If $p\in T$ satisfies $\mathcal R\nu(p)>0$,
		there is a function $v_0$ such that
		\[
			\begin{aligned}
				&(\Delta+\Lambda)v_0=0\quad\hbox{in }\overline K_{a,h},\qquad
				\nabla v_0(p)\parallel\nabla\nu(p),\\
				&\partial_n v_0<0\quad\hbox{on each open side of }K_{a,h}.
			\end{aligned}
		\]
	\end{Lemma}

	\begin{proof}
		Put $\ell=|P_3P_4|=(b^2+h^2)^{1/2}$.
		By \textcolor{black}{\eqref{ext.eq:height}}, $ab<h^2<a^2$.
		We first prove $\Lambda\ell^2<\pi^2$ and then construct the required comparison function.
		Put $c=(a-1/2)/h$ and choose
		\begin{equation}\label{eq:kite-direct-trial-function}
			f(x,y)=\sin\bigl(\pi(x-cy)\bigr)\sin\left(\frac{\pi y}{2h}\right).
		\end{equation}
		This function vanishes on $P_1P_3$ and is an admissible mixed trial function.
		Its gradient is
		\begin{equation}\label{eq:kite-direct-trial-derivatives}
			\nabla f=\pi\cos\bigl(\pi(x-cy)\bigr)
			\sin\left(\frac{\pi y}{2h}\right)(1,-c)
			+\frac{\pi}{2h}\sin\bigl(\pi(x-cy)\bigr)
			\cos\left(\frac{\pi y}{2h}\right)(0,1).
		\end{equation}
		For fixed $0<y<h$, the horizontal section of $T$ is
		$ay/h<x<1-by/h$, so
		\[
			\frac{y}{2h}<x-cy<1-\frac{y}{2h}.
		\]
		This interval is symmetric about $1/2$. Since
		$\sin(\pi t)\cos(\pi t)$ is odd about $t=1/2$, the cross term in
		$|\nabla f|^2$ integrates to zero on each horizontal section.
		Direct integration gives
		\begin{equation}\label{eq:kite-direct-trial-integrals}
			\int_T f^2\,d\sigma=\frac h8,\qquad
			\int_T|\nabla f|^2\,d\sigma
			=(1+c^2)\frac{h(\pi^2-8)}8+\frac{\pi^2+8}{32h}.
		\end{equation}
		The mixed variational principle and \textcolor{black}{\eqref{eq:kite-direct-trial-integrals}}
		now give
		\begin{equation}\label{eq:kite-radial-spectral-bound}
			\begin{aligned}
				\Lambda&\le\frac{\int_T|\nabla f|^2\,d\sigma}{\int_T f^2\,d\sigma}
				=(\pi^2-8)\left(1+\frac{(2a-1)^2}{4h^2}\right)
				+\frac{\pi^2+8}{4h^2}\\
				&= (\pi^2-8)\left(1-\frac{ab}{h^2}\right)+\frac{\pi^2}{2h^2}.
			\end{aligned}
		\end{equation}
		Since $h^2>ab$, $a>b$ and
		$\ell^2=h^2+b^2<a^2+b^2<1$, we have
		$(\pi^2-8)\ell^2<\pi^2/2$. Multiplying
		\textcolor{black}{\eqref{eq:kite-radial-spectral-bound}} by $\ell^2$ therefore gives
		\begin{equation}\label{eq.rigor.annular-width}
			\Lambda\ell^2
			<\frac{\pi^2(h^2-ab)}{2h^2}+\frac{\pi^2(h^2+b^2)}{2h^2}
			=\pi^2\left(1-\frac{b(a-b)}{2h^2}\right)<\pi^2.
		\end{equation}

		Write $p=(x,y)$. We construct $v_0$ as a radial function with center $q$,
		that is, $v_0(z)=w(|z-q|)$, where $w$ is to be determined.
		Since its gradient at $p$ is a scalar multiple of $p-q$, the required condition
		$\nabla v_0(p)\parallel\nabla\nu(p)$ is satisfied by choosing
		$q$ so that $p-q\parallel\nabla\nu(p)$.
		We also choose the circle centered at $q$ to pass through $P_1$ and $P_3$.
		Thus $q$ lies on their perpendicular bisector and has the form $q=(1/2,Q)$,
		where $Q$ is its vertical coordinate, to be determined.
		By Lemma~\textcolor{black}{\ref{ext.lem:signs}}, $h\nu_x+b\nu_y>0$.
		Lemma~\ref{ext.prop:verticalcut} gives $\nu_x(p)<0$, so
		$\nu_y(p)/(-\nu_x(p))>h/b$.
		The condition $p-q\parallel\nabla\nu(p)$ determines $Q$; writing $\rho$
		for the radius of this circle, we obtain
		\begin{equation}\label{eq:kite-comparison-circle}
			Q=y+\left(x-\frac12\right)\frac{\nu_y(p)}{-\nu_x(p)},
			\qquad q=(1/2,Q),\qquad
			\rho=(Q^2+1/4)^{1/2}.
		\end{equation}
		Then $|q-P_1|=|q-P_3|=\rho$.
		Equation~\textcolor{black}{\eqref{eq:kite-comparison-circle}} and the
		hypothesis $\mathcal R\nu(p)>0$ give
		\[
			Q-h=\frac{\mathcal R\nu(p)+(a-1/2)\nu_y(p)}{-\nu_x(p)}
			>\frac{(a-1/2)h}{b},\qquad
			Q>\frac{h}{2b}>h.
		\]
		Thus $q\notin\overline K_{a,h}$. By the definitions of $q$ and $Q$,
		\begin{equation}\label{eq.rigor.gradient-line}
			q-p=\left(\frac12-x,\ Q-y\right)
			=\left(\frac12-x,\ \left(x-\frac12\right)
			\frac{\nu_y(p)}{-\nu_x(p)}\right)
			=\frac{1/2-x}{\nu_x(p)}\nabla\nu(p).
		\end{equation}

		We verify the boundary conditions separately on the upper and lower sides
		of $K_{a,h}$.

		First consider the open sides $P_1P_4$ and $P_3P_4$.
		Since $2hQ>h^2/b>a>b$, we have
		$|P_4-q|^2-\rho^2=h^2-ab-2hQ<0$.
		Thus $|P_4-q|<\rho$. For $i=1,3$, a point on the open side $P_iP_4$
		has the form $z=(1-t)P_i+tP_4$, with $0<t<1$. Since $|P_i-q|=\rho$,
		\[
			|z-q|\le(1-t)|P_i-q|+t|P_4-q|
			<(1-t)\rho+t\rho=\rho.
		\]
		Since $aQ>bQ>h/2$, the outward unit normals give
		\begin{equation}\label{eq:kite-circle-normal-signs}
			(z-q)\cdot n=
			\begin{cases}
				(h/2-aQ)/|P_1P_4|<0 &\text{on }P_1P_4,\\
				(h/2-bQ)/\ell<0 &\text{on }P_3P_4.
			\end{cases}
		\end{equation}
		Thus $|z-q|<\rho$ and $(z-q)\cdot n<0$ on both open sides.

		Next consider the open sides $P_1P_2$ and $P_3P_2$.
		For $0<s<1$, points on these sides satisfy
		\[
			\begin{aligned}
				|sP_2-q|^2-\rho^2&=s\bigl(s(a^2+h^2)+2hQ-a\bigr)>0,\\
				|P_3+s(P_2-P_3)-q|^2-\rho^2
				&=s\bigl(s(b^2+h^2)+2hQ-b\bigr)>0.
			\end{aligned}
		\]
		Hence $|z-q|>\rho$ on the open sides $P_1P_2$ and $P_3P_2$.
		The outward unit normals give
		\begin{equation}\label{eq:kite-circle-lower-normal-signs}
			(z-q)\cdot n=
			\begin{cases}
				(h/2+aQ)/|P_1P_2|>0 &\text{on }P_1P_2,\\
				(h/2+bQ)/\ell>0 &\text{on }P_3P_2.
			\end{cases}
		\end{equation}
		Thus $|z-q|>\rho$ and $(z-q)\cdot n>0$ on both open sides.

		We next prove the distance estimate
		\[
			\bigl||z-q|-\rho\bigr|\le\ell<\frac{\pi}{\sqrt\Lambda}
			\qquad(z\in\overline K_{a,h}).
		\]
		For any point $z=(\xi,\eta)\in\overline K_{a,h}$, we have
		$h\xi-a\eta\ge0$, since the kite lies in this half-plane bounded by $P_1P_4$.
		Hence $a-\xi\le a(h-\eta)/h$. Since $h^2>ab$,
		$Q-h>(a-1/2)h/b>a(a-1/2)/h$, and hence
		\[
			(q-P_4)\cdot(z-P_4)
			=(a-1/2)(a-\xi)-(Q-h)(h-\eta)
			\le\bigl(\tfrac{a(a-1/2)}{h}-(Q-h)\bigr)(h-\eta)\le0.
		\]
		Consequently,
		\[
			|q-z|^2=|q-P_4|^2+|z-P_4|^2
			-2(q-P_4)\cdot(z-P_4)\ge|q-P_4|^2.
		\]
		Thus $P_4$ is the nearest point. Also,
		$|q-P_2|^2-\rho^2=h^2-ab+2hQ>0$, so
		$|q-P_4|<\rho=|q-P_1|=|q-P_3|<|q-P_2|$.
		Convexity of the distance function then gives
		$|q-P_4|\le|q-z|\le|q-P_2|$ throughout the closed kite.
		The reverse triangle inequality and
		$|P_3P_4|=|P_3P_2|=\ell$ yield
		\[
			\rho-|q-P_4|\le\ell,\qquad |q-P_2|-\rho\le\ell.
		\]
		Combining these bounds with \textcolor{black}{\eqref{eq.rigor.annular-width}}, we obtain
		\begin{equation}\label{eq.rigor.radial-widths}
			\bigl||z-q|-\rho\bigr|\le\ell<\frac{\pi}{\sqrt\Lambda}
			\qquad(z\in\overline K_{a,h}).
		\end{equation}
		Put $k=\sqrt\Lambda$ and let $w$ be the solution of the following initial-value problem:
		\begin{equation}\label{eq:kite-radial-initial-value}
			w''+r^{-1}w'+k^2w=0,\qquad w(\rho)=1,\quad w'(\rho)=0.
		\end{equation}
		Differentiating the equation for $w$ and setting $g(r)=\sqrt r\,w'(r)$,
		we obtain
		\begin{equation}\label{eq:kite-radial-derivative-equation}
			g''+\left(k^2-\frac{3}{4r^2}\right)g=0,\qquad
			g(\rho)=0,\quad g'(\rho)=-k^2\sqrt\rho<0.
		\end{equation}
		Let $r_1<r_2$ be consecutive positive zeros of $g$.
		Multiplying \textcolor{black}{\eqref{eq:kite-radial-derivative-equation}}
		by $g$ and integrating over $(r_1,r_2)$ give
		\[
			\int_{r_1}^{r_2}gg''\,dr
			+\int_{r_1}^{r_2}\left(k^2-\frac{3}{4r^2}\right)g^2\,dr=0.
		\]
		Since $g(r_1)=g(r_2)=0$, integration by parts yields
		\[
			\int_{r_1}^{r_2}gg''\,dr
			=[gg']_{r_1}^{r_2}-\int_{r_1}^{r_2}|g'|^2\,dr
			=-\int_{r_1}^{r_2}|g'|^2\,dr.
		\]
		The one-dimensional Poincar\'e inequality therefore gives
		\begin{equation}\label{eq:kite-radial-zero-spacing}
			\frac{\pi^2}{(r_2-r_1)^2}\int_{r_1}^{r_2}g^2\,dr
			\le\int_{r_1}^{r_2}|g'|^2\,dr
			=\int_{r_1}^{r_2}\left(k^2-\frac{3}{4r^2}\right)g^2\,dr
			<k^2\int_{r_1}^{r_2}g^2\,dr.
		\end{equation}
		Thus $r_2-r_1>\pi/k$. Since $g(\rho)=0$ and $g'(\rho)<0$,
		we have $g>0$ immediately to the left of $\rho$ and $g<0$ immediately to its right.
		The inequality $r_2-r_1>\pi/k$ excludes any other zero of $g$ in
		$(\max\{0,\rho-\pi/k\},\rho+\pi/k)$.
		Hence $g$ is positive on $(\max\{0,\rho-\pi/k\},\rho)$ and negative on
		$(\rho,\rho+\pi/k)$. Since $w'(r)=g(r)/\sqrt r$, we obtain
		\begin{equation}\label{eq:kite-radial-derivative-signs}
			\begin{cases}
				w'(r)>0,&\max\{0,\rho-\pi/k\}<r<\rho,\\
				w'(r)<0,&\rho<r<\rho+\pi/k.
			\end{cases}
		\end{equation}
		Set $v_0(z)=w(|z-q|)$. Since $q\notin\overline K_{a,h}$,
		$v_0$ is smooth near the closed kite and solves $(\Delta+\Lambda)v_0=0$.
		Put $r=|z-q|$. On an open upper side, $r<\rho$ and $(z-q)\cdot n<0$;
		on an open lower side, $r>\rho$ and $(z-q)\cdot n>0$.
		Equations~\textcolor{black}{\eqref{eq:kite-circle-normal-signs}},
		\textcolor{black}{\eqref{eq:kite-circle-lower-normal-signs}},
		\textcolor{black}{\eqref{eq.rigor.radial-widths}},
		and \textcolor{black}{\eqref{eq:kite-radial-derivative-signs}} therefore yield
		\begin{equation}\label{eq:kite-radial-boundary-flux}
			\partial_n v_0(z)
			=w'(|z-q|)\frac{(z-q)\cdot n}{|z-q|}<0
			\quad\hbox{on every open side}.
		\end{equation}
		Finally, \textcolor{black}{\eqref{eq.rigor.gradient-line}} gives
		\begin{equation}\label{eq:kite-radial-gradient-relation}
			\nabla v_0(p)
			=w'(|p-q|)\frac{p-q}{|p-q|}\parallel\nabla\nu(p).
		\end{equation}
	\end{proof}

	We finish by turning the radial comparison function into a solution
	with a nodal critical point, which the comparison lemma forbids.
	\begin{proof}[\textcolor{black}{Proof of Theorem~\ref{thm1.3} for critical-height convex kites}]
		Equation~\textcolor{black}{\eqref{eq:kite-square-eigenfunction-gradient}} proves the square case $a=h=1/2$.
		For $a\ne1/2$, the reflection symmetry established above
		reduces the proof to $1/2<a<1$, with $\varphi$ oriented as in
		Lemma~\ref{ext.lem:signs}.
		The eigenfunctions $u=a_1\varphi$ and $u=a_2\nu$
		have no interior critical point by
		Proposition~\ref{prop:pure-modes-interior}.

		If both coefficients are nonzero, use the normalization \textcolor{black}{\eqref{eq:kite-critical-combination}}.
		Equations~\textcolor{black}{\eqref{eq:kite-reflected-gradient}} and
		\textcolor{black}{\eqref{eq:kite-lower-half-directional-sign}}, together with the sign of $u_y$
		on the symmetry segment, show that any remaining critical point $p$ must lie in $T$.
		Equation~\textcolor{black}{\eqref{eq:kite-critical-point-rotation}} then gives
		$\mathcal R\nu(p)>0$, and Lemma~\textcolor{black}{\ref{ext.prop:verticalcut}} gives
		$\nu_x(p)<0$.

		Lemma~\textcolor{black}{\ref{prop.kite.separation}}, specifically
		\textcolor{black}{\eqref{eq:kite-radial-gradient-relation}} and \textcolor{black}{\eqref{eq:kite-radial-boundary-flux}}, provides $v_0$ with
		$\nabla v_0(p)=\kappa_0\nabla\nu(p)$ for some $\kappa_0\in\mathbb R$ and with
		$\partial_n v_0<0$ on each open side. Suppose that
		$\nabla u(p)=0$. The function $u$ satisfies $\partial_nu=0$ on each
		open side and belongs to $H^1(K_{a,h})$. Thus
		Lemma~\textcolor{black}{\ref{lem.criticalpoint-comparison}} gives $u(p)\ne0$.
		Set $c=(v_0(p)-\kappa_0\nu(p))/u(p)$ and $\psi=v_0-\kappa_0\nu-cu$.
		The three terms solve the same equation, so $\psi\in H^1(K_{a,h})$ and
		$(\Delta+\Lambda)\psi=0$. At $p$, we have $\psi(p)=0$ and $\nabla\psi(p)=0$.
		Since $u$ and $\nu$ satisfy the Neumann condition,
		$\partial_n\psi=\partial_n v_0<0$ on each open side, so $\psi\not\equiv0$.
		Lemma~\textcolor{black}{\ref{lem.criticalpoint-comparison}} now gives a contradiction.
	\end{proof}

	\subsubsection[Concave kites: \texorpdfstring{$1<a<2$}{1 < a < 2}]{Concave kites: $1<a<2$}\label{subsubsec:concave-kite-critical}

	Let $1<a<2$ and $h=\widetilde h(a)$.
	Lemma~\ref{lem:kite-large} gives $h^2<a(2-a)$, so
	\[
		|P_1P_4|>|P_1P_3|>|P_3P_4|.
	\]
	The vertex $P_3$ is obtuse. By
	\cite[Proposition~2.10(ii)--(iii)]{Chen2026Invent}, the extrema of
	$\varphi$ occur at $P_1,P_4$. Its zero mean gives
	$\varphi(P_1)\ne0$; choose its sign so that $\varphi(P_1)>0$.
	The tangential monotonicity along the intermediate side $P_1P_3$
	in the proof of that proposition gives
	\[
		\varphi_x<0\qquad
		\text{in }\overline T\setminus\{P_1,P_3,P_4\}.
	\]
	By Lemma~\ref{lem:kite-mixed-monotonicity} and the Hopf lemma,
	$\nu_x,\nu_y>0$ in $T$ and $\nu_y>0$ on the open side $P_1P_3$.

	\begin{proof}[\textcolor{black}{Proof of Theorem~\ref{thm1.3} for critical-height concave kites}]
		Let $u$ be as in \textcolor{black}{\eqref{eq:kite-critical-combination}}.
		Since $u\circ S=\varphi-t\nu$, the signs established above give
		\begin{equation}\label{eq:kite-concave-lower-half-sign}
			u_x(x,-y)=\varphi_x(x,y)-t\nu_x(x,y)<0
			\qquad\text{for }(x,y)\in T.
		\end{equation}
		Thus $u$ has no critical points in the lower half. To exclude critical
		points in the upper half, we prove
		\begin{equation}\label{eq.add7.rotation}
			\mathcal R\varphi<0,\qquad \mathcal R\nu<0\quad\hbox{in }T.
		\end{equation}
		Rotations commute with the Laplacian, and hence both rotational derivatives
		satisfy $(\Delta+\Lambda)w=0$ in $T$.  On each of the rays $P_4P_1$ and $P_4P_3$
		the vector field $(h-y,x-a)$ is normal to the ray.  The Neumann conditions
		therefore give $\mathcal R\varphi=\mathcal R\nu=0$ on
		$P_1P_4\cup P_3P_4$.  On $P_1P_3$, where $\varphi_y=0$, $\nu=0$, and
		$\nu_y>0$, we have
		\[
			\mathcal R\varphi=h\varphi_x<0,\qquad
			\mathcal R\nu=(x-a)\nu_y<0\qquad\text{on }P_1P_3.
		\]
		These boundary signs give \textcolor{black}{\eqref{eq.add7.rotation}} by the maximum
		principle once the corner behavior is controlled. All corners are
		regular for $\varphi$; only the mixed corner $P_3$ requires a separate
		argument for $\nu$.
		At each Neumann corner of $T$, its angle $\omega<\pi$ and
		\textcolor{black}{\eqref{eq:corner-gradient-bounds}} gives $\nabla\varphi\to0$.
		For $\nu$, the acute mixed corner $P_1$ has exponent
		$\pi/(2\angle P_3P_1P_4)>1$, while the Neumann estimate
		\textcolor{black}{\eqref{eq:corner-gradient-bounds}} applies at $P_4$.
		Thus $\nabla\nu\to0$ at $P_1,P_4$, and
		\begin{equation}\label{eq:kite-rotation-regular-corners}
			\mathcal R\varphi\to0\quad\text{at }P_1,P_3,P_4,\qquad
			\mathcal R\nu\to0\quad\text{at }P_1,P_4.
		\end{equation}
		In particular, $\mathcal R\varphi$ extends continuously to
		$\overline T$. Its boundary values are nonpositive and are strictly
		negative on $P_1P_3$, so Lemmas~\ref{lem2.3} and \ref{lem2.1} give
		$\mathcal R\varphi<0$ in $T$.

		For $\mathcal R\nu$, it remains to control the obtuse mixed corner
		$P_3$. Use polar coordinates $(r,{\color{black}\vartheta})$ centered at $P_3$, with
		${\color{black}\vartheta}=0$ on the ray $P_3P_4$ and ${\color{black}\vartheta}$ increasing into $T$. Put
		$\beta=\arctan(h/(a-1))$, so $0\le{\color{black}\vartheta}\le\pi-\beta$.
		For ${\color{black}\delta}=\pi/(2(\pi-\beta))\in(1/2,1)$, apply
		\textcolor{black}{\eqref{eq:corner-mixed-first}} with leading coefficient $A>0$.
		The remainder is $O(r^{3{\color{black}\delta}})$, and its first derivatives are
		$O(r^{3{\color{black}\delta}-1})$ by \textcolor{black}{\eqref{eq:corner-remainder-derivatives}}.
		Substituting \textcolor{black}{\eqref{eq:corner-mixed-first}} into
		$\mathcal R\nu=(h-y)\nu_x+(x-a)\nu_y$ and using
		\textcolor{black}{\eqref{eq:corner-remainder-derivatives}} gives
		\begin{equation}\label{eq:kite-rotation-corner-asymptotic}
			\mathcal R\nu
			=-A\delta|P_3P_4|r^{\delta-1}\sin((1-\delta)\vartheta)
			+O(r^{\min\{{\color{black}\delta},3{\color{black}\delta}-1\}}).
		\end{equation}
		Since $0\le(1-{\color{black}\delta}){\color{black}\vartheta}\le\pi/2-\beta$ and ${\color{black}\delta}>1/2$,
		the leading term in \textcolor{black}{\eqref{eq:kite-rotation-corner-asymptotic}} is
		nonpositive and the remainder tends to zero. Hence
		$\sup_{0\le{\color{black}\vartheta}\le\pi-\beta}
		(\mathcal R\nu(r,{\color{black}\vartheta}))^+\to0$ as $r\downarrow0$.
		To absorb these small positive boundary values, choose a fixed comparison
		function $V$ satisfying $(\Delta+\Lambda)V=-1$ in $T$ and $V=1$ on
		$\partial T$. It exists uniquely because $\Lambda<\lambda_1(\Delta,T)$.
		Since $(-\Delta-\Lambda)(V-1)=1+\Lambda>0$ with zero boundary trace,
		Lemma~\ref{lem2.2} gives $V\ge1$.

		Fix ${\color{black}\sigma}>0$ and put
		$T_\varepsilon=T\setminus\overline{B_\varepsilon(P_3)}$, removing the
		closed disk of radius $\varepsilon$ centered at $P_3$.
		By this limit, the circular boundary
		has $\mathcal R\nu\le{\color{black}\sigma}$ for all sufficiently small $\varepsilon$;
		the remaining boundary values are nonpositive. Consequently,
		\[
			\begin{cases}
				(\Delta+\Lambda)(\mathcal R\nu-{\color{black}\sigma} V)={\color{black}\sigma}>0
				&\text{in }T_\varepsilon,\\
				\mathcal R\nu-{\color{black}\sigma} V\le0
				&\text{on }\partial T_\varepsilon.
			\end{cases}
		\]
		Since $\Lambda<\lambda_1(\Delta,T)\le\lambda_1(\Delta,T_\varepsilon)$,
		Lemma~\ref{lem2.2} gives $\mathcal R\nu(z)\le{\color{black}\sigma} V(z)$
		for $z\in T_\varepsilon$.
		For fixed $z\in T$, let $\varepsilon\downarrow0$ and then
		${\color{black}\sigma}\downarrow0$. We obtain $\mathcal R\nu(z)\le0$.
		Lemma~\ref{lem2.1} makes the inequality strict because its trace on
		$P_1P_3$ is negative. This completes \textcolor{black}{\eqref{eq.add7.rotation}}.

		It follows that $\mathcal Ru<0$ in the upper half, so $\nabla u$ cannot
		vanish there. On the open symmetry segment, $\varphi_y=0$ and the Hopf
		lemma gives $\nu_y>0$, whence $u_y=t\nu_y>0$. Thus $u$ has no interior
		critical point anywhere in the kite.  Since an interior global extremum of
		a smooth nonconstant eigenfunction is a critical point, the hot-spots
		assertion follows.
	\end{proof}

	Together with Proposition~\ref{prop:pure-modes-interior}, this gives
	the hot-spots property for concave kites with $1<a<2$ and $h>0$.

	\begin{proof}[\textcolor{black}{Completion of the proof of Theorem~\ref{thm1.3}}]
		\begingroup\color{black}
		Proposition~\ref{prop:pure-modes-interior} and
		Subsections~\ref{subsec:trapezoid-multiplicity-two}
		and~\ref{subsec:kite-multiplicity-two} exclude interior critical
		points in all cases except the equilateral triangle
		$a=1$, $h=1/\sqrt3$, for which the same conclusion follows
		from \cite{Judge2020Ann,Judge2022Ann,Chen2026Invent}.
		Since an interior global extremum must be a critical point,
		the boundary-extrema assertion follows.
		\endgroup
	\end{proof}

\noindent {\bf Declaration.} 
After the present work was completed, OpenAI announced a proof of the hot spots conjecture for bounded simply connected planar domains with \(C^\infty\) boundaries in the preprint Strict hot spots and absence of interior critical points on smooth simply connected planar domains (October 7, 2026, https://github.com/openai/math). Their setting, however, excludes polygonal domains with corners, including the isosceles trapezoids and kites studied here, so their theorem does not cover our results. Moreover, we characterize how the symmetry of second Neumann eigenfunctions changes with the height of the domain and establish the uniqueness of the transition height. We also obtain monotonicity properties of the eigenfunctions and characterize the distribution of nonvertex boundary critical points when the second Neumann eigenvalue is simple. These conclusions are not contained in the result of OpenAI on the absence of interior critical points.

In addition, the two papers also employ different methods to exclude interior critical points. When the second Neumann eigenvalue is simple, our proof relies on the symmetry and monotonicity of the eigenfunctions. When it has multiplicity two, we use Jacobian determinants and directional derivative estimates to rule out interior critical points. OpenAI instead employs a variational principle for vector fields and multipliers constructed from Green kernels, deriving a contradiction from the eigenspace dimension bound or boundary unique continuation. Thus, the two papers differ in both the scope of their results and their methods of proof.

\noindent {\bf Declaration of AI use.} 
During the preparation of this manuscript, the authors used ChatGPT (OpenAI) as an auxiliary tool for language editing and preliminary examination of mathematical derivations. All mathematical statements, proofs, references, and conclusions were independently reviewed and verified by the authors, who take full responsibility for the accuracy and integrity of the manuscript.

	\begin{flushleft}
		Haiyun Deng\\
		Department of Applied Mathematics, Nanjing Audit University,\\ Nanjing, 211815, China\\
		\textit{e-mail:} \verb"hydeng@nau.edu.cn"
	\end{flushleft}
	
	\begin{flushleft}
		Changfeng Gui\\
		Department of Mathematics, University of Macau,\\ Taipa, Macau\\
		\textit{e-mail:} \verb"changfenggui@um.edu.mo"
	\end{flushleft}

	\begin{flushleft}
		Xuyong Jiang\\
		Department of Mathematics, Changzhou University,\\ Changzhou, 213164, China\\
		\textit{e-mail:} \verb"jiangxy@cczu.edu.cn"
	\end{flushleft}
	
	\begin{flushleft}
		Xiaoping Yang\\
		School of Mathematics, Nanjing University,\\ Nanjing, 210093, China\\
		\textit{e-mail:} \verb"xpyang@nju.edu.cn"
	\end{flushleft}
	
	\begin{flushleft}
		Ruofei Yao\\
		School of Mathematics, South China University of Technology,\\ Guangzhou, 510641, China\\
		\textit{e-mail:} \verb"ruofeiyaopde@gmail.com"
	\end{flushleft}
	
	\begin{flushleft}
		Jun Zou\\
		School of Mathematics, Nanjing University,\\ Nanjing, 210093, China\\
		\textit{e-mail:} \verb"jun_zou@smail.nju.edu.cn"
	\end{flushleft}

\end{document}